\documentclass[11pt,reqno]{amsart}
\usepackage{amsmath,amsbsy,amssymb,amscd,amsfonts,latexsym,amstext,delarray, amsmath,color,caption,comment}
\usepackage{tikz,lscape}
\usetikzlibrary{matrix,arrows}
\usepackage{graphicx}
\usepackage{multicol}
\usepackage{scrextend}
\usepackage{mathtools}
\usepackage{amssymb} 
\usepackage{calligra}
\usepackage{calrsfs}
\usepackage{eucal}
\usepackage{calrsfs}
\usepackage{epsfig}
\usepackage{xcolor,import}
\usepackage{transparent}
\usepackage{tikz-cd}
\usepackage{epstopdf}
\usepackage{subcaption}
\usepackage{enumitem}
\usepackage{float}
\numberwithin{equation}{section}
\numberwithin{figure}{section}
\usepackage{hyperref}
\usepackage{amsthm}
\usepackage{thmtools}
\usepackage{cleveref}
\usepackage{mathrsfs} 
\DeclareFontFamily{U}{mathb}{\hyphenchar\font45}
\DeclareFontShape{U}{mathb}{m}{n}{
      <5> <6> <7> <8> <9> <10> gen * mathb
      <10.95> mathb10 <12> <14.4> <17.28> <20.74> <24.88> mathb12
      }{}
\DeclareSymbolFont{mathb}{U}{mathb}{m}{n}

\DeclareMathSymbol{\precneq}{3}{mathb}{"AC}

\usepackage{mathtools}

\usepackage{graphicx}
\usepackage{tikz}
\tikzset{
  font={\fontsize{10pt}{12}\selectfont}}
  \usetikzlibrary{arrows.meta,arrows}
  \tikzset{>=latex}
\usepackage{caption, cleveref} 
\usepackage{xfrac}  
\usetikzlibrary{arrows, fit, backgrounds, patterns, positioning, calc, intersections, decorations.markings, decorations.pathmorphing, decorations.pathreplacing, shapes.misc}
\usetikzlibrary{shapes.multipart}
\usepackage[utf8]{inputenc}
\usepackage{multicol}
\usepackage{subcaption}
\usepackage{enumitem}

\usepackage[normalem]{ulem} 

\DeclareMathAlphabet{\pazocal}{OMS}{zplm}{m}{n}

\input xy

\xyoption{all}
    \advance \topmargin by -\headheight
    \advance \topmargin by -\headsep

    \evensidemargin \oddsidemargin
\makeatletter
\newcommand{\leqnomode}{\tagsleft@true\let\veqno\@@leqno}
\newcommand{\reqnomode}{\tagsleft@false\let\veqno\@@eqno}
\makeatother

\newtheorem {theorem}    {Theorem}[section]
\newtheorem {problem}    {Problem}

\theoremstyle{definition}
\newtheorem {lemma}      [theorem]    {Lemma}
\newtheorem {corollary}  [theorem]    {Corollary}
\newtheorem {proposition}[theorem]    {Proposition}

\theoremstyle{definition}
\newtheorem{definition}[theorem]{Definition}
\theoremstyle{definition}
\newtheorem{remark}[theorem]{Remark}

\newcommand{\defeq}{\vcentcolon=}
\newcommand{\gen}[1]{\langle #1 \rangle}

\def\TM{{\rm TM}}

\def\a{\alpha}                
\def\b{\beta}

\def\eps{\varepsilon}

\def\a{\alpha}
\def\b{\beta}

\def\Q{\mathbb{Q}}     
\def\N{\mathbb{N}}     
\def\lab{{\text{Lab}}}

\DeclareMathOperator{\CL}{CL}

\newcommand{\tm}{\mathrm{tm}}

\newcommand{\REACH}{\mathsf{REACH}}

\renewcommand{\int}{\mathrm{int}}

\def\b{\beta}

\def\vertexradius{.1}
\def\vertex(#1){\fill (#1) circle (\vertexradius)}

\begin{document}

\title{\bf Conjugator Length in Finitely Presented groups}
\maketitle
\begin{center}

Conan Gillis, Francis Wagner

\end{center}

\bigskip

\begin{center}

\textbf{Abstract}

\end{center}

\begin{addmargin}[2em]{2em}

The conjugator length function of a finitely generated group is the function $f$ so that $f(n)$ is the minimal upper bound on the length of a word realizing the conjugacy of two words of length at most $n$.  We study herein the spectrum of functions which can be realized as the conjugator length function of a finitely presented group, showing that it contains every function that can be realized as the Dehn function of a finitely presented group.  In particular, given a real number $\alpha\geq2$ which is computable in double-exponential time, we show there exists a finitely presented group whose conjugator length function is asymptotically equivalent to $n^\alpha$.  This yields a substantial refinement to results of Bridson and Riley.  We attain this result through the computational model of $S$-machines, achieving the more general result that any sufficiently large function which can be realized as the time function of an $S$-machine can also be realized as the conjugator length function of a finitely presented group.  Finally, we use the constructed groups to explore the relationship between the conjugator length function, the Dehn function, and the annular Dehn function in finitely presented groups.

\end{addmargin}

\bigskip

\bigskip


\section{Introduction}
Much of the development of combinatorial and geometric group theory has been driven by three famous questions due to Max Dehn which can be posed for any finitely generated group $G=\langle X\rangle$ and finitely presented group $H=\gen{Y\mid\pazocal{R}}$:
\begin{itemize}
\item \textbf{Word Problem:} Given a word $w\in (X\cup X^{-1})^\ast$, does $w$ represent the identity element of $G$?
\item \textbf{Conjugacy Problem:} Given two words $u,v\in (X\cup X^{-1})^\ast$, do $u$ and $v$ represent conjugate elements of $G$?
\item \textbf{Isomorphism Problem:} Given a finitely presentated group $H'=\langle Y'\mid \pazocal{R}'\rangle$, is $H'$ isomorphic to $H$?
\end{itemize}

These questions are easily translated into computational decision problems, whereby they provide fruitful ground for interaction between, on the one hand, computability and complexity theory, and on the other, group theory. For the first two of the above questions, however, there has also been a persistent focus on their ``intrinsic" difficulty as reflected in the geometry of $G$, rather than in an abstract model of computation. 

With respect to the Word Problem of a finitely presented group, this is measured via the \textit{Dehn function}, the smallest isoperimetric function of a Cayley complex for the group.  Prominent work of Brady-Bridson-Forester-Shankar \cite{BradyBridson, BBFS}, Sapir-Birget-Rips-Ol'shanskii \cite{BORS,SBR}, and Ol'shanskii \cite{O18} details the set of functions which \textit{can} be realized as the Dehn function of a finitely presented group (see \cite{Brady2007} for a survey of these and related results).  

Recently, research has focused on completing the analogous picture for conjugator length functions, a measure of the complexity the Conjugacy Problem for the group.

\begin{definition} \label{def-CL} Let $G$ be a finitely generated group with finite generating set $X$. Define the function $c_{G,X}:\left((X\cup X^{-1})^\ast\right)\times \left((X\cup X^{-1})^\ast\right)\to \mathbb{N}\cup\{-\infty\}$ by taking $c_{G,X}(u,v)$ to be: 

\begin{itemize}

\item the length of the shortest $\gamma\in (X\cup X^{-1})^\ast$ such that $\gamma u\gamma^{-1}=v$ in $G$; or 

\item $-\infty$ if no such $\gamma$ exists 

\end{itemize}

Then, noting that $c_{G,X}(u,u)=0$ for all $u\in(X\cup X^{-1})^\ast$, define the \textit{conjugator length function} $\CL_{G,X}:\mathbb{N}\to \mathbb{N}$ of $G$ with respect to $X$ to be the function given by $$\CL_{G,X}(n)=\max_{u,v}\{c_{G,X}(u,v)\mid |u|+|v|\leq n\}.$$
\end{definition}

To free the above definition from the particular finite generating set considered, we adopt the following standard equivalence (and preorder) for functions on the natural numbers:

\begin{definition}
Given $f,g:\mathbb{N}\to\mathbb{N}$, it is taken that $f\preceq  g$ if there exists a constant $C\in\N$ such that $f(n)\leq Cg(Cn+C)+Cn+C$ for all $n$.  Naturally, we then say $f\sim g$ if both $f(n)\preceq g(n)$ and $g(n)\preceq f(n)$.\end{definition}

It is then easy to see that $\CL_{G,X}\sim\CL_{G,Y}$ for any pair of finite generating sets $X$ and $Y$ of $G$.  As such, the \textit{conjugator length function of $G$}, $\CL_G$, to be the equivalence class of $\CL_{G,X}$ for any finite generating set $X$.

Bridson and Riley show that, for every positive integer $m$, there exists a finitely presented group $G_m$ with $\CL_{G_m}(n)\sim n^m$ \cite{BridsonRileyFil,BridsonRileyNilp}. More recently, the same pair of authors show that the set of $\a$ for which $n^\alpha$ is equivalent to the conjugator length function of a finitely presented group forms a dense subset of $\mathbb{Q}\cap [2,\infty)$ \cite{BridsonRileySnowflake}, and that for any $k$ there exists a finitely presented group whose conjugator length function grows like functions in the $k$-th level of the Grzegorczyk hierarchy of primitive recursive functions \cite{BridsonRileyFastGrowing}. For a survey of results on conjugator length functions, we refer the reader to \cite{BRS}.

We refine and unify these results, showing that virtually any `reasonably computable' function which is at least quadratic can be realized as the conjugator length function of a finitely presented group.  In particular, the set of $\a$ for which $n^\a$ is equivalent to the conjugator length function of a finitely presented group \textit{contains} $\Q\cap[2,\infty)$, and in fact contains every number in $[2,\infty)$ which is computable in double-exponential time.

We achieve this by means of a technical result concerning $S$-machines.

\begin{theorem} \label{main-theorem}

For any recognizing $S$-machine $\textbf{S}$, there exists a finitely presented group $G_\textbf{S}$ such that $\TM_\textbf{S}(n)\preceq \CL_{G_\textbf{S}}(n)\preceq n^2+\TM_\textbf{S}(n)$, where $\TM_\textbf{S}$ is the time function of $\textbf{S}$.

\end{theorem}


Applying the emulation theorems of \cite{SBR} and \cite{CW}, \Cref{main-theorem} quickly recovers the main theorem of \cite{BridsonRileyFastGrowing}.  Indeed, by constructing a suitable $S$-machine, \Cref{main-theorem} implies the spectrum of conjugator length functions of finitely presented groups contains that of Dehn functions:

\begin{theorem} \label{Dehn-to-CL}

For any finitely presented group $G$, there exists a finitely presented group $H$ such that $\CL_H\sim\text{Dehn}_G$.

\end{theorem}

Combining this statement with the various results relating to the spectrum of Dehn functions then immediately implies the following statement:

\begin{corollary} \label{main-corollary}

For any $\a\geq2$ which is computable in time $O(2^{2^m})$, there exists a finitely presented group $G_\a$ such that $\CL_{G_\a}(n)\sim n^\alpha$.

\end{corollary}

For example, for any real algebraic $\a\geq2$ there exists a finitely presented groups with conjugator length function equivalent to $n^\a$.  Moreover, there exist finitely presented groups with conjugator length functions equivalent to $n^\pi$, $n^e$, etc.  More generally, Ol'shanskii's constructions in \cite{O18} imply an even wider array of functions that may be realized:

\begin{corollary} \label{log-corollary}

For any real numbers $\a\geq 2$ and $\b_1,\ldots,\b_\ell$, all of which are computable in time $O(2^{2^m})$ and such that $\a+|\b_1|+|\b_2|+\cdots+|\b_i|=2$ implies $\b_{i+1}\geq0$ for $i<\ell$, there exists a finitely presented group $G_{\a,\b_1,\ldots,\b_\ell}$ whose conjugator length function is equivalent to $n^\a(\log n)^{\b_1}(\log\log n)^{\b_2}\cdots(\log^{(\ell)}n)^{\beta_\ell}$, where $\log^{(\ell)}n$ is the $\ell$-times iterated logarithm function. 

\end{corollary}

Furthermore, Theorem \ref{Dehn-to-CL} gives a strengthening of \cite[Theorem 2]{BridsonRileyFastGrowing}, as well as the first example of a conjugator length function equivalent to a recursive, but not primitive-recursive, function:

\begin{corollary}
    Let $A_k(n)$ be the $k$-th Ackermann function. Then there exists a group $G_k$ whose conjugator length function is equivalent to $A_k(n)$. Moreover, there exists a group whose conjugator length function is equivalent to the non-primitive recursive Ackermann function $A(n)=A_n(n)$.
\end{corollary}

There is currently no known finitely generated group whose conjugator length function is subquadratic and superlinear, and our results do not produce one. There is some potential for our approach to shed light on this question, however for reasons we will discuss below, the methods of this paper are insufficient. In particular, we have been unable to remove the $n^2$ term from the upper bound in the statement of Theorem \ref{main-theorem}, and it seems likely that an entirely new approach is required in order to do so. 
As such, the following question remains open:

\begin{problem} \label{gap-problem}

Does there exist a finitely generated group $G$ such that $n\prec \CL_G(n)\prec n^2$?

\end{problem}

Vandeputte \cite{vandeputte2026residualfinitenesspropertieshalls} has answered \Cref{gap-problem} in the affirmative, however no explicit description of a conjugator length function in this range is yet known. Moreover, the following refinement remains open:

\vspace{5pt}

\textbf{Problem 1'.}  \textit{Does there exist a finitely} presented \textit{group $G$ such that $n\prec \CL_G(n)\prec n^2$?}

\vspace{5pt}

This question resembles an important one related to Dehn functions, which is resolved by the `isoperimetric gap' that no finitely presented group has Dehn function strictly between linear and quadratic \cite{Gromov} \cite{Bowditch} \cite{O91}. The difficulty in finding a group resolving \Cref{gap-problem}' may thus seem to suggest some deeper connection between the two invariants at hand. However, a group $G$ with $\text{Dehn}_G(n)$ computable and $\CL_G$ non-computable is given in, e.g., \cite{Miller71} \cite{OS19}. We expand substantially on this fact, establishing that there is no computable upper bound for a finitely presented group's conjugator length function in terms of its Dehn function, even when both are computable:


\begin{theorem}\label{noCLBound}
For any recognizing $S$-machine $\textbf{S}$ such that $\TM_\textbf{S}(n)\succeq n^2$, there exists a finitely presented group $H_\textbf{S}$ such that $\TM_\textbf{S}\sim \CL_{H_\textbf{S}}$ and $\text{Dehn}_{H_\textbf{S}}(n)\sim n^3$.
\end{theorem}

Moreover, this result makes progress on the following question of Bridson-Riley-Sale \cite{BRS}:

\begin{problem} \label{CL-Dehn-pairs}

For what pairs $(f,g)$ of functions on the natural numbers do there exist finitely presented groups $G_{f,g}$ such that $\text{Dehn}_{G_{f,g}}\sim f$ and $\CL_{G_{f,g}}\sim g$?

\end{problem}

With this terminology, \Cref{noCLBound} can be interpreted as saying that the pair $(n^3,g(n))$ can be realized by a finitely presented group for just about any `reasonably computable' $g(n)\succeq n^2$.  


Finally, we use \Cref{noCLBound} to study the \textit{annular Dehn function} introduced by Brick and Corson in \cite{BrickCorson}.  As the name suggests, interpreting in the natural way the Dehn function of a finitely presented group $G$ through van Kampen diagrams over one of its finite presentations, the annular Dehn function $\text{Ann}_G:\N\to\N$ is the analogous function for annular (Schupp) diagrams.  As such, annular Dehn functions are proposed as a means of studying the group's Conjugacy Problem.

\begin{theorem} \label{annular-Dehn}

For any $S$-machine $\textbf{S}$ with $\TM_\textbf{S}(n)\succeq n^2$, the finitely presented group $G_\textbf{S}$ given in \Cref{main-theorem} satisfies $\TM_\textbf{S}^2\preceq\text{Ann}_{G_\textbf{S}}\preceq\TM_\textbf{S}^3$.

\end{theorem}

While this statement does not give the equivalence class of the annular Dehn function if $\TM_\textbf{S}$ is bounded above by a polynomial, the statement exhibits a broad new scope of functions that can be realized as the annular Dehn function of a finitely presented group. Indeed, Theorem \ref{annular-Dehn} has the following immediate Corollary:

\begin{corollary}
  For any function $f$ such that $f^3\sim f$, there exists a finitely presented group $G_f$ whose annular Dehn function is equivalent to $f$.
\end{corollary}

We point out that, for any $\a>0$, $2^{n^\a}$ satisfies the above relation -- this comports nicely with a recent result of Bridson and Riley \cite{BridsonRileySnowflake}, who showed that for all $\a$ in a dense subset of $[2,\infty)$, $n^\a$ is the annular Dehn function of some finitely presented group. We leave it to further work to elucidate the full set of possible annular Dehn functions. Indeed, it seems likely that an analogue of \Cref{noCLBound} in which the Dehn function is quadratic would provide a major step toward doing so, since the upper bound given in \Cref{annular-Dehn} is given by the composition of the Dehn function and the conjugator length function, and so a quadratic Dehn function would give an upper bound that matches the lower bound in \Cref{annular-Dehn}.

\medskip


\subsection{Discussion of contents} \

Our methods run through \textit{$S$-machines}, the tool introduced by Sapir in the 1990s for addressing questions related to the Dehn function \cite{SBR}.  See \Cref{sec-S-machines} for precise definitions.

There are many equivalent interpretations of $S$-machines (see \cite{S06}, \cite{CW}), but the two vital to this article are as follows: (1) An $S$-machine can be viewed as a rewriting system resembling a multi-tape, non-deterministic Turing machine that works with group words, and (2) An $S$-machine can be viewed as a multiple HNN-extension of a free group.  With these interpretations, it was shown that $S$-machines form a `robust' computational model in that any non-deterministic Turing machine can be polynomially emulated by an $S$-machine \cite{SBR}, while these machines also provide a relational framework which is convenient for study through meticulous and well-established group-theoretic techniques.  Accordingly, $S$-machines facilitate the construction of finitely presented groups satisfying various desirable properties.

The key to addressing many of the group-theoretic arguments related to the groups associated to an $S$-machine is the study of van Kampen diagrams over their presentations.  Several intricate techniques have been built up for studying these diagrams (see, for example, \cite{O18}, \cite{W}, and \cite{WEmb}), and there has been a few applications to the study of Schupp diagrams over the relevant presentations \cite{OSconj}, \cite{OS19}, \cite{WMal}.

After introducing the basic definitions and some elementary constructions involving $S$-machines in \Cref{sec-S-machines}, we present our main construction in \Cref{sec-enhanced}: Starting with an arbitrary recognizing $S$-machine $\textbf{S}$, we alter its computational makeup in particular specified ways to build the associated `enhanced' cyclic $S$-machine $\textbf{E}_\textbf{S}$.  The construction of this enhanced machine and verifying its computational makeup accounts for much of our work.  Its most important properties given by Lemmas \ref{E0 language} and \ref{universal complexity}: 

\begin{enumerate}

\item The enhanced machine accepts the same language as $\textbf{S}$

\item Its time function is equivalent to $\TM_\textbf{S}$ 

\item For any pair $W_1,W_2$ of admissible words of $\textbf{E}_\textbf{S}$ with `circular base', if there exists a computation between $W_1$ and $W_2$, then there exists one whose length is bounded by a function of the $a$-length of admissible subwords of $W_1,W_2$ with small base.

\end{enumerate}

In \Cref{sec-groups}, we translate the machinery into the setting of finitely presented groups.  After defining the relevant group in \Cref{sec-associated-groups}, we investigate the basic properties of diagrams over its presentation in \Cref{sec-bands-annuli} and how the group mimics the machinery in \Cref{sec-trapezia}.  In \Cref{subsec-modified-length}, we introduce a modified length function on the words over the generators which is convenient for the arguments that follow while also being $\Theta$-equivalent to the typical word length norm.

Note that the groups associated to $S$-machines considered herein are `hubless', not containing a relation killing the accept configuration.  The hub relation has been essential for constructing groups with exotic Dehn functions.  As such, the investigation of the conjugator length function of associated groups with hubs seems a promising path toward answering \Cref{CL-Dehn-pairs}.

We then use Schupp diagrams over the relevant presentation to prove \Cref{main-theorem}: (1) and (2) effectively give the lower bound, while (3) implies the upper bound.  Note that \Cref{rmk-quadratic} accounts for the difficulty in removing the quadratic term in the upper bound described in the statement of \Cref{main-theorem}.

Finally, we use \Cref{main-theorem} to complete proofs of \Cref{Dehn-to-CL}, \Cref{noCLBound}, and \Cref{annular-Dehn}.  These results mostly come quickly, following from a computational argument mirroring one in \cite{SBR}, simple combinatorial arguments on diagrams, elementary combination properties for the conjugator length function, and properties of the annular Dehn function proved in \cite{BrickCorson}.

\bigskip

\textbf{Acknowledgements:} The authors would like to thank Tim Riley for his support and several useful discussions related to these results.


\section{\texorpdfstring{$S$}--machines} \label{sec-S-machines}

In this section, we give a formal definition of the computational model of $S$-machines.  This definition is laid out in much the same way as it is in \cite{WEmb}.

\subsection{Rewriting systems} \

There are many equivalent ways of defining $S$-machines.  Following the conventions of \cite{BORS}, \cite{O18}, \cite{OS01}, \cite{OSconj}, \cite{OS06}, \cite{OS19}, \cite{SBR}, \cite{WEmb}, \cite{W}, and others, the model is presented here as a rewriting system for group words.

A \textit{hardware} is a pair of finite sets $(Y,Q)$ with a fixed partition $Q=\sqcup_{i=0}^NQ_i$ and $Y=\sqcup_{i=1}^NY_i$.  The subsets $Q_i$ and $Y_i$ are called the \textit{parts} of $Q$ and $Y$, respectively.  Moreover, any letter of $Q_i\sqcup Q_i^{-1}$ is called a \textit{state letter} or \textit{$q$-letter}, while any letter of $Y_i\sqcup Y_i^{-1}$ is called a \textit{tape letter} or \textit{$a$-letter}.

An \textit{admissible word} $W$ over $(Y,Q)$ is a reduced word over $(Y\cup Q)^{\pm1}$ of the form $q_0^{\eps_0}w_1q_1^{\eps_1}\dots w_sq_s^{\eps_s}$ such that for each $i\in\{1,\dots,s\}$, there exists an integer $j(i)\in\{1,\dots,N\}$ such that the subword $q_{i-1}^{\eps_{i-1}}w_iq_i^{\eps_i}$ is either:

\begin{enumerate}

\item an element of $(Q_{j(i)-1}F(Y_{j(i)})Q_{j(i)})^{\pm1}$,

\item of the form $qwq^{-1}$ where $q\in Q_{j(i)-1}$ and $w\in F(Y_{j(i)})$, or

\item is of the form $q^{-1}wq$ where $q\in Q_{j(i)}$ and $w\in F(Y_{j(i)})$.

\end{enumerate}

In this case, the \textit{base} of $W$ is the (possibly unreduced) word $Q_{j(0)}^{\eps_0}Q_{j(1)}^{\eps_1}\dots Q_{j(s)}^{\eps_s}$, while the subword $q_{i-1}^{\eps_{i-1}}w_iq_i^{\eps_i}$ is called the \textit{$Q_{j(i)-1}^{\eps_{i-1}}Q_{j(i)}^{\eps_i}$-sector} of $W$.  Note that an admissible word may have many different subwords which are sectors of the same name.  

The \textit{standard base} of $(Y,Q)$ is the base $Q_0Q_1\dots Q_N$, while a \textit{configuration} is an admissible word whose base is standard.  

%
%
%
%
%


The \textit{$a$-length} of the admissible word $W$ is the number $|W|_a$ of $a$-letters comprising $W$.  The \textit{$q$-length} $|W|_q$ is defined similarly.

An \textit{$S$-rule} on the hardware $(Y,Q)$ is a rewriting rule $\theta$ with the following associated information:

\begin{itemize}

\item A subset $Q(\theta)=\sqcup_{i=0}^N Q_i(\theta)$ of $Q$ such that each $Q_i(\theta)$ is a singleton $\{q_i\}$ consisting of a letter from $Q_i$.

\item A subset $Q'(\theta)=\sqcup_{i=0}^N Q_i'(\theta)$ of $Q$ such that each $Q_i'(\theta)$ is a singleton $\{q_i'\}$ consisting of a letter from $Q_i$.

\item A subset $Y(\theta)=\sqcup_{i=1}^N Y_i(\theta)$ of $Y$ such that each $Y_i(\theta)$ is a subset of $Y_i$.

\item A set of words $\{\a_{i,\theta},\omega_{i,\theta}\in F(Y_i(\theta)):i=1,\dots,N\}$

\end{itemize}

An admissible word $W$ is \textit{$\theta$-admissible} if every letter comprising it is a letter of $Q(\theta)^{\pm1}$ or $Y(\theta)^{\pm1}$.  In this case, the result of applying $\theta$ to $W$ is the admissible word $W\cdot\theta$ obtained by simultaneously:

\begin{itemize}

\item Replacing any occurrence of the $q_i^{\pm1}$ with the subword $(\omega_{i,\theta}q_i'\a_{i+1,\theta})^{\pm1}$, where for completeness $\omega_{0,\theta}$ and $\a_{N,\theta}$ are empty.

\item Making any necessary free reductions so that the word is again reduced.

\item Removing any prefix or suffix consisting of tape letters so that the word is again admissible.

\end{itemize}

It is common to represent $\theta$ in this case with the notation:
$$\theta=[q_0\to q_0'\a_{1,\theta}, \ q_1\to \omega_{1,\theta}q_1'\a_{2,\theta}, \ \dots, \ q_{N-1}\to \omega_{N-1,\theta}q_{N-1}'\a_{N,\theta}, \ q_N\to \omega_{N,\theta}q_N']$$
If $Y_i(\theta)=\emptyset$, then $\theta$ is said to \textit{lock the $Q_{i-1}Q_i$-sector}.  In this case, $\a_{i,\theta}$ and $\omega_{i,\theta}$ are both necessarily empty, while the notation $q_{i-1}\xrightarrow{\ell}\omega_{i-1,\theta}q_i'$ is used in the representation of the rule.

Note that this notation does not quite capture the full makeup of $\theta$, as it omits the \textit{domain} $Y(\theta)$.  In most cases, however, each set $Y_i(\theta)$ is either taken to be $Y_i$ or $\emptyset$, so that the locking notation suffices.  As such, this is taken as an implicit assumption unless explicitly stated otherwise.

The \textit{inverse $S$-rule} of $\theta$ is then the $S$-rule $\theta^{-1}$ on $(Y,Q)$ given by taking:

\begin{itemize}

\item $Q(\theta^{-1})=Q'(\theta)$ and $Q'(\theta^{-1})=Q(\theta)$,

\item $Y(\theta^{-1})=Y(\theta)$, and

\item $\omega_{\theta^{-1},i}=\omega_{\theta,i}^{-1}$ and $\a_{\theta^{-1},i}=\a_{\theta,i}^{-1}$.

\end{itemize}

This terminology is justified by noting that an admissible word $W$ is $\theta$-admissible if and only if $W\cdot\theta$ is $\theta^{-1}$-admissible, with $(W\cdot\theta)\cdot\theta^{-1}\equiv W$.

An \textit{$S$-machine} is then a rewriting system $\textbf{S}$ with a fixed hardware $(Y,Q)$ and a \textit{software} consisting of a finite symmetric set of $S$-rules $\Theta(\textbf{S})$ over $(Y,Q)$.  It is convenient to partition the hardware $\Theta(\textbf{S})=\Theta^+(\textbf{S})\sqcup\Theta^-(\textbf{S})$ such that $\theta\in\Theta^+(\textbf{S})$ if and only if $\theta^{-1}\in\Theta^-(\textbf{S})$; in this case, $\theta$ is called a \textit{positive rule} of $\textbf{S}$, while its inverse is called a \textit{negative rule}.

A \textit{computation} of $\textbf{S}$ is a finite sequence $\pazocal{C}:W_0\to\dots\to W_t$ of admissible words with fixed rules $\theta_1,\dots,\theta_t\in \Theta(\textbf{S})$ such that $W_{i-1}\cdot\theta_i\equiv W_i$.  In this case, $t$ is called the \textit{length} of $\pazocal{C}$ and the word $\theta_1\cdots\theta_t$ the \textit{history} of $\pazocal{C}$.  The computation $\pazocal{C}$ is said to be \textit{reduced} if its history is a reduced word in $F(\Theta^+(\textbf{S}))$; note that any computation can be made reduced without altering its initial or terminal admissible words simply by deleting any pairs of consecutive mutually inverse rules.

Given an $S$-machine $\textbf{S}$ with hardware $(Y,Q)$, it is common to fix in every part $Q_i$ of $Q$ a \textit{start} and an \textit{end} state letter.  A configuration of $\textbf{S}$ is called a \textit{start} or \textit{end} configuration if all the state letters that comprise it are start or end state letters, respectively.  

The $S$-machine $\textbf{S}$ is said to be \textit{recognizing} if there are fixed sectors $Q_{i-1}Q_i$ which are deemed to be \textit{input sectors}.  A start configuration $W$ of the recognizing $S$-machine $\textbf{S}$ is called an \textit{input configuration} if all of its non-input sectors have empty tape word.   In this case, the word obtained from $W$ by deleting all of its state letters is called its \textit{input}.  In contrast, the \textit{accept configuration} is the end configuration for which every sector has empty tape word.

A configuration $W$ of the recognizing $S$-machine $\textbf{S}$ is said to be \textit{accepted} if there exists an \textit{accepting computation} which begins with the configuration $W$ and terminates with the accept configuration.  If $W$ is an input configuration, then its input is also said to be accepted.

For a configuration $W$ accepted by $\textbf{S}$, define $\tm(W)$ to be the minimal length of a computation accepting $W$.  The \textit{time function} of $\textbf{S}$ is then the non-decreasing function $\TM_{\textbf{S}}:\N\to\N$ given by: $$\TM_{\textbf{S}}(n)=\max\{\tm(W): W\text{ is an accepted input configuration of } \textbf{S}, \ |W|_a\leq n\}$$
The \textit{generalized time functions} of $\textbf{S}$ is defined in much the same way, but accounts for any accepted configuration rather than just input configurations.

If two recognizing $S$-machines have the same language of accepted words and $\Theta$-equivalent time functions, then they are said to be \textit{equivalent}.

The following statement simplifies how one approaches the rules of a recognizing $S$-machine:

\begin{lemma} [Lemma 2.1 of \cite{O18}] \label{simplify rules}
 
Every recognizing $S$-machine $\textbf{S}$ is equivalent to a recognizing $S$-machine in which $|\a_{\theta,i}|_a,|\omega_{\theta,i}|_a\leq1$ for every rule $\theta$.

\end{lemma}

As a result of Lemma \ref{simplify rules} and the assumption that every $a$-letter of $\a_{\theta,i}$ and $\omega_{\theta,i}$ is in the domain of $\theta$, it may be assumed that each part of every rule of an $S$-machine is of the form $q_i\to bq_i'a$ where $\|a\|,\|b\|\leq1$ (note that the corresponding part of $\theta^{-1}$ is then interpreted as $q_i'\to b^{-1}q_ia^{-1}$).

In fact, the analogous statement of \cite{O18} implies that one can enforce the stronger condition that $\|a\|+\|b\|\leq1$.  However, it will be convenient to allow $\|a\|=\|b\|=1$ in the definitions of the rules of the $S$-machines constructed in future sections, so we use the weaker statement of \Cref{simplify rules} for consistency.


\medskip


\subsection{Elementary properties} \

This section serves to recall several features of $S$-machines which follow from the definition adopted above and will serve useful in the analysis of the machines constructed in the ensuing sections.


\begin{lemma} \label{locked sectors}

If the rule $\theta$ locks the $Q_iQ_{i+1}$-sector, i.e it has a part $q_i\xrightarrow{\ell}aq_i'$ for some $q_i,q_i'\in Q_i$, then the base of any $\theta$-admissible word has no subword of the form $Q_iQ_i^{-1}$ or $Q_{i+1}^{-1}Q_{i+1}$.

\end{lemma}

Through the rest of this manuscript, we will often use copies of words over disjoint alphabets. To be precise, let $A$ and $B$ be disjoint alphabets and $\varphi:A\to B$ be an injection.  Then for any word $a_1^{\eps_1}\dots a_k^{\eps_k}$ with $a_i\in A$ and $\eps_i\in\{\pm1\}$, its \textit{copy} over the alphabet $B$ formed by $\varphi$ is the word $\varphi(a_1)^{\eps_1}\dots\varphi(a_k)^{\eps_k}$. Typically, the injection defining the copy will be contextually clear.

Alternatively, a copy of an alphabet $A$ is a disjoint alphabet $A'$ which is in one-to-one correspondence with $A$. For a word over $A$, its copy over $A'$ is defined by the bijection defining the correspondence between the alphabets.


\begin{lemma} [Lemma 2.7 of \cite{O18}] \label{multiply one letter}

Let $X_i$ be a subset of $Y_i\cup Y_i^{-1}$ with $X_i\cap X_i^{-1}=\emptyset$.  Let $\Theta_i^+$ be a set of positive rules in correspondence with $X_i$ such that each rule multiplies the $Q_{i-1}Q_i$-sector by the corresponding letter on the left (respectively on the right).  Let $\pazocal{C}:W_0\to\dots\to W_t$ be a reduced computation with base $Q_{i-1}Q_i$ and history $H\in F(\Theta_i^+)$.  Denote the tape word of $W_j$ as $u_j$ for each $0\leq j\leq t$.  Then:


\begin{enumerate} [label=({\alph*})]

\item $H$ is the natural copy of the reduced form of the word $u_tu_0^{-1}$ read from right to left (respectively the word $u_0^{-1}u_t$ read left to right). In particular, if $u_0\equiv u_t$, then the computation is empty

\item $\|H\|\leq\|u_0\|+\|u_t\|$

\item if $\|u_{j-1}\|<\|u_j\|$ for some $1\leq j\leq t-1$, then $\|u_j\|<\|u_{j+1}\|$

\item $\|u_j\|\leq\max(\|u_0\|,\|u_t\|)$

\end{enumerate}

\end{lemma}

Note that in the setting of \Cref{multiply one letter}, for any two words $w_1,w_2\in F(Y_i)$ comprised entirely of letters from $X_i\cup X_i^{-1}$, there exists a computation of the same form as that described in \Cref{multiply one letter} with initial and terminal tape words $w_1$ and $w_2$, respectively.

%
%
%
%
%
%
%

\begin{lemma} [Lemma 2.8 of \cite{O18}] \label{multiply two letters}

Let $X_\ell$ and $X_r$ be disjoint subsets of $Y_i$ which are copies of some set $X$.  Let $\Theta_i^+$ be a set of positive (or negative) rules in correspondence with $X$ such that each rule multiplies the $Q_{i-1}Q_i$-sector on the left by the corresponding letter's copy in $X_\ell$ and on the right by the copy in $X_r$.  Let $\pazocal{C}:W_0\to\dots\to W_t$ be a reduced computation with base $Q_{i-1}Q_i$ and history $H\in F(\Theta_i^+)$.  Denote the tape word of $W_j$ as $u_j$ for each $0\leq j\leq t$.  Then:


\begin{enumerate}[label=({\alph*})]

\item if $\|u_{j-1}\|<\|u_j\|$ for some $0\leq j\leq t-1$, then $\|u_j\|<\|u_{j+1}\|$

\item $\|u_j\|\leq\max(\|u_0\|,\|u_t\|)$ for each $j$

\item $\|H\|\leq\frac{1}{2}(\|u_0\|+\|u_t\|)$.

\end{enumerate}

\end{lemma}

%
%
%
%
%
%

\begin{lemma} [Lemma 3.6 of \cite{OS12}] \label{unreduced base}

Suppose $\pazocal{C}:W_0\to\dots\to W_t$ is a reduced computation of an $S$-machine with base $Q_iQ_i^{-1}$ (respectively $Q_i^{-1}Q_i$). For $0\leq j\leq t$, let $u_j$ be the tape word of $W_j$. Suppose each rule of $\pazocal{C}$ multiplies the $Q_iQ_{i+1}$-sector (respectively the $Q_{i-1}Q_i$-sector) by a letter from the left (respectively from the right), with different rules corresponding to different letters. Then there exists a factorization $H\equiv H_1H_2^\ell \bar{H}_2H_3$ such that:

\begin{enumerate}

\item $\ell\in\N$

\item $\bar{H_2}$ is a proper prefix of $H_2$.

\item $|W_i|_a=|W_{i-1}|_a-2$ for all $1\leq i\leq\|H_1\|$

\item $|W_i|_a=|W_{i-1}|_a=\|H_2\|$ for all $\|H_1\|+1\leq i\leq t-\|H_3\|$

\item $|W_i|_a=|W_{i-1}|_a+2$ for all $t-\|H_3\|+1\leq i\leq t$

\end{enumerate}


\end{lemma}

Note that the statement of \Cref{unreduced base} above takes a different form from its analogue in \cite{OS12}, representing the distinct purpose it is used for here.  However, it follows from an identical proof to that presented in \cite{OS12}.

Further, in the setting of \Cref{unreduced base} the subcomputation with history $H_2^\ell$ cyclically permutes the tape word $\ell$ complete times, while the subcomputation with history $\bar{H}_2$ cyclically permutes just a proper prefix or suffix of the word.  In particular, letting $h_i=\|H_i\|$, we have $W_{h_1}\equiv W_{h_1+kh_2}$ for all $0\leq k\leq \ell$.

The next statement, new to this manuscript, is a consequence of \Cref{unreduced base} and will be vital to the arguments of \Cref{sec-defective}.

\begin{lemma} \label{unreduced base quadratic}

Let $\textbf{S}$ be an $S$-machine satisfying \Cref{simplify rules} and suppose every rule of $\textbf{S}$ multiplies the $Q_iQ_{i+1}$ sector (respectively the $Q_{i-1}Q_i$-sector) by a letter from the left (respectively from the right), with different rules corresponding to different letters.  Let $\pazocal{C}:W_0\to\dots\to W_t$ be a nonempty reduced computation whose base $B$ contains a subword of the form $Q_iQ_i^{-1}$ (respectively $Q_i^{-1}Q_i)$.  Suppose $\pazocal{C}$ has minimal length among computations between $W_0$ and $W_t$.  Then there exists a two-letter subword $UV$ of $B$ such that:

\begin{enumerate}

\item There exists a rule in the history of $\pazocal{C}$ which multiplies the $UV$-sector by a letter on the left or on the right

\item Letting $\pazocal{C}':W_0'\to\dots\to W_t'$ the restriction of $\pazocal{C}$ to the base $UV$, we have $t\leq 12n^2+2n$ for $n=\max(|W_0'|_a,|W_t'|_a)$.

\end{enumerate}

\end{lemma}

\begin{proof}

Let $B^{(1)}$ be the two-letter subword of $B$ of the form $Q_iQ_i^{-1}$ (respectively $Q_i^{-1}Q_i$) given by the hypothesis.  Note that $B^{(1)}$ satisfies condition (1) of the statement.

Let $\pazocal{C}^{(1)}:W_0^{(1)}\to\dots\to W_t^{(1)}$ be the restriction of $\pazocal{C}$ to $B^{(1)}$.  By hypothesis, $\pazocal{C}^{(1)}$ satisfies the hypotheses of \Cref{unreduced base}, so that there exists a factorization $H\equiv H_1H_2^\ell \bar{H}_2 H_3$ of the history of $\pazocal{C}$ satisfying the conclusion of that statement.

Let $h_j=\|H_j\|$ and $\bar{h}_2=\|\bar{H}_2\|$.  Note that condition (2) of \Cref{unreduced base} yields $\bar{h}_2\leq h_2$.  Moreover, condition (3) yields $|W_0^{(1)}|_a=|W_{h_1}^{(1)}|_a+2h_1$, while condition (5) yields $|W_t^{(1)}|_a=|W_{t-h_3}^{(1)}|_a+2h_3$.  By construction, $W_{h_1}^{(1)}\equiv W_{h_1+\ell h_2}^{(1)}$, with $|W_{h_1}^{(1)}|_a=h_2$.

Suppose $\ell\leq1$.  Then $|W_0^{(1)}|_a+|W_t^{(1)}|_a=2(h_1+h_2+h_3)\geq 2h_1+h_2+\bar{h}_2+2h_3\geq t$, so that $t\leq2\max(|W_0^{(1)}|_a,|W_t^{(1)}|_a)$.  Hence, the statement is satisfied by setting $UV=B^{(1)}$.  


As a result, it suffices to assume that $\ell\geq2$.  In particular, $W_{h_1}$ and $W_{h_1+h_2}$ must be $\theta$-admissible for the first letter of $H_2$, and so have the same state letters.  Iterating, $W_{h_1}$ and $W_{h_1+\ell h_2}$ also have the same state letters.

Note that if $W_{h_1}\equiv W_{h_1+\ell h_2}$, then $H_1\bar{H}_2 H_3$ is the history of a reduced computation between $W_0$ and $W_t$.  So, the minimality hypothesis would imply $\ell=0$, contradicting our assumption above.  

Hence, it suffices to assume that there exists a two-letter subword $B^{(2)}$ such that for the restriction $\pazocal{C}^{(2)}:W_0^{(2)}\to\dots\to W_t^{(2)}$ of $\pazocal{C}$ to $B^{(2)}$, $W_{h_1}^{(2)}$ and $W_{h_1+\ell h_2}^{(2)}$ have different tape words.  As such, $B^{(2)}$ satisfies condition (1) of the statement.

Let $w$ be the tape word of $W_{h_1}^{(2)}$.  By \Cref{simplify rules}, there exist tape words $u,v$ with $\|u\|,\|v\|\leq h_2$ such that the tape word of $W_{h_1+h_2}^{(2)}$ is freely equal to $uwv$.  As such, the tape word $w'$ of $W_{h_1+\ell h_2}^{(2)}$ is freely equal to $u^\ell w v^\ell$.  

It then follows from Lemma 8.1 of \cite{OSconj} that $\|w'\|\geq\ell-\|w\|-2h_2$, so that $\|w\|+\|w'\|\geq\ell-2h_2$.

Note that by \Cref{simplify rules}, $|W_0^{(2)}|_a\geq\|w\|-2h_1$ and $|W_t^{(2)}|_a\geq\|w'\|-2\bar{h}_2-2h_3$.

\textbf{1.} Suppose $\ell\leq 12h_2$.  Then 
\begin{align*}
\|H\|&=h_1+\ell h_2+\bar{h}_2+h_3\leq h_1+(\ell+1)h_2+h_3\leq h_1+12h_2^2+h_2+h_3
\end{align*}
Since $|W_0^{(1)}|_a+|W_t^{(1)}|_a\geq2(h_1+h_2+h_3)$ implies $h_1+h_2+h_3\leq\max(|W_0^{(1)}|_a,|W_t^{(1)}|_a)$, it then follows that $\|H\|\leq 12n^2+n$ for $n=\max(|W_0^{(1)}|_a,|W_t^{(1)}|_a)$.  Hence, the statement is again satisfied by setting $UV=B^{(1)}$.

As a result, it suffices to assume $\ell\geq12h_2$, so that $\|H\|\leq h_1+\frac{1}{12}\ell(\ell+1)+h_3$.

\textbf{2.} Suppose $\ell\leq12(h_1+h_3)$.

Then $\|H\|\leq (h_1+h_3)(\ell+2)\leq 12(h_1+h_3)^2+2(h_1+h_3)$.  But $|W_0^{(1)}|_a+|W_t^{(1)}|_a\geq2(h_1+h_3)$, so that the statement is once more satisfied by setting $UV=B^{(1)}$.

\textbf{3.} Hence, it suffices to assume $\ell\geq12(h_1+h_3)$, so that $\|H\|\leq\frac{1}{12}\ell(\ell+2)\leq\frac{1}{6}\ell^2$ since $\ell\geq2$.  Note then that 
\begin{align*}
|W_0^{(2)}|_a+|W_t^{(2)}|_a&\geq\|w\|+\|w'\|-2(h_1+\bar{h}_2+h_3)\geq\ell-2(h_1+2h_2+\bar{h}_2+h_3) \\
&\geq \ell-2(h_1+3h_2+h_3)\geq\ell-2\ell/3=\ell/3
\end{align*}
So, setting $n=\max(|W_0^{(2)}|_a,|W_t^{(2)}|_a)$, we have $n\geq \ell/6$, so that $\|H\|\leq\frac{1}{6}\ell^2\leq 6n^2$.  Thus, the statement is satisfied by setting $UV=B^{(2)}$.

\end{proof}

\medskip


\subsection{Primitive machines} \

One of the most useful tools in shaping the computational makeup of an $S$-machine is the implementation of \textit{primitive machines}.  These are simple $S$-machines that `run' the state letters past a certain tape word and back, locking the sector in between.  In practice, this type of machine is used as a submachine to force the base of an admissible word to be reduced in order to carry out specific types of computations.

The first primitive machine is the machine customarily named $\textbf{LR}(Y)$ for some alphabet $Y$.  The standard base of the machine is $PQR$, with $P=\{p_1,p_2\}$, $Q=\{q_1,q_2\}$, and $R=\{r_1,r_2\}$.  The state letters with subscript $1$ are the start state letters, while those with subscript $2$ are the end state letters.  The tape alphabets $Y_1$ and $Y_2$ of the machine are two copies of $Y$.  For simplicity, the letter of $Y_i$ corresponding to the letter $y\in Y$ is denoted $y_i$.

The positive rules of $\textbf{LR}(Y)$ are defined as follows:

\begin{itemize}

\item For every $y\in Y$, a rule $\tau_1(y)=[p_1\to p_1, \ q_1\to y_1^{-1}q_1y_2, \ r_1\to r_1]$

\item The \textit{connecting rule} $\zeta=[p_1\xrightarrow{\ell} p_2, \ q_1\to q_2, \ r_1\to r_2]$

\item For every $y\in Y$, a rule $\tau_2(y)=[p_2\to p_2, \ q_2\to y_1q_2y_2^{-1}, \ r_2\to r_2]$

\end{itemize}

For any configuration $W$ of $\textbf{LR}(Y)$, one can associate a reduced word over $Y\cup Y^{-1}$ by deleting the state letters and taking the natural copies of the tape letters.  Note that the application of any rule does not alter this associated word.  The use of this or a similar observation is called a \textit{projection argument}.

One can interpret the function of the rules as follows: The rules $\tau_1(y)^{\pm1}$ are used to delete the letters from the $PQ$-sector and move them to the $QR$-sector; this continues until the $PQ$-sector is empty, at which point the connecting rule may be applied; from there, the rules $\tau_2(y)^{\pm1}$ are used to move the letters back from the $QR$-sector to the $PQ$-sector.  This informal description motivate the naming of the machine, as one can view the state letter of $Q$ as `moving left toward $P$ and then right toward $R$'.

These notions are made precise by the following statement, which can be understood simply through projection arguments:

\begin{lemma}[Lemma 3.1 of \cite{O18}] \label{primitive computations}

Let $\pazocal{C}:W_0\to\cdots\to W_t$ be a reduced computation of $\textbf{LR}(Y)$ in the standard base. Then:

\begin{enumerate}[label=({\arabic*})]

\item if $|W_{i-1}|_a<|W_i|_a$ for some $1\leq i\leq t-1$, then $|W_i|_a<|W_{i+1}|_a$

\item $|W_i|_a\leq\max(|W_0|_a,|W_t|_a)$ for each $i$

\item Suppose $W_0\equiv p_1~u~q_1r_1$ and $W_t\equiv p_2~v~q_2r_2$ for some $u,v\in F(Y_1)$.  Then $u\equiv v$, $|W_i|_a=\|u\|\defeq\ell$ for each $i$, $t=2\ell+1$, and the $PQ$-sector is locked in the rule $W_\ell\to W_{\ell+1}$. Moreover, letting $\bar{u}$ be the word obtained by reading $u$ from right to left, the history $H$ of $\pazocal{C}$ is a copy of $\bar{u}\zeta u$.

\item if $W_0\equiv p_j~u~q_jr_j$ (resp $W_0\equiv p_jq_j~u~r_j$) and $W_t\equiv p_j~v~q_jr_j$ (resp $W_t\equiv p_jq_j~v~r_j$) for some $u,v$ and $j\in\{1,2\}$, then $u\equiv v$ and the computation is empty (i.e $t=0$)


\item if $W_0$ is of the form $p_j~u~q_jr_j$ or $p_jq_j~v~r_j$ for $j\in\{1,2\}$, then $|W_i|_a\geq|W_0|_a$ for every $i$.

\end{enumerate}

\end{lemma}

It is a useful observation that for any $u\in F(Y)$, there exists a computation of $\textbf{LR}(Y)$ of the form detailed in \Cref{primitive computations}(3).  

The following statement has no explicit analogue in previous settings, but follows quickly in much the same way as \Cref{primitive computations}.

\begin{lemma} \label{primitive time}

Let $\pazocal{C}:W_0\to\dots\to W_t$ be a reduced computation of $\textbf{LR}(Y)$ in the standard base.  Then $t\leq2\max(|W_0|_a,|W_t|_a)+1$.

\end{lemma}

\begin{proof}

If the history $H$ of $\pazocal{C}$ contains no connecting rule, then we may apply \Cref{multiply one letter} to one of the two sectors, yielding $t\leq\max(|W_0|_a,|W_t|_a)$.

Hence, we may assume $H\equiv H_1\zeta^{\pm1}H_2$.  Let $\pazocal{C}_2:W_x\to\dots\to W_t$ be the subcomputation with history $H_2$.

As $W_x$ is $\zeta^{\mp1}$-admissible, it follows from \Cref{primitive computations}(4) that $H_2$ contains no occurrence of a connecting rule.  As above, this implies $t-x\leq\max(|W_x|_a,|W_t|_a)$.  But \Cref{primitive computations}(5) also implies $|W_x|_a\leq|W_t|_a$, so that $t-x\leq|W_t|_a$.

By the analogous argument applied to the subcomputation with history $H_1$, we have $x-1\leq|W_0|_a$.  Thus the statement follows.

\end{proof}

The next statement helps understand computations of the machine with an unreduced base:

\begin{lemma} [Lemma 3.4 of \cite{OS19}] \label{primitive unreduced} Suppose $W_0\to\dots\to W_t$ is a reduced computation of $\textbf{LR}(Y)$ with base $PQQ^{-1}P^{-1}$ (or $R^{-1}Q^{-1}QR$) such that $W_0\equiv p_jq_j~u~q_j^{-1}p_j^{-1}$ (or $W_0\equiv r_j^{-1}q_j^{-1}~v~q_jr_j$) for $j\in\{1,2\}$ and some word $u$ (or $v$). Then $|W_0|_a\leq\dots\leq|W_t|_a$ and all state letters of $W_t$ have the index $j$.  In particular, if $W_t\equiv p_jq_j~u'~q_j^{-1}p_j^{-1}$ (or $W_t\equiv r_j^{-1}q_j^{-1}~v'~q_jr_j$), then $t=0$.

\end{lemma}

It is useful to consider also a machine analogous to $\textbf{LR}(Y)$ which instead `runs' the middle state letter right and then left.  This machine, denoted $\textbf{RL}(Y)$, is also considered a primitive machine.

To be precise, identifying the hardware of this machine with that of $\textbf{LR}(Y)$, the positive rules of $\textbf{RL}(Y)$ are:

\begin{itemize}

\item For every $y\in Y$, a rule $\tau_1(y)=[p_1\to p_1, \ q_1\to y_1q_1y_2^{-1}, \ r_1\to r_1]$

\item The \textit{connecting rule} $\xi=[p_1\to p_2, \ q_1\xrightarrow{\ell} q_2, \ r_1\to r_2]$

\item For every $y\in Y$, a rule $\tau_2(y)=[p_2\to p_2, \ q_f\to y_1^{-1}q_2y_2, \ r_2\to r_2]$

\end{itemize}

There are obvious analogues of Lemmas \ref{primitive computations} and \ref{primitive unreduced} in the setting of $\textbf{RL}(Y)$, which can be verified in much the same ways.

\medskip
	

%
%
%

\medskip
	

\section{Enhanced machine} \label{sec-enhanced}

Our goal in this section is to construct an $S$-machine whose computational makeup is sufficient for the proof of \Cref{main-theorem}.  Per the hypotheses of the statement, this construction begins with an arbitrary recognizing $S$-machine $\textbf{S}$ with designated time function $\TM_\textbf{S}$.  We then alter this machine in carefully tailored ways, producing some auxiliary machines and finally the desired `enhanced machine' $\textbf{E}_\textbf{S}$.

%

\subsection{Historical sectors} \

The first auxiliary machine $\textbf{S}_h$ in our construction adds several new `historical' sectors which keep track of the commands applied throughout a computation.  This is done in much the same way as in \cite{O18}, \cite{OS19}, and \cite{WCubic}.

Let $Q_0Q_1\dots Q_s$ be the standard base and $\Phi^+$ be the positive rules of the given machine $\textbf{S}$.  The hardware of the machine $\textbf{S}_h$ is then given as follows:

\begin{itemize}

\item The standard base is $Q_{0,\ell}Q_{0,r}Q_{1,\ell}Q_{1,r}\dots Q_{s,\ell}Q_{s,r}$, where each of $Q_{i,\ell}$ and $Q_{i,r}$ is a copy of the part $Q_i$ of the hardware of $\textbf{S}$.

\item The tape alphabet of the $Q_{i-1,r}Q_{i,\ell}$-sector is that of the $Q_{i-1}Q_i$-sector of $\textbf{S}$.

\item The tape alphabet of each $Q_{i,\ell}Q_{i,r}$-sector is two disjoint copies of $\Phi^+$, called the \textit{left} and the \textit{right historical alphabets} of the sector.

\end{itemize}

Any sector whose tape alphabet consists of two copies of $\Phi^+$ {\frenchspacing (e.g. the $Q_{i,\ell}Q_{i,r}$-sectors)} is called a \textit{historical sector}.  Every other sector corresponding to a two-letter subword of the standard base is called a \textit{working sector}.  Note that the working sectors correspond to sectors $\textbf{S}$.  The input sectors of $\textbf{S}_h$ are then taken to be the working sectors corresponding to the input sectors of $\textbf{S}$.

The positive rules of $\textbf{S}_h$ are in correspondence with those of $\textbf{S}$.  For any $\theta\in\Phi^+$, the corresponding positive rule $\theta_h$ operates as follows:

\begin{itemize}

\item The domain of $\theta_h$ in the working sectors is the same as the domain of $\theta$ in the corresponding sector of the standard base.

\item The domain of $\theta_h$ in each historical sector is the entire tape alphabet.

\item For each $i$, $Q_{i,\ell}(\theta_h)$ and $Q_{i,r}(\theta_h)$ are the natural copies of $Q_i(\theta)$; analogously, $Q_{i,\ell}'(\theta_h)$ and $Q_{i,r}'(\theta_h)$ are the natural copies of $Q_i'(\theta)$.
 
\item The operation of $\theta_h$ in each working sector is analogous to the operation of $\theta$ in the corresponding sector of the standard base.

\item In each $Q_{i,\ell}Q_{i,r}$-sector, $\theta_h$ multiplies on the left by the copy of $\theta^{-1}$ over the left historical alphabet and on the right by the copy of $\theta$ over the right historical alphabet.

\end{itemize}

A benefit of adding historical sectors is in providing a linear estimate for the lengths of reduced computations in terms of the $a$-lengths of its initial and terminal admissible words.

\begin{lemma}[Lemma 3.9 of \cite{O18}] \label{one alphabet historical words}

Let $W_0\to\dots\to W_t$ be a reduced computation of $\textbf{S}_h$ with base $Q_{i,\ell}Q_{i,r}$ and history $H$.  Suppose the $a$-letters of $W_0$ are all from the left (respectively right) historical alphabet.  Then $\|H\|\leq|W_t|_a$ and $|W_0|_a\leq|W_t|_a$.  Moreover, if $t\geq1$, then the tape word of $W_t$ contains letters from the right (respectively left) historical alphabet.

\end{lemma}

The next statement then follows in just the same way:

\begin{lemma}[Lemma 5.14 of \cite{WEmb}] \label{one alphabet historical words unreduced}

Let $W_0\to\dots\to W_t$ be a reduced computation of $\textbf{S}_h$ with base $Q_{i,\ell}Q_{i,\ell}^{-1}$ (respectively $Q_{i,r}^{-1}Q_{i,r}$) and history $H$.  Suppose the $a$-letters of $W_0$ are all from the corresponding right (respectively left) historical alphabet.  Then $|W_0|_a=|W_t|_a-2\|H\|$.  Moreover, if $t\geq1$, then the tape word of $W_t$ contains letters from both the left and right historical alphabets.

\end{lemma}

%
%

%
%

\medskip


\subsection{Adding locked sectors} \

The next auxiliary machine, denoted $\textbf{S}_h'$, is obtained from $\textbf{S}_h$ by simply adding some perpetually locked sectors around the historical sectors.  This `padding' has no meaningful impact on the computational makeup, but is useful for the constructions in the next section.

In particular, the standard base of $\textbf{S}_h'$ is taken to be
$$\left(P_0Q_{0,\ell}Q_{0,r}R_0\right)\left(P_1Q_{1,\ell}Q_{1,r}R_1\right)\dots\left(P_sQ_{s,\ell}Q_{s,r}R_s\right)$$
The hardware of the $Q_{i,\ell}Q_{i,r}$-sectors arises in just the same way as the same named sector of $\textbf{S}_h$.  Moreover, the parts $P_i$ and $R_i$ are in correspondence with the part $Q_i$ of the hardware of $\textbf{S}$ (analogous to the makeup of the parts $Q_{i,\ell}$ and $Q_{i,r}$).

The tape alphabet of the $P_iQ_{i,\ell}$- and $Q_{i,r}R_i$-sectors consist of left and right historical alphabets, each of which is identified with $\Phi^+$.  Accordingly, the list of possible historical sectors grows to include these sectors and any of their unreduced analogues.

Finally, the tape alphabet of the $R_{i-1}P_i$-sector is identified with that of the $Q_{i-1}Q_i$-sector of $\textbf{S}$ (or the $Q_{i-1,r}Q_{i,\ell}$-sector of $\textbf{S}_h$).

The positive rules of $\textbf{S}_h'$ are in bijection with the rules of $\textbf{S}_h$ (and so with the positive rules of $\textbf{S}$), with each rule locking the $P_iQ_{i,\ell}$- and $Q_{i,r}R_i$-sectors and operating on the rest of the standard base in the obvious way.  As such, there are analogues of Lemmas \ref{one alphabet historical words} and \ref{one alphabet historical words unreduced} in this setting.

The length of the standard base of $\textbf{S}_h'$ is henceforth denoted by $N$.  

\medskip


\subsection{Composition of machines} \

The next machine in our construction, denoted $\textbf{E}_\textbf{S}^0$ and called the \textit{standard enhanced machine}, is the \textit{composition} of $\textbf{S}_h'$ with four simple $S$-machines.

In general, suppose the recognizing $S$-machines $\textbf{M}_1,\dots,\textbf{M}_k$ have the same standard base and identical tape alphabets.  Then the composition $\textbf{M}$ of these machines can be understood as the `concatenation' of $k$ `submachines', each of which is a recognizing $S$-machine identified with $\textbf{M}_i$.  Naturally, $\textbf{M}$ is understood to be a recognizing $S$-machine by taking:

\begin{itemize}

\item The start state letters to be those corresponding to the start letters of $\textbf{M}_1$,

\item The end state letters to be those corresponding to the end letters of $\textbf{M}_k$, and

\item The input sectors to be the same as those of $\textbf{M}_1$.

\end{itemize}

A computation of the machine $\textbf{M}$ generally progresses as a concatenation of computations of these submachines, so that the computational makeup of $\textbf{M}$ is informed by that of the machines $\textbf{M}_1,\dots,\textbf{M}_k$.

To achieve this concatenation, each part of the state letters of $\textbf{M}$ is the disjoint union of $k$ subsets, each of which corresponds to the analogous part of the submachine.  Beyond the software of these submachines, the positive rules of $\textbf{M}$ contain \textit{transition rules} $\sigma(i,i+1)$ which switch the states from the the end states of $\textbf{M}_i$ to the start states of $\textbf{M}_{i+1}$ (and perhaps have restricted domains).  For simplicity, denote $\sigma(i,i+1)^{-1}\equiv\sigma(i+1,i)$.

As such, the standard base and tape alphabets of $\textbf{E}_\textbf{S}^0$ is the same as that of $\textbf{S}_h'$, and the makeup of the machine can be fully understood by describing (i) the submachines $\textbf{E}_\textbf{S}^0(1),\dots,\textbf{E}_\textbf{S}^0(5)$, and (ii) the domains of the transition rules in each sector.  To this end:

\begin{itemize}

\item Each part of the state letters of the submachine $\textbf{E}_\textbf{S}^0(1)$ consists of a single letter, functioning as both the start and end state letter of the corresponding part.  The input sectors are the working sectors corresponding to the input sectors of $\textbf{S}$.  The positive rules of the machine are in one-to-one correspondence with $\Phi^+$, with the rule corresponding to $\theta\in\Phi^+$ multiplying each $Q_{i,\ell}Q_{i,r}$-sector on the right by the copy of $\theta$ in the left historical alphabet (the left alphabet being its domain in this sector) and locking all other non-input sectors.

\item The submachine $\textbf{E}_\textbf{S}^0(2)$ operates `in parallel' as the primitive machine $\textbf{LR}=\textbf{LR}(\Phi^+)$ on each of the subwords $Q_{i,\ell}Q_{i,r}R_i$ of the standard base, taking the corresponding left historical alphabets as the copies of $\Phi^+$ and locking all other non-input sectors.  To match the makeup of $\textbf{LR}$, each part of the state letters contains two letters, with the transition from the start to the end letters given by the copy of the connecting rule.

\item $\textbf{E}_\textbf{S}^0(3)$ is a copy of the recognizing machine $\textbf{S}_h'$.

\item The submachine $\textbf{E}_\textbf{S}^0(4)$ operates in parallel as the primitive machine $\textbf{RL}=\textbf{RL}(\Phi^+)$ on each of the subwords $P_iQ_{i,\ell}Q_{i,r}$ of the standard base, taking the corresponding right historical alphabets as the copies of $\Phi^+$ and locking all other sectors.  Again, each part of the state letters contains two letters, with the transition from the start to the end letters given by the copy of the connecting rule.

\item The submachine $\textbf{E}_\textbf{S}^0(5)$ functions in the same way as $\textbf{E}_\textbf{S}^0(1)$, except that the rule corresponding to $\theta\in\Phi^+$ multiplies each $Q_{i,\ell}Q_{i,r}$-sector on the left by the copy of $\theta^{-1}$ in the right alphabet (its domain in this sector) and locks all other (including input) sectors.

\item The transition rules $\sigma(12)$ and $\sigma(23)$ have identical domains: Their domain in each input sector is the entire tape alphabet, their domain in each $Q_{i,\ell}Q_{i,r}$-sector is the left historical alphabet, and their domain in any other sector is empty.

\item The transition rules $\sigma(34)$ and $\sigma(45)$ also have identical domains.  However, their domain in each $Q_{i,\ell}Q_{i,r}$-sector is the right historical alphabet, while that in all other sectors (including the input sectors) is empty.

\end{itemize}

As evidenced by the arguments in the next sections, the composition of $\textbf{S}_h'$ with these four machines functions to control the computational makeup of the machine: 

\begin{itemize}

\item The submachines $\textbf{E}_\textbf{S}^0(1)$ and $\textbf{E}_\textbf{S}^0(5)$, which function to add and delete history words, make the time function equivalent to $\TM_\textbf{S}$ (whereas the time function of $\textbf{S}_h'$ is linear).

\item The submachines $\textbf{E}_\textbf{S}^0(2)$ and $\textbf{E}_\textbf{S}^0(4)$, on the other hand, give tight control on how a reduced computation with unreduced base can proceed.

\end{itemize}

Note that by the assignment of the input sectors, there is a correspondence between the input configurations of $\textbf{S}$ and those of $\textbf{E}_\textbf{S}^0$ given by simply adding several locked sectors.  Given an input configuration $W$ of $\textbf{S}$, we denote the corresponding input configuration of $\textbf{E}_\textbf{S}^0$ by $I(W)$.  Note that the input of $I(W)$ is the same as that of $W$.

\medskip


\subsection{Computations of the standard enhanced machine} \

By construction, the history $H$ of a reduced computation of $\textbf{E}_\textbf{S}^0$ can be factored in such a way that each factor is either a transition rule or a maximal nonempty product of rules of one of the five defining submachines $\textbf{E}_\textbf{S}^0(j)$. The \textit{step history} of a reduced computation is then defined so as to capture the order of the types of these factors. To do this, denote the transition rule $\sigma(ij)$ by the pair $(ij)$ and a factor that is a product of rules in $\textbf{E}_\textbf{S}^0(i)$ simply by $(i)$.

For example, if $H\equiv H'H''H'''$ where $H'$ is a product of rules from $\textbf{E}_\textbf{S}^0(2)$, $H''\equiv\sigma(23)$, and $H'''$ is a product of rules from $\textbf{E}_\textbf{S}^0(3)$, then the step history of a computation with history $H$ is $(2)(23)(3)$. So, the step history of a computation is some concatenation of the `letters'
$$\{(1),(2),(3),(4),(5),(12),(23),(34),(45),(21),(32),(43),(54)\}$$ 
One can omit reference to a transition rule when its existence is clear from its necessity. For example, given a reduced computation with step history $(2)(23)(3)$, one can instead write the step history as $(2)(3)$, as the rule $\sigma(23)$ must occur in order for the subcomputation of $\textbf{E}_\textbf{S}^0(3)$ to be possible.


Certain subwords cannot appear in the step history of a reduced computation. For example, it is clearly impossible for the step history of a reduced computation to contain the subword $(1)(3)$. 

The next statements display the impossibility of some less obvious potential subwords.

\begin{lemma} \label{E primitive step history}

Let $\pazocal{C}:W_0\to\dots\to W_t$ be a reduced computation of $\textbf{E}_\textbf{S}^0$ with base $B$.

\begin{enumerate}[label=(\alph*)]

\item If $B$ contains a subword $B'$ of the form $(Q_{i,\ell}Q_{i,r}R_i)^{\pm1}$, of the form $Q_{i,\ell}Q_{i,r}Q_{i,r}^{-1}Q_{i,\ell}^{-1}$, or of the form $R_i^{-1}Q_{i,r}^{-1}Q_{i,r}R_i$, then the step history of $\pazocal{C}$ is not $(12)(2)(21)$ or $(32)(2)(23)$.

\item If $B$ contains a subword $B'$ of the form $(P_iQ_{i,\ell}Q_{i,r})^{\pm1}$, of the form $P_iQ_{i,\ell}Q_{i,\ell}^{-1}P_i^{-1}$, or of the form $Q_{i,r}^{-1}Q_{i,\ell}^{-1}Q_{i,\ell}Q_{i,r}$, then the step history of $\pazocal{C}$ is not $(34)(4)(43)$ or $(54)(4)(45)$.

\end{enumerate}

\end{lemma}

\begin{proof}

In each case, assuming that the step history is of one of the forbidden forms, the restriction of the subcomputation $W_1\to\dots\to W_{t-1}$ to $B'$ can be identified with a reduced computation of a primitive machine.  But then \Cref{primitive computations}(4) and \Cref{primitive unreduced} produce contradictions.

\end{proof}

\begin{lemma} \label{E run step history}

Let $\pazocal{C}:W_0\to\dots\to W_t$ be a reduced computation of $\textbf{E}_\textbf{S}^0$ with base $B$.

\begin{enumerate}[label=(\alph*)]

\item If $B$ contains a subword $UV$ of the form $(Q_{i,\ell}Q_{i,r})^{\pm1}$ or of the form $Q_{i,r}^{-1}Q_{i,r}$, then the step history of $\pazocal{C}$ is not $(23)(3)(32)$.

\item If $B$ contains a subword $UV$ of the form $(Q_{i,\ell}Q_{i,r})^{\pm1}$ or of the form $Q_{i,\ell}Q_{i,\ell}^{-1}$, then the step history of $\pazocal{C}$ is not $(43)(3)(34)$.

\end{enumerate}

\end{lemma}

\begin{proof}

Suppose $UV$ is of one of the forms described in case (a) and the step history of $\pazocal{C}$ is $(23)(3)(32)$.  Let $\pazocal{C}':W_0'\to\dots\to W_t'$ be the restriction of $\pazocal{C}$ to the $UV$-sector.  Then since both $W_1$ and $W_{t-1}$ are $\sigma(32)$-admissible, the tape words of both $W_1'$ and $W_{t-1}'$ consist entirely of letters from the left historical tape alphabet.  If $t\geq3$, though, then the subcomputation $W_1'\to\dots\to W_{t-1}'$ contradicts either \Cref{one alphabet historical words} or \Cref{one alphabet historical words unreduced}.  But then $t=2$ and the history of $\pazocal{C}$ is the unreduced word $\sigma(23)\sigma(32)$.

The other case follows by the same argument.

\end{proof}

\begin{lemma} \label{inputs accepted}

Let $\pazocal{D}$ be a reduced computation of $\textbf{S}$ accepting the input configuration $W$.  Letting $H\in F(\Phi^+)$ be the history of $\pazocal{D}$, there exists a reduced computation of $\textbf{E}_\textbf{S}^0$ with step history $(1)(2)(3)(4)(5)$ and length $7\|H\|+6$ which accepts the input configuration $I(W)$.

\end{lemma}

\begin{proof}

By construction, there exists a reduced computation $\pazocal{C}_1:I(W)\equiv W_0\to\dots\to W_\ell$ of $\textbf{E}_\textbf{S}^0(1)$ which, in parallel, writes the copy of $H$ in the left historical alphabet of each $Q_{i,\ell}Q_{i,r}$-sector.  It then follows that $W_\ell$ is $\sigma(12)$-admissible and $\ell=\|H\|$.

Next, \Cref{primitive computations}(3) provides a reduced computation $\pazocal{C}_2:W_\ell\cdot\sigma(12)\equiv W_{\ell+1}\to\dots\to W_r$ of $\textbf{E}_\textbf{S}^0(2)$ which applies the `standard' computation of a primitive machine in parallel.  By construction, $W_r$ is $\sigma(23)$-admissible and $r-\ell-1=2\|H\|+1$.

The machine $\textbf{S}_h'$ then provides a computation $\pazocal{C}_3:W_r\cdot\sigma(23)\equiv W_{r+1}\to\dots\to W_x$ of $\textbf{E}_\textbf{S}^0(3)$ which operates as $\pazocal{D}$ in the working sectors.  As such, $W_x$ is $\sigma(34)$-admissible and $x-r-1=\|H\|$.

As above, \Cref{primitive computations}(3) provides a reduced computation $\pazocal{C}_4:W_x\cdot\sigma(34)\equiv W_{x+1}\to\dots\to W_y$ of $\textbf{E}_\textbf{S}^0(4)$ which applies the `standard' primitive computation in parallel, so that $W_y$ is $\sigma(45)$-admissible and $y-x-1=2\|H\|+1$.

Finally, there exists a reduced computation $\pazocal{C}_5:W_y\cdot\sigma(45)\equiv W_{y+1}\to\dots\to W_t$ of $\textbf{E}_\textbf{S}^0(5)$ which erases the copy of $H$ in the right historical alphabet from each $Q_{i,\ell}Q_{i,r}$-sector.  It then follows that $W_t$ is the accept configuration and $t-y-1=\|H\|$.

Combining these computations thus produces a reduced computation satisfying the statement.

\end{proof}

\begin{lemma} \label{E0 language}

The language of accepted inputs of $\textbf{E}_\textbf{S}^0$ is the same as that of $\textbf{S}$.  Moreover, the time function of $\textbf{E}_\textbf{S}^0$ is $\sim$-equivalent to $\TM_\textbf{S}$.

\end{lemma}

\begin{proof}

Let $W'$ be an accepted input configuration of $\textbf{E}_\textbf{S}^0$.  The correspondence between the input configurations of $\textbf{E}_\textbf{S}^0$ and those of $\textbf{S}$ implies there exists an input configuration $W$ of $\textbf{S}$ such that $W'\equiv I(W)$.

Now, let $\pazocal{C}:W'\equiv W_0\to\dots\to W_t$ be a reduced computation accepting $W'$.  By \Cref{E primitive step history} and \Cref{E run step history}, the step history of $\pazocal{C}$ must have prefix $(1)(2)(3)(4)(5)$.

Let $\pazocal{C}_3:W_x\to\dots\to W_y$ be the maximal subcomputation with step history $(3)$.  Let $H\in F(\Phi^+)$ be the word in the rules of $\textbf{S}$ such that the history of $\pazocal{C}_3$ is the natural copy of $H$ in the rules of $\textbf{E}_\textbf{S}^0(3)$.  Since no rule of $\textbf{E}_\textbf{S}^0(1)$ or of $\textbf{E}_\textbf{S}^0(2)$ alters the working sectors, $W_x$ has the same tape words as $W_0$ written in each such sector.  

As $W_x$ is $\sigma(32)$-admissible and $W_y$ is $\sigma(34)$-admissible, $\pazocal{C}_3$ can be identified with a reduced computation of $\textbf{S}_h'$ between a start and an end configuration.  Moreover, since the tape words of $W_x$ and $W_y$ in the $Q_{i,\ell}Q_{i,r}$-sectors are formed over the corresponding left and right historical alphabets, respectively, it follows that the tape word in each such sector of $W_x$ {\frenchspacing(resp. $W_y$)} is the natural copy of $H$ over the left {\frenchspacing(resp. right)} historical alphabet.

In particular, since the working sectors of $W_y$ must have empty tape words, $\pazocal{C}_3$ can be identified with a reduced computation of $\textbf{S}$ accepting the start configuration formed by the working sectors of $W_x$.  But this is the same as $W$, so that the input accepted by $\pazocal{C}$ is the same as that of $W$.  Combining this with \Cref{inputs accepted}, we then have that the language of accepted inputs of $\textbf{E}_\textbf{S}^0$ is the same as that of $\textbf{S}$.

Moreover, note that the length of $\pazocal{C}$ is at least $\|H\|\geq\tm(W)$.  As $\pazocal{C}$ is an arbitrary computation accepting $W'$, it follows that $\tm(W')\geq\tm(W)$.

Hence, the equivalence of the time functions of the two machines follows from \Cref{inputs accepted}.

\end{proof}

\begin{lemma} \label{E standard one-step}

Let $\pazocal{C}:W_0\to\dots\to W_t$ be a reduced computation of $\textbf{E}_\textbf{S}^0$ with step history $(j)$ for some $j\in\{1,\dots,5\}$.  Suppose the base of $\pazocal{C}$ is of the form $(P_iQ_{i,\ell}Q_{i,r}R_i)^{\pm1}$.

\begin{enumerate}[label=(\alph*)]

\item If $j\in\{1,5\}$, then $t\leq2\max(|W_0|_a,|W_t|_a)$.

\item If $j=3$, then $t\leq\max(|W_0|_a,|W_t|_a)$.  Moreover, if $W_0$ is $\sigma(32)$- or $\sigma(34)$-admissible, then $|W_0|_a\leq|W_t|_a$.

\item If $j\in\{2,4\}$, then $t\leq2\max(|W_0|_a,|W_t|_a)+1$.  Moreover, if $W_0$ is $\sigma(j,j-1)$- or $\sigma(j,j+1)$-admissible, then $|W_0|_a\leq|W_t|_a$.

\end{enumerate}

\end{lemma}

\begin{proof}

Suppose $j\in\{1,5\}$.  Then every rule of $\textbf{E}_\textbf{S}^0(j)$ locks the $P_iQ_{i,\ell}$- and $Q_{i,r}R_i$-sectors.  As such, $\pazocal{C}$ can be identified with a reduced computation of a two-letter base satisfying the hypotheses of \Cref{multiply one letter}.  Hence, $t\leq2\max(|W_0|_a,|W_t|_a)$.

Similarly, if $j=3$, then $\pazocal{C}$ can be identified with a reduced computation of $\textbf{S}_h$ with base $Q_{i,\ell}Q_{i,r}$.  Such a computation satisfies the hypotheses of \Cref{multiply two letters}, so that again $t\leq\max(|W_0|_a,|W_t|_a)$.  Moreover, if $W_0$ is $\sigma(32)$- or $\sigma(34)$-admissible, then the computation satisfies the hypotheses of \Cref{one alphabet historical words}, so that $|W_0|_a\leq|W_t|_a$.

Finally, if $j\in\{2,4\}$, then $\pazocal{C}$ can be identified with a reduced computation of a primitive machine in the standard base, so that \Cref{primitive time} implies $t\leq2\max(|W_0|_a,|W_t|_a)+1$.  Moreover, if $W_0$ is $\sigma(j,j-1)$- or $\sigma(j,j+1)$-admissible, then \Cref{primitive computations}(5) implies $|W_0|_a\leq|W_t|_a$.

\end{proof}

\begin{lemma} \label{E time (12)}

Let $\pazocal{C}:W_0\to\dots\to W_t$ be a reduced computation of $\textbf{E}_\textbf{S}^0$ whose base is of the form $(P_iQ_{i,\ell}Q_{i,r}R_i)^{\pm1}$.  If the step history of $\pazocal{C}$ or its inverse is a subword of $(2)(3)(4)$, then $t\leq6\max(|W_0|_a,|W_t|_a)+5$.  Moreover, if $W_0$ is $\sigma(21)$- or $\sigma(45)$-admissible, then $|W_0|_a\leq|W_t|_a$.

\end{lemma}

\begin{proof}

Suppose the step history of $\pazocal{C}$ is a subword of $(2)(3)(4)$; the other case is proved identically.  Let $W_0\to\dots\to W_x$ be the maximal initial subcomputation with step history $(j)$ for some $j\in\{2,3,4\}$.  By \Cref{E standard one-step}, in any case $x\leq2\max(|W_0|_a,|W_x|_a)+1$, while if $W_0$ is $\sigma(21)$-admissible (and so $j=2$) then (c) implies $|W_0|_a\leq|W_x|_a$.  

Hence, it suffices to assume $x<t$.  Note that since $W_x$ is $\sigma(j,j+1)$-admissible, \Cref{E standard one-step}(b) and (c) imply $|W_x|_a\leq|W_0|_a$.  In particular, if $W_0$ is $\sigma(21)$-admissible, then $|W_0|_a=|W_x|_a$.

Let $W_{x+1}\to\dots\to W_y$ be the maximal subcomputation with step history $(j+1)$.  Then as above \Cref{E standard one-step} implies $y-x-1\leq2\max(|W_x|_a,|W_y|_a)+1$ and $|W_x|_a\leq|W_y|_a$.  As a result, $y\leq4\max(|W_0|_a,|W_y|_a)+3$ and $|W_0|_a\leq|W_y|_a$ if $W_0$ is $\sigma(21)$-admissible.

Hence, it suffices to assume $y<t$, so that the step history of $\pazocal{C}$ is $(2)(3)(4)$.  Similar to above this implies $|W_y|_a\leq|W_0|_a$ with equality if $W_0$ is $\sigma(21)$-admissible.  

But then again the subcomputation $W_{y+1}\to\dots\to W_t$ of step history $(4)$ satisfies the hypotheses of \Cref{E standard one-step}(c), so that $t\leq6\max(|W_0|_a,|W_t|_a)+5$ and $|W_y|_a\leq|W_t|_a$.

\end{proof}

\begin{lemma} \label{E standard no (1)}

Let $\pazocal{C}:W_0\to\dots\to W_t$ be a reduced computation of $\textbf{E}_\textbf{S}^{0}$ whose base is of the form $(P_iQ_{i,\ell}Q_{i,r}R_i)^{\pm1}$.  If the step history of $\pazocal{C}$ does not contain the letters $(1)$, $(12)$, or $(21)$, then $t\leq14\max(|W_0|_a,|W_t|_a)+12$.

\end{lemma}

\begin{proof}

By Lemmas \ref{E primitive step history} and \ref{E run step history}, the step history of $\pazocal{C}$ must be a subword of $(2)(3)(4)(5)(4)(3)(2)$.  Moreover, by \Cref{E time (12)} it suffices to assume that $\pazocal{C}$ contains a nonempty maximal subcomputation $W_x\to\dots\to W_y$ with step history $(5)$.

\Cref{E standard one-step}(a) then implies $y-x\leq2\max(|W_x|_a,|W_y|_a)$, so that we may assume $x>0$ or $y<t$.  In the former case, the subcomputation $W_0\to\dots\to W_{x-1}$ satisfies the hypotheses of \Cref{E time (12)}, so that $x\leq6|W_0|_a+6$ and $|W_x|_a\leq|W_0|_a$.  In the latter, the same argument applies to the subcomputation $W_{y+1}\to\dots\to W_t$, so that $t-y\leq6|W_t|_a+6$ and $|W_y|_a\leq|W_t|_a$.

Thus, $t\leq6|W_0|_a+6|W_t|_a+2\max(|W_0|_a,|W_t|_a)+12$, implying the statement.

\end{proof}

The next statement follows in just the same way:

\begin{lemma} \label{E standard no (5)}

Let $\pazocal{C}:W_0\to\dots\to W_t$ be a reduced computation of $\textbf{E}_\textbf{S}^{0}$ whose base is of the form $(P_iQ_{i,\ell}Q_{i,r}R_i)^{\pm1}$.  If the step history of $\pazocal{C}$ does not contain the letters $(5)$, $(54)$, or $(45)$, then $t\leq14\max(|W_0|_a,|W_t|_a)+12$.

\end{lemma}





\begin{lemma} \label{E standard accepted (4)}

If $\pazocal{C}:W_0\to\dots\to W_t$ is a reduced computation of $\textbf{E}_\textbf{S}^0$ in the standard base with step history $(34)(4)(45)$, then $W_t$ is accepted by a reduced computation of $\textbf{E}_\textbf{S}^0(5)$.

\end{lemma}

\begin{proof}

Let $\pazocal{C}_4:W_1\to\dots\to W_{t-1}$ be the maximal subcomputation with step history $(4)$ and $\pazocal{C}_{4,i}:W_{1,i}\to\dots\to W_{t-1,i}$ be its restriction to the subword $P_iQ_{i,\ell}Q_{i,r}$ of the standard base.  Then $\pazocal{C}_{4,i}$ can be identified with a reduced computation of $\textbf{RL}$ satisfying the hypotheses of \Cref{primitive computations}(3).  As such, the history of $\pazocal{C}_{4,i}$ is uniquely determined by the tape word of $W_{t-1,i}$ in the $Q_{i,\ell}Q_{i,r}$-sector, and hence the history of $\pazocal{C}$ is uniquely determined by the tape word of $W_t$ in this sector.

In particular, there exists $H\in F(\Phi^+)$ such that the tape word of $W_t$ in each $Q_{i,\ell}Q_{i,r}$-sector is the natural copy of $H$ over the corresponding right historical alphabet.  But then the natural copy of $H$ in the rules of $\textbf{E}_\textbf{S}^0(5)$ is the history of a reduced computation which accepts $W_t$.

\end{proof}

\begin{lemma} \label{E standard accepted (3)}

If $\pazocal{C}:W_0\to\dots\to W_t$ is a reduced computation of $\textbf{E}_\textbf{S}^0$ in the standard base with step history $(23)(3)(34)$, then $W_t$ is accepted by a reduced computation with step history $(4)(5)$.

\end{lemma}

\begin{proof}

Let $\pazocal{C}_3:W_1\to\dots\to W_{t-1}$ be the maximal subcomputation with step history $(3)$ and $\pazocal{C}_{3,i}:W_{1,i}\to\dots\to W_{t-1,i}$ be its restriction to the subword $Q_{i,\ell}Q_{i,r}$ of the standard base.  Since $W_1$ is $\sigma(32)$-admissible and $W_{t-1}$ is $\sigma(34)$-admissible, the tape word of $W_{1,i}$ is formed over the corresponding left historical alphabet and that of $W_{t-1,i}$ is formed over the corresponding right historical alphabet.

Let $H\in F(\Phi^+)$ be such that the history of $\pazocal{C}_3$ is the natural copy of $H$ over the rules of $\textbf{E}_\textbf{S}^0(3)$.  By construction, it follows that the tape word of $W_{1,i}$ is the natural copy of $H$ over the corresponding left historical alphabet and that of $W_{t-1,i}$ is the natural copy over the corresponding right historical alphabet.  In particular, the tape word of $W_t$ in each $Q_{i,\ell}Q_{i,r}$-sector is the natural copy of $H$ over the corresponding right historical alphabet.

Hence, the reduced computation of $\textbf{RL}$ corresponding to $H$ described by \Cref{primitive computations}(3) produces a reduced computation $\pazocal{D}:W_t\to\dots\to W_x$ with step history $(4)$ such that $W_x$ is $\sigma(45)$-admissible.  Thus, $W_{t-1}\to W_t\to\dots\to W_x\to W_x\cdot\sigma(45)$ is a reduced computation with step history $(34)(4)(45)$, and so the statement follows from \Cref{E standard accepted (4)}.

\end{proof}


Let $W$ be a configuration of $\textbf{E}_\textbf{S}^0$.  For each $i\in\{0,\dots,s\}$, let $W(i)$ be the admissible subword of $W$ which has base $P_iQ_{i,\ell}Q_{i,r}R_i$.  In this case, the \textit{working length} of $W$ is defined to be the number of tape letters from working sectors comprising $W$, {\frenchspacing i.e. $|W|_{wk}=|W|_a-\sum|W(i)|_a$,  }

Note that if $W$ is accepted, then since each rule of $\textbf{E}_\textbf{S}^0$ operates in parallel on the relevant sectors, $W(i)$ and $W(j)$ are copies of one other.  

\begin{lemma} \label{E standard (1)}

Let $W$ be an accepted configuration of $\textbf{E}_\textbf{S}^0$ whose state letters belong to the hardware of the submachine $\textbf{E}_\textbf{S}^0(1)$.  Then there exists a reduced computation accepting $W$ with step history $(1)(2)(3)(4)(5)$ and length at most $7\TM_\textbf{S}(|W|_{wk})+|W(i)|_a+6$ for all $i\in\{0,\dots,s\}$.

\end{lemma}

\begin{proof}

Note that every rule for which a configuration of $\textbf{E}_\textbf{S}^0$ may be admissible locks the $P_iQ_{i,\ell}$-sectors, the $Q_{i,r}R_i$-sectors, and the non-input working sectors.  So, as $W$ is an accepted configuration, its tape words in these sectors must be empty.

Moreover, every rule for which a configuration of $\textbf{E}_\textbf{S}^0$ may be admissible has domain restricted to the left historical alphabet in the $Q_{i,\ell}Q_{i,r}$-sectors.  As a result, the tape words of $W$ in these sectors must be words over the left historical alphabet.

Now, as $W$ is accepted, there exists $H\in F(\Phi^+)$ such that the tape word of $W$ in each $Q_{i,\ell}Q_{i,r}$-sector is the natural copy of $H$ over the corresponding left historical alphabet.  As such, there exists a reduced computation of $\textbf{E}_\textbf{S}^0$ with length $\|H\|$ which begins with $W$ and ends with an accepted input configuration $W'$.  \Cref{inputs accepted} then provides a reduced computation of $\textbf{E}_\textbf{S}^0$ accepting $W'$ with length at most $7\TM_\textbf{S}(|W'|_a)+6$.  

The statement thus follows by noting that $|W|_{wk}=|W'|_a$ and $\|H\|=|W(i)|_a$ for all $i$.

\end{proof}


\begin{lemma} \label{E generalized time}

For any accepted configuration $W$ of $\textbf{E}_\textbf{S}^0$, there exists an accepting computation of length at most $7\TM_\textbf{S}(12\|W\|)+14|W|_a+12$.  

\end{lemma}

\begin{proof}

Let $\pazocal{C}:W\equiv W_0\to\dots\to W_t$ be an accepting computation whose step history is of minimal length and let $\pazocal{C}(i):W_0(i)\to\dots\to W_t(i)$ be its restriction to the subword $P_iQ_{i,\ell}Q_{i,r}R_i$ of the standard base.  By Lemmas \ref{E primitive step history}, \ref{E run step history}, and \ref{E standard accepted (3)}, the step history of $\pazocal{C}$ is then a suffix of $$(3)(32)(2)(21)(1)(12)(2)(23)(3)(34)(4)(45)(5)$$
Let $W_y\to\dots\to W_t$ be the maximal subcomputation whose step history does not contain the letter $(1)$ or $(12)$.  By \Cref{E standard no (1)}, it follows that $t-y\leq14\max(|W_y(i)|_a,|W_t(i)|_a)+12$.  But $W_t$ is the accept configuration, so that $|W_t|_a=0$, and hence $t-y\leq14|W_y(i)|_a+12$.  Hence, it suffices to assume $y>0$.



Now, let $W_x\to\dots\to W_{y-1}$ be the maximal subcomputation with step history $(1)$.  By \Cref{E standard (1)}, we may assume without loss of generality that $t-x\leq 7\TM_\textbf{S}(|W_x|_{wk})+|W_x(i)|_a+6$.  Hence, it suffices to assume $x>0$.

As a result, the step history of the initial subcomputation $W_0\to\dots\to W_{x-1}$ is a suffix of $(3)(2)$.  But then since $W_{x-1}$ is $\sigma(21)$-admissible, applying \Cref{E time (12)} to the inverse computation $W_{x-1}(i)\to\dots\to W_0(i)$ implies $|W_x(i)|_a\leq|W_0(i)|_a$ and $x-1\leq6|W_0(i)|_a+5$ for all $i\in\{0,\dots,s\}$.

By \Cref{simplify rules}, it follows that $|W_x|_{wk}\leq|W_0|_{wk}+2sx\leq|W_0|_{wk}+12\sum (|W_0(i)|_a+1)\leq12\|W\|$.  The statement thus follows.

\end{proof}

\medskip


\subsection{The enhanced machine} \

The main machine in our construction, the \textit{enhanced machine} $\textbf{E}_\textbf{S}$, is the `circular analogue' of the standard enhanced machine $\textbf{E}_\textbf{S}^0$.  The definition of this machine is given in much the same way as in \cite{O18}, \cite{WEmb}, \cite{W}, and many others.

In particular, the standard base of $\textbf{E}_\textbf{S}$ is the same as that of $\textbf{E}_\textbf{S}^0$, and each part of the state letters is identical to its counterpart in that setting.  However, a tape alphabet is also assigned to the space after the final letter of the standard base, corresponding to the tape alphabet of the $R_sP_0$-sector.  As such, it is permitted for an admissible word of $\textbf{E}_\textbf{S}$ to have base
$$P_sQ_{s,\ell}Q_{s,r}R_sP_0Q_{0,\ell}Q_{0,\ell}^{-1}P_0^{-1}R_s^{-1}Q_{s,r}^{-1}$$
{\frenchspacing i.e. to} essentially `wrap around' the standard base. An $S$-machine with this property is called a \textit{cyclic machine}, as one can think of the standard base as being written on a circle.  As will be seen in \Cref{sec-associated-groups}, this is a natural consideration given the structure of the associated groups.

In this setting, the tape alphabet assigned to the $R_sP_0$-sector is empty, while the tape alphabets of the other sectors are defined as in $\textbf{E}_\textbf{S}^0$.

The positive rules of $\textbf{E}_\textbf{S}$ correspond to those of $\textbf{E}_\textbf{S}^0$, operating on the copy the hardware of $\textbf{E}_\textbf{S}^0$ in the same way and, naturally, locking the new sector.  The input sectors are also the same as that of $\textbf{E}_\textbf{S}^0$. The corresponding definitions (for example, submachines, historical sectors, working sectors, {\frenchspacing etc.) then} extend in the obvious way, as do all statements pertaining to the machine.

Now, fix an arbitrary cyclic $S$-machine $\textbf{M}$. The base of an admissible word in the hardware of $\textbf{M}$ is said to be \textit{circular} if it starts and ends with the same base letter.   Specifically, an unreduced circular base is called \textit{defective}.  As in previous literature, a circular base is said to be \textit{revolving} if none of its proper subwords is circular, while a revolving defective base is called \textit{faulty}.

Suppose $W$ is an admissible word in the hardware of $\textbf{M}$ whose base $B\equiv xvx$ is circular. If $v$ has the form $v_1yv_2$ for some letter $y$, then the word $B'\equiv yv_2xv_1y$ is also a circular base of an admissible word $W'$ satisfying $|W'|_a=|W|_a$. In this case, $B'$ and $W'$ are called \textit{cyclic permutations} of $B$ and $W$, respectively.  Note that if $B$ is revolving, then $B'$ is also revolving. 

Given a reduced computation $\pazocal{C}:W_0\to\dots\to W_t$ with circular base $B$ and history $H$, for any cyclic permutation $B'$ of $B$ there exists a reduced computation $\pazocal{C}':W_0'\to\dots\to W_t'$ with base $B'$ and history $H$ and such that $W_i'$ is the corresponding cyclic permutation of $W_i$.  In this case, $\pazocal{C}'$ is also called a \textit{cyclic permutation} of $\pazocal{C}$.

Similar to the terminology introduced in \cite{CW}, define the \textit{universal reach relation} of $\textbf{M}$ to be the binary relation $\REACH^{uni}_\textbf{M}$ on the set of admissible words of $\textbf{M}$ with circular base given by $(W_1,W_2)\in \REACH^{uni}_\textbf{M}$ if and only if there exists a reduced computation between $W_1$ and $W_2$.

Given $(W_1,W_2)\in\REACH^{uni}_\textbf{M}$, a reduced computation between $W_1$ and $W_2$ is said to be \textit{minimal} if the length of its history is minimal amongst all computations realizing the relation.

\medskip


\subsection{Reduced circular bases} \

We restrict our attention in this section to elements $(W_1,W_2)\in\REACH^{uni}_{\textbf{E}_\textbf{S}}$ such that $W_1$ and $W_2$ have reduced circular bases, with the goal being to bound the length of a minimal computation between $W_1$ and $W_2$ in terms of the $a$-length of their revolving subwords.


To begin this analysis, note that for any admissible word $W$ whose base is reduced and circular, there exists a cyclic permutation $W'$ of $W$ whose base is $\left((B_{std})^kP_0\right)^\eps$, where

\begin{itemize}

\item $B_{std}\equiv P_0Q_{0,\ell}Q_{0,r}R_0\dots P_sQ_{s,\ell}Q_{s,r}R_s$ is the standard base of the machine,

\item $\eps\in\{\pm1\}$, and 

\item $k\in\N$  

\end{itemize}

Hence, passing to this cyclic permutation or its inverse, it suffices to assume that the base of both $W_1$ and $W_2$ is $B_k\equiv (B_{std})^kP_0$ for some $k\in\N$.

By the construction of the machine, note that for any admissible word $V$ of $\textbf{E}_\textbf{S}$ with base $B_k$, there exist configurations $V^{(1)},\dots,V^{(k)}$ of $\textbf{E}_\textbf{S}$ such that $V\equiv V^{(1)}\dots V^{(k)}p_0$ for some $p_0\in P_0$.

\begin{lemma} \label{B_k accepted time}

Let $(W_1,W_2)\in\REACH_{\textbf{E}_\textbf{S}}^{uni}$ such that each $W_i$ has base $B_k$.  Suppose there exist accepted configurations $W_1',W_2'$ of $\textbf{E}_\textbf{S}$ such that $W_i^{(j)}\equiv W_i'$ for $i=1,2$ and $j\in\{1,\dots,k\}$.  Then the length of a minimal computation between $W_1$ and $W_2$ is at most $14\TM_\textbf{S}(12n+12N)+28n+24$, where $n=\max(|W_1'|_a,|W_2'|_a)$.

\end{lemma}

\begin{proof}


By \Cref{E generalized time}, for $i=1,2$ there exists a reduced computation $\pazocal{C}_i'$ accepting $W_i'$ of length at most $7\TM_\textbf{S}(12\|W_i'\|)+14|W_i'|_a+12\leq7\TM_\textbf{S}(12n+12N)+14n+12$.  Let $H_i$ be the history of $\pazocal{C}_i'$.

Assuming $W_1\neq W_2$, then since the tape alphabet of the $R_sP_0$-sector is empty, $\pazocal{C}_i'$ can then be extended to a reduced computation $\pazocal{C}_i$ between $W_i$ and the admissible word $(W_{ac})^k p_0^{(f)}$, where $W_{ac}$ is the accept configuration and $p_0^{(f)}$ is the end state letter in $P_0$.  Note that $H_i$ is the history of $\pazocal{C}_i$.

Hence, $H_1H_2^{-1}$ is freely equal to the history of a reduced computation between $W_1$ and $W_2$ with length at most $\|H_1\|+\|H_2\|\leq14\TM_\textbf{S}(12n+12N)+28n+24$.

\end{proof}

\begin{lemma} \label{B_k (4)}

Let $\pazocal{C}:V_0\to\dots\to V_t$ be a reduced computation of $\textbf{E}_\textbf{S}$ with base $B_k$ and step history $(34)(4)(45)$.  Then there exist accepted configurations $V_i'$ of $\textbf{E}_\textbf{S}$ such that $V_i^{(j)}\equiv V_i'$ for $i\in\{0,\dots,t\}$ and $j\in\{1,\dots,k\}$.

\end{lemma}

\begin{proof}

Let $\pazocal{C}':V_1\to\dots\to V_{t-1}$ be the maximal subcomputation with step history $(4)$.  As $V_1$ is $\sigma(43)$- and $V_{t-1}$ is $\sigma(45)$-admissible, the restriction of $\pazocal{C}'$ to any subword of $B_k$ of the form $P_iQ_{i,\ell}Q_{i,r}$ can be identified with a reduced computation of $\textbf{RL}$ that satisfies the hypotheses of \Cref{primitive computations}(3).  Hence, the tape word of $V_t$ in each $Q_{i,\ell}Q_{i,r}$-sector is uniquely determined by the history of $\pazocal{C}$, and so there exists a configuration $V_t'$ such that $V_t^{(j)}\equiv V_t'$ for all $j\in\{1,\dots,k\}$.  That $V_t'$ is an accepted configuration follows from \Cref{E standard accepted (4)}.

\end{proof}

\begin{lemma} \label{B_k (3)}

Let $\pazocal{C}:V_0\to\dots\to V_t$ be a reduced computation of $\textbf{E}_\textbf{S}$ with base $B_k$ and step history $(23)(3)(34)$.  Then there exist accepted configurations $V_i'$ of $\textbf{E}_\textbf{S}$ such that $V_i^{(j)}\equiv V_i'$ for $i\in\{0,\dots,t\}$ and $j\in\{1,\dots,k\}$.

\end{lemma}

\begin{proof}

As in the proof of \Cref{E standard accepted (3)}, the tape word of $V_t$ in each $Q_{i,\ell}Q_{i,r}$-sector is uniquely determined by the history of $\pazocal{C}$.  As such, there exists a configuration $V_t'$ such that $V_t^{(j)}\equiv V_t'$ for all $j\in\{1,\dots,k\}$.  But \Cref{E standard accepted (3)} implies $V_t'$ is accepted.

\end{proof}

%
%

\begin{lemma} \label{E reduced circular}

Let $(W_1,W_2)\in\REACH_{\textbf{E}_\textbf{S}}^{uni}$ such that the base $B$ of $W_i$ is reduced and circular.  Then there exists a revolving subword $B'$ of $B$ such that the length of a minimal computation between $W_1$ and $W_2$ is at most $$14\TM_\textbf{S}(12n+12N)+28n+24$$ for $n=\max(|W_1'|_a,|W_2'|_a)$, where $W_i'$ is the admissible subword of $W_i$ with base $B'$.

\end{lemma}

\begin{proof}

As $B$ is necessarily a cyclic permutation of $B_k^{\pm1}$ for some $k$, it suffices to assume that $B\equiv B_k$ and show that the statement is satisfied if we choose $B'$ to be any subword of the form $B_{std}P_0$.  Moreover, since the $R_sP_0$-sector has empty tape alphabet, it suffices to show the analogous statement when $B'$ is taken to be a subword of the form $B_{std}$.

First, note that \Cref{B_k accepted time} implies that it suffices to assume there exists no accepted configuration $V_1'$ such that $W_1^{(j)}\equiv V_1'$ for all $j\in\{1,\dots,k\}$.

Now let $\pazocal{C}:W_1\equiv V_0\to\dots\to V_t\equiv W_2$ be a reduced computation between $W_1$ and $W_2$.  Lemmas \ref{B_k (4)} and \ref{B_k (3)} then imply that the step history of $\pazocal{C}$ contains no subwords of the form $(34)(4)(45)$, $(54)(4)(43)$, $(23)(3)(34)$, or $(43)(3)(32)$.


But then Lemmas \ref{E primitive step history} and \ref{E run step history} imply that the step history of $\pazocal{C}$ either cannot contain:

\begin{itemize}

\item the letters $(1)$, $(12)$, and $(21)$; or 

\item the letters $(5)$, $(54)$, and $(45)$.  

\end{itemize}

In the former case, the restriction of $\pazocal{C}$ to any subword of the form $P_iQ_{i,\ell}Q_{i,r}R_i$ satisfies the hypotheses of \Cref{E standard no (1)}; in the latter, such a restriction satisfies the hypotheses of \Cref{E standard no (5)}.  In either case, we achieve a sufficient bound.

\end{proof}

\medskip


\subsection{Defective bases} \label{sec-defective} \

Now, our main goal is to achieve a bound analogous to that of \Cref{E reduced circular} for any related pair $(W_1,W_2)\in\REACH_{\textbf{E}_\textbf{S}}^{uni}$.  Given that this bound is established for pairs with reduced circular base, our focus now shifts to pairs with unreduced circular base, {\frenchspacing i.e. with} defective base.

To begin, the next pair of statements help to understand the possible step history of the relevant reduced computations.

\begin{lemma} \label{Defective (34)(4)(45)}

Let $\pazocal{C}:W_0\to\dots\to W_t$ be a reduced computation of $\textbf{E}_\textbf{S}$ with base $B$.  If the step history of $\pazocal{C}$ is $(34)(4)(45)$ or $(54)(4)(43)$, then $B$ must be reduced.

\end{lemma}

\begin{proof}

By hypothesis, the history of $\pazocal{C}$ must contain a rule $\sigma$ of the form $\sigma(34)^{\pm1}$ and a rule $\zeta$ operating in parallel as the connecting rule of a primitive machine.  Note that all sectors are locked by $\sigma$ or by $\zeta$.  As such, the statement follows from \Cref{locked sectors}.

\end{proof}

\begin{lemma} \label{Defective (23)(3)(34)}

Let $\pazocal{C}:W_0\to\dots\to W_t$ be a reduced computation of $\textbf{E}_\textbf{S}$ with base $B$.  If the step history of $\pazocal{C}$ is $(23)(3)(34)$ or $(43)(3)(32)$, then $B$ is reduced.

\end{lemma}

\begin{proof}

Assume that the step history is $(23)(3)(34)$ and suppose $B$ is unreduced.

As $\sigma(34)$ locks every sector of the standard base except for those of the form $Q_{i,\ell}Q_{i,r}$, it follows from \Cref{locked sectors} that $B$ must contain a two-letter subword $UV$ of the form $Q_{i,\ell}Q_{i,\ell}^{-1}$ or of the form $Q_{i,r}^{-1}Q_{i,r}$.  

Let $W_i'$ be the admissible subword of $W_i$ with base $UV$.  Since $W_1$ is $\sigma(34)$-admissible, the tape word of $W_1'$ consists entirely of letters over the left historical alphabet.  Similarly, the tape word of $W_{t-1}'$ consists entirely of letters over the right historical alphabet.

Since $\pazocal{C}$ is reduced, $t\geq3$, so that $W_1'\to\dots\to W_{t-1}'$ is a nonempty reduced computation.  But in either case for $UV$, this computation (or its inverse) satisfies the hypotheses of \Cref{one alphabet historical words unreduced} and contradicts the conclusion.

\end{proof}

\begin{lemma} \label{Defective PQQR}

Let $\pazocal{C}:W_0\to\dots\to W_t$ be a reduced computation of $\textbf{E}_\textbf{S}$ with defective base $B$.  Suppose $B$ contains a subword $B'$ of the form $(P_iQ_{i,\ell}Q_{i,r}R_i)^{\pm1}$.  Then letting $\pazocal{C}':W_0'\to\dots\to W_t'$ be the restriction of $\pazocal{C}$ to $B'$, we have $t\leq14\max(|W_0'|_a,|W_t'|_a)+12$.

\end{lemma}

\begin{proof}

Lemmas \ref{Defective (34)(4)(45)} and \ref{Defective (23)(3)(34)} imply that the step history of $\pazocal{C}$ either cannot contain either:

\begin{itemize}

\item the letters $(1)$, $(12)$, and $(21)$; or 

\item the letters $(5)$, $(54)$, and $(45)$.  

\end{itemize}

The statement then follows by \Cref{E standard no (1)} or \Cref{E standard no (5)}.

\end{proof}

As a result of \Cref{Defective PQQR}, it suffices to restrict our attention to defective bases that contain no subword of the form $(P_iQ_{i,\ell}Q_{i,r}R_i)^{\pm1}$.  Indeed, even if a cyclic permutation of a defective base contains such a subword, then \Cref{Defective PQQR} may be applied.  

To this end, a defective base said to be \textit{strongly defective} if none of its cyclic permutations contain a subword of the form $(P_iQ_{i,\ell}Q_{i,r}R_i)^{\pm1}$.

\begin{lemma} \label{strongly defective (34)}

Let $\pazocal{C}:W_0\to\dots\to W_t$ be a minimal computation with strongly defective base $B$.  
Suppose the history of $\pazocal{C}$ contains a letter of the form $\sigma(34)^{\pm1}$ which is neither the first nor the last letter.  Then:

\begin{enumerate}[label=(\alph*)]

\item There exists a subword $B'$ of a cyclic permutation of $B$ of the form $P_iQ_{i,\ell}Q_{i,\ell}^{-1}P_i^{-1}$

\item Letting $\pazocal{C}':W_0'\to\dots\to W_t'$ be the restriction of a cyclic permutation of $\pazocal{C}$ to the subword $B'$, then $t\leq2\max(|W_0'|_a,|W_t'|_a)+1$.

\end{enumerate}

\end{lemma}

\begin{proof}

Since $B$ is defective and $\sigma(34)$ locks every sector of the standard base other than those of the form $Q_{i,\ell}Q_{i,r}$, it follows that $B$ must contain a two-letter subword of the form $Q_{i,\ell}Q_{i,\ell}^{-1}$ (and another of the form $Q_{i,r}^{-1}Q_{i,r}$).  So, as the tape word of such a sector must be over the corresponding right tape alphabet, a word which is $\sigma(34)$-admissible cannot be $\sigma(32)$-admissible.

As a result, perhaps passing to the inverse computation it suffices to assume the existence of indices $0\leq x<y<z-1<t$ such that:

\begin{itemize}

\item $W_x\to\dots\to W_y$ is a maximal subcomputation with step history $(3)$,

\item the transition $W_y\to W_{y+1}$ is given by $\sigma(34)$, and

\item $W_{y+1}\to\dots\to W_z$ is a maximal subcomputation with step history $(4)$.

\end{itemize}

Now, since $B$ is defective and contains a subword of the form $Q_{i,\ell}Q_{i,\ell}^{-1}$, part (a) of the statement follows from \Cref{locked sectors}.  Let $\pazocal{C}':W_0'\to\dots\to W_t'$ be the restriction of a cyclic permutation of $\pazocal{C}$ to the subword $B'$.  

Since $W_{y+1}$ is $\sigma(43)$-admissible, the subcomputation $W_{y+1}'\to\dots\to W_z'$ can be identified with a reduced computation of a primitive machine satisfying the hypotheses of \Cref{primitive unreduced}.  As such, $W_z$ cannot be $\sigma(43)$-admissible, so that \Cref{Defective (34)(4)(45)} implies $z=t$.  Moreover, the restriction to the $P_iQ_{i,\ell}$-sector satisfies the hypotheses of \Cref{multiply one letter}, so that $t-y-1\leq|W_t'|_a$.

Further, the restriction of the inverse subcomputation $W_y'\to\dots\to W_x'$ to the $Q_{i,\ell}Q_{i,\ell}^{-1}$-sector satisfies the hypotheses of \Cref{one alphabet historical words unreduced}, so that $x=0$ and $|W_y'|_a=|W_x'|_a-2y$, {\frenchspacing i.e. $y\leq|W_0'|_a$}.


\end{proof}

\begin{lemma} \label{strongly defective (23)}

Let $\pazocal{C}:W_0\to\dots\to W_t$ be a minimal computation with strongly defective base $B$.  
Suppose the history of $\pazocal{C}$ contains a letter of the form $\sigma(23)^{\pm1}$ which is neither the first nor the last letter.  Then:

\begin{enumerate}

\item There exists a subword $B'$ of a cyclic permutation of $B$ of the form $R_i^{-1}Q_{i,r}^{-1}Q_{i,r}R_i$

\item Letting $\pazocal{C}':W_0'\to\dots\to W_t'$ be the restriction of a cyclic permutation of $\pazocal{C}$ to the subword $B'$, then $t\leq2\max(|W_0'|_a,|W_t'|_a)+1$.

\end{enumerate}

\end{lemma}

\begin{proof}

First, note that $\sigma(23)$ locks every sector of the standard base other than the input sectors and those of the form $Q_{i,\ell}Q_{i,r}$.  So, since $B$ is strongly defective, $B$ cannot contain a subword of the form $(Q_{i,\ell}Q_{i,r})^{\pm1}$.

Next, note that the standard base is constructed so that there exists a historical $Q_{i,\ell}Q_{i,r}$-sector between every pair of working sectors.  As such, it follows that $B$ must contain a subword $UV$ of the form $Q_{i,\ell}Q_{i,\ell}^{-1}$ or one of the form $Q_{i,r}^{-1}Q_{i,r}$.  Either way, a word with base $B$ which is $\sigma(32)$-admissible cannot be $\sigma(34)$-admissible.

Hence, as in the proof of \Cref{strongly defective (34)}, we may assume $0\leq x<y<z-1<t$ where:

\begin{itemize}

\item $W_x\to\dots\to W_y$ is a maximal subcomputation with step history $(2)$,

\item the transition $W_y\to W_{y+1}$ is given by $\sigma(23)$, and

\item $W_{y+1}\to\dots\to W_z$ is a maximal subcomputation with step history $(3)$.

\end{itemize}

Given the subword $UV$ of $B$, \Cref{locked sectors} implies the history of $\pazocal{C}$ cannot contain a connecting rule of $\textbf{E}_\textbf{S}(2)$.  In particular, $W_i$ must have the same state letters for all $x\leq i\leq y$.  

Moreover, as $\pazocal{C}$ is minimal and $y>x$, $W_x$ and $W_y$ cannot be equal.  As such, there must be a tape word that is altered in the subcomputation between these words, and so by the software of the submachine $B$ must contain a letter of the form $Q_{i,r}^{\pm1}$.  Part (a) of the statement thus follows.

Let $\pazocal{C}':W_0'\to\dots\to W_t'$ be the restriction of a cyclic permutation as in part (b) of the statement.  Then as in the proof of \Cref{strongly defective (34)}, it follows from Lemmas \ref{primitive unreduced} and \ref{one alphabet historical words unreduced} that $x=0$, $z=t$, $y\leq|W_0'|_a$, and $t-y-1\leq|W_t'|_a$.

\end{proof}

\begin{lemma} \label{strongly defective connecting (4)}

Let $\pazocal{C}:W_0\to\dots\to W_t$ be a minimal computation with strongly defective base $B$ and step history $(4)$.  Suppose the history of $\pazocal{C}$ contains a connecting rule that is neither the first nor the last letter.  Then:

\begin{enumerate}[label=(\alph*)]

\item There exists a subword $B'$ of a cyclic permutation of $B$ of the form $Q_{i,r}^{-1}Q_{i,\ell}^{-1}Q_{i,\ell}Q_{i,r}$

\item Letting $\pazocal{C}':W_0'\to\dots\to W_t'$ be the restriction of a cyclic permutation of $\pazocal{C}$ to the subword $B'$, then $t\leq2\max(|W_0'|_a,|W_t'|_a)+1$.

\end{enumerate}

\end{lemma}

\begin{proof}

By hypothesis, there exists $0\leq x<y<z-1<t$ such that:

\begin{itemize}

\item The histories of the subcomputation $W_x\to\dots\to W_y$ and $W_{y+1}\to\dots\to W_z$ contain no connecting rules.

\item The transition $W_y\to W_{y+1}$ is given by a connecting rule.

\end{itemize}

Since $B$ is defective and the connecting rules of $\textbf{E}_\textbf{S}(4)$ lock all sectors of the standard base except for those of the form $P_iQ_{i,\ell}$, the base $B$ must contain a two-letter subword of the form $Q_{i,\ell}^{-1}Q_{i,\ell}$ (and one of the form $P_iP_i^{-1}$).  Part (a) of the statement then follows.  

Let $\pazocal{C}':W_0'\to\dots\to W_t'$ be the restriction of a cyclic permutation of $\pazocal{C}$ as in part (b) of the statement.  Then the restriction of the subcomputation $W_{y+1}'\to\dots\to W_z'$ to the subword $Q_{i,\ell}Q_{i,r}$ of $B'$ satisfies the hypotheses of \Cref{multiply one letter}, so that the tape word of $W_z'$ must be nonempty.  So, $W_z$ is not admissible for a connecting rule, and so $z=t$.  Moreover, $t-y-1\leq|W_t'|_a$.  

A similar argument implies $x=0$ and $y\leq|W_0'|_a$.

\end{proof}

\begin{lemma} \label{strongly defective connecting (2)}

Let $\pazocal{C}:W_0\to\dots\to W_t$ be a minimal computation with strongly defective base $B$ and step history $(2)$.  Suppose the history of $\pazocal{C}$ contains a connecting rule that is neither the first nor the last letter.  Then:

\begin{enumerate}[label=(\alph*)]

\item There exists a subword $B'$ of a cyclic permutation of $B$ of the form $Q_{i,\ell}Q_{i,r}Q_{i,r}^{-1}Q_{i,\ell}^{-1}$

\item Letting $\pazocal{C}':W_0'\to\dots\to W_t'$ be the restriction of a cyclic permutation of $\pazocal{C}$ to the subword $B'$, then $t\leq2\max(|W_0'|_a,|W_t'|_a)+1$.

\end{enumerate}

\end{lemma}

\begin{proof}

By hypothesis, there exists $0\leq x<y<z-1<t$ just as in the proof of \Cref{strongly defective connecting (4)}.

Note that the connecting rules of $\textbf{E}_\textbf{S}(2)$ lock every sector of the standard base other than the input sectors and those of the form $Q_{i,r}R_i$.  So, since $B$ is strongly defective, it cannot contain any subword of the form $(Q_{i,r}R_i)^{\pm1}$.

As in the proof of \Cref{strongly defective (23)}, the subcomputation $W_x\to\dots\to W_y$ must have a tape word altered, and so $B$ must contain a subword of the form $Q_{i,r}Q_{i,r}^{-1}$.  

Part (a) of the statement then follows from the definition of defective, while part (b) follows in just the same way as in the proof of \Cref{strongly defective connecting (4)}.

\end{proof}

\begin{lemma} \label{strongly defective (5)}

Let $\pazocal{C}:W_0\to\dots\to W_t$ be a nonempty minimal computation with strongly defective base $B$ and step history $(5)$.  Then:

\begin{enumerate}[label=(\alph*)]

\item There exists a subword $B'$ of a cyclic permutation of $B$ of the form $P_iQ_{i,\ell}Q_{i,\ell}^{-1}P_i^{-1}$

\item Letting $\pazocal{C}':W_0'\to\dots\to W_t'$ be the restriction of a cyclic permutation of $\pazocal{C}$ to the subword $B'$, then $t\leq12n^2+2n$ for $n=\max(|W_0'|_a,|W_t'|_a)$.

\end{enumerate}

\end{lemma}

\begin{proof}

First, rules of $\textbf{E}_\textbf{S}(5)$ only multiply the sectors $Q_{i,\ell}Q_{i,\ell}^{-1}$ and $(Q_{i,\ell}Q_{i,r})^{\pm1}$ by tape letters.  However, since every rule of $\textbf{E}_\textbf{S}(5)$ locks the $P_iQ_{i,\ell}$- and $Q_{i,r}R_i$-sectors, that $B$ is strongly defective implies it contains no two-letter subword of the form $(Q_{i,\ell}Q_{i,r})^{\pm1}$.

Hence, since $\pazocal{C}$ is nonempty and minimal, $B$ must contain a two-letter subword of the form $Q_{i,\ell}Q_{i,\ell}^{-1}$. $\pazocal{C}$ then satisfies the hypotheses of \Cref{unreduced base quadratic}, so that there exists a two-letter subword $B''$ of $B$ such that for $\pazocal{C}'':W_0''\to\dots\to W_t''$ of $\pazocal{C}$ to $B''$:

\begin{itemize}

\item There exists a rule in the history of $\pazocal{C}$ which multiplies the $B''$-sector by a letter on the left or on the right

\item $t\leq12\max(|W_0''|_a,|W_t''|_a)^2+2\max(|W_0''|_a,|W_t''|_a)$.

\end{itemize}

As above, the first point implies $B''$ is of the form $Q_{i,\ell}Q_{i,\ell}^{-1}$.  But \Cref{locked sectors} and the definition of defective implies there exists a subword $B'$ of a cyclic permutation of $B$ of the form $P_iQ_{i,\ell}Q_{i,\ell}^{-1}P_i^{-1}$ containing this copy of $Q_{i,\ell}Q_{i,\ell}^{-1}$, so that the statement follows.

\end{proof}

\begin{lemma} \label{strongly defective (1)}

Let $\pazocal{C}:W_0\to\dots\to W_t$ be a nonempty minimal computation with strongly defective base $B$ and step history $(1)$.  Then:

\begin{enumerate}[label=(\alph*)]

\item There exists a subword $B'$ of a cyclic permutation of $B$ of the form $R_i^{-1}Q_{i,r}^{-1}Q_{i,r}R_i$

\item Letting $\pazocal{C}':W_0'\to\dots\to W_t'$ be the restriction of a cyclic permutation of $\pazocal{C}$ to the subword $B'$, then $t\leq12n^2+2n$ for $n=\max(|W_0'|_a,|W_t'|_a)$.

\end{enumerate}

\end{lemma}

\begin{proof}

Again, every rule of $\textbf{E}_\textbf{S}(1)$ locks the $P_iQ_{i,\ell}$- and $Q_{i,r}R_i$-sectors, so that $B$ cannot contain a two-letter subword of the form $(Q_{i,\ell}Q_{i,r})^{\pm1}$.  But the rules of $\textbf{E}_\textbf{S}(5)$ only multiplies the tape words of sectors of the form $(Q_{i,\ell}Q_{i,r})^{\pm1}$ and $Q_{i,r}^{-1}Q_{i,r}$, so that $B$ must contain a two-letter subword of the form $Q_{i,r}^{-1}Q_{i,r}$.  The statement then follows as in the proof of \Cref{strongly defective (5)}.

\end{proof}

\begin{lemma} \label{strongly defective (45)}

Let $\pazocal{C}:W_0\to\dots\to W_t$ be a minimal computation with strongly defective base $B$.  
Suppose the history of $\pazocal{C}$ contains a letter of the form $\sigma(45)^{\pm1}$ which is neither the first nor the last letter.  Then:

\begin{enumerate}[label=(\alph*)]

\item There exists a subword $B'$ of a cyclic permutation of $B$ of the form $P_iQ_{i,\ell}Q_{i,\ell}^{-1}P_i^{-1}$

\item Letting $\pazocal{C}':W_0'\to\dots\to W_t'$ be the restriction of a cyclic permutation of $\pazocal{C}$ to the subword $B'$, then $t\leq12n^2+4n+2$ for $n=\max(|W_0'|_a,|W_t'|_a)$.

\end{enumerate}

\end{lemma}

\begin{proof}

As in the proofs of previous statements, perhaps passing to the inverse computation we may assume there exist $0\leq x<y<z-1< t$ such that:

\begin{itemize}

\item $W_x\to\dots\to W_y$ is a maximal subcomputation with step history $(4)$,

\item the transition $W_y\to W_{y+1}$ is given by $\sigma(45)$, and

\item $W_{y+1}\to\dots\to W_z$ is a maximal subcomputation with step history $(5)$.

\end{itemize}

Applying \Cref{strongly defective (5)} to the subcomputation $W_{y+1}\to\dots\to W_z$ yields a $B'$ satisfying part (a) of the statement.  Moreover, letting $\pazocal{C}':W_0'\to\dots\to W_t'$ be the restriction to $B'$, we have $z-y-1\leq12m^2+2m$ for $m=\max(|W_y'|_a,|W_z'|_a)$.

The inverse subcomputation $W_y'\to\dots\to W_x'$ can then be identified with a reduced computation of a primitive machine satisfying the hypotheses of \Cref{primitive unreduced}.  As such, $W_x$ is not $\sigma(45)$- or $\sigma(43)$-admissible, so that $x=0$.  Moreover, $|W_y'|_a\leq|W_0'|_a$ and the restriction to the $P_iQ_{i,\ell}$-sector satisfies the hypotheses of \Cref{multiply one letter}, so that $y\leq|W_0'|_a$.

A similar argument may then be applied to any subcomputation $W_{z+1}'\to\dots\to W_t'$, so that $|W_z'|_a\leq|W_t'|_a$ and $t-z\leq|W_t'|_a+1$.  Thus, $t\leq12m^2+2m+2n+2\leq12n^2+4n+2$.

\end{proof}

The next statement then follows in exactly the same way:

\begin{lemma} \label{strongly defective (12)}

Let $\pazocal{C}:W_0\to\dots\to W_t$ be a minimal computation with strongly defective base $B$.  
Suppose the history of $\pazocal{C}$ contains a letter of the form $\sigma(12)^{\pm1}$ which is neither the first nor the last letter.  Then:

\begin{enumerate}[label=(\alph*)]

\item There exists a subword $B'$ of a cyclic permutation of $B$ of the form $R_i^{-1}Q_{i,r}^{-1}Q_{i,r}R_i$

\item Letting $\pazocal{C}':W_0'\to\dots\to W_t'$ be the restriction of a cyclic permutation of $\pazocal{C}$ to the subword $B'$, then $t\leq12n^2+4n+2$ for $n=\max(|W_0'|_a,|W_t'|_a)$.

\end{enumerate}

\end{lemma}

Thus, we arrive at the desired bound on the length of minimal computations:

\begin{lemma} \label{universal complexity}

Suppose $(W_1,W_2)\in\REACH_{\textbf{E}_\textbf{S}}^{uni}$.  Then there exists a subword $B'$ of a cyclic permutation of the base $B$ of $W_i$ with $\|B'\|\leq N+1$ such that the length of a minimal computation between $W_1$ and $W_2$ is at most $14\TM_\textbf{S}(12n+12N)+12n^2+28n+24$
for $n=\max(|W_1'|_a,|W_2'|_a)$, where $W_i'$ is the admissible subword of a cyclic permutation of $W_i$ with base $B'$. 

\end{lemma}

\begin{proof}

By Lemmas \ref{E reduced circular} and \ref{Defective PQQR} it suffices to assume the base $B$ is strongly defective.  Let $\pazocal{C}:W_1\equiv V_0\to\dots\to V_t\equiv W_2$ be a minimal computation between $W_1$ and $W_2$.  Of course, it suffices to assume $\pazocal{C}$ is nonempty, so that $W_1$ and $W_2$ are not equal.

By Lemmas \ref{strongly defective (34)}-\ref{strongly defective connecting (2)} and \ref{strongly defective (45)}-\ref{strongly defective (12)}, it suffices to assume that any occurrence of a transition or connecting rule in the history $H$ of $\pazocal{C}$ is as the first or last letter.  Hence, as such rules do not alter tape words, it suffices to assume that $H$ contains no such rule and show the bound holds for $14\TM_\textbf{S}(12n+12N)+12n^2+28n+22$.  The step history of $\pazocal{C}$ is thus assumed to be $(j)$ for some $j\in\{1,\dots,5\}$.  Moreover, Lemmas \ref{strongly defective (5)} and \ref{strongly defective (1)} imply it suffices to assume $j\in\{2,3,4\}$.

Suppose $j=2$.  Since $H$ contains no connecting rule, the state letters of $W_1$ and $W_2$ are equal.  As such, there exists a sector whose tape word is different in $W_1$ and $W_2$.  By the definition of the rules of $\textbf{E}_\textbf{S}(2)$, it then follows that $B$ contains a letter of the form $Q_{i,\ell}^{\pm1}$.  The definition of strongly defective then implies $B$ contains a two-letter subword $UV$ of the form $Q_{i,\ell}^{-1}Q_{i,\ell}$ or $Q_{i,\ell}Q_{i,\ell}^{-1}$.  Considering just the rules of $\textbf{E}_\textbf{S}(2)$ making up $H$, we may then identify $\pazocal{C}$ with a reduced computation of a machine satisfying the hypotheses of \Cref{unreduced base quadratic}.  Thus, taking $B'$ to be the two-letter subword given by that lemma, the bound follows.

An identical argument applies if $j=4$, and thus it suffices to assume $j=3$.  By the definition of strongly defective and makeup of the standard base, $B$ must then contain a two-letter subword of the form $Q_{i,\ell}Q_{i,\ell}^{-1}$ or $Q_{i,r}^{-1}Q_{i,r}$.  In either case, the computation $\pazocal{C}$ of $\textbf{E}_\textbf{S}(3)$ can be seen to satisfy the hypotheses of \Cref{unreduced base quadratic}, again yielding the desired bound.

\end{proof}

\begin{remark} \label{rmk-quadratic}

Note that the quadratic term in the statement of \Cref{universal complexity} accounts for the quadratic term in the upper bound of \Cref{main-theorem}.  The difficulty in removing this term lies in the following example:

Suppose there exists a reduced computation $\pazocal{C}:W_0\to\dots\to W_t$ of $\textbf{S}$ with base $Q_i^{-1}Q_i$ such that:

\begin{itemize}

\item The state letters of $W_0$ and $W_t$ are the same

\item No rule of the history $H$ of $\pazocal{C}$ locks the $Q_{i-1}Q_i$-sector

\item The tape words of $W_0$ and $W_t$ are not the same.

\end{itemize}

Consider then the admissible word $V_0$ of $\textbf{E}_\textbf{S}(3)$ with circular base $P_i^{-1}P_iQ_{i,\ell}Q_{i,\ell}^{-1}P_i^{-1}$ such that:

\begin{itemize}

\item The state letters correspond to those of $W_0$

\item The tape word in the $P_i^{-1}P_i$-sector is the same as the tape word of $W_0$

\item The tape word $H_\ell$ in the $Q_{i,\ell}Q_{i,\ell}^{-1}$ sector is the copy of $H$ over the left historical alphabet

\item All other tape words are empty

\end{itemize}

Identifying $H$ with the rules in $\textbf{E}_\textbf{S}(3)$, it then follows that $V_0$ is $H^k$-admissible for all $k$.  The tape word of $V_0\cdot H^k$ in the $Q_{i,\ell}Q_{i,\ell}^{-1}$ is again $H_\ell$, and the tape words $u$ and $u'$ of $V_0$ and $V_0\cdot H^k$, respectively, in the $P_i^{-1}P_i$-sector are not the same.

Without knowing what the operation of $\pazocal{C}$ is, though, we may only bound $\max(\|u\|,\|u'\|)$ from below by $k$.  As such, we may bound $\max(|V_0|_a,|V_0\cdot H^k|_a)$ from below by $k+\|H\|$.  But assuming $H$ is cyclically reduced, $\|H^k\|=k\|H\|$.  Taking $k=\|H\|$ then exhibits the difficulty.

Clearly, some additional control over the computational structure of $\textbf{S}$ is necessary for this purpose.

\end{remark}

\bigskip


\section{Groups Associated to an \texorpdfstring{$S$}--machine and their Diagrams} \label{sec-groups}

\subsection{The groups} \label{sec-associated-groups} \

As in previous literature (for example \cite{O18}, \cite{OS19}, \cite{WCubic}, \cite{WEmb}, \cite{W}), we now associate a finitely presented group to any cyclic $S$-machine $\textbf{S}$. This group, denoted $M(\textbf{S})$, `simulates' the work of $\textbf{S}$ in the precise sense described in \Cref{sec-trapezia}.

Let $\textbf{S}$ be an arbitrary cyclic $S$-machine with hardware $(Y,Q)$, where $Q=\sqcup_{i=0}^s Q_i$ and $Y=\sqcup_{i=1}^{s+1} Y_i$, and software $\Theta=\Theta^+\sqcup\Theta^-$. For notational purposes, set $Q_0=Q_{s+1}$.

For $\theta\in\Theta^+$, Lemma \ref{simplify rules} allows us to assume that $\theta$ takes the form $$\theta=[q_0\to v_{s+1}q_0'u_1, \ q_1\to v_1q_1'u_2, \ \dots, \ q_{s-1}\to v_{s-1}q_{s-1}'u_s, \ q_s\to v_sq_s'u_{s+1}]$$ where $q_i,q_i'\in Q_i$, $u_i$ and $v_i$ are either empty or single letters in the domain $Y_i(\theta)^{\pm1}$, and some of the arrows may take the form $\xrightarrow{\ell}$. 

Define $\Theta^+_*=\{\theta_i: \theta\in\Theta^+,0\leq i\leq s\}$. For notational convenience, set $\theta_{s+1}=\theta_0$ for all $\theta\in\Theta^+$.

The group $M(\textbf{S})$ is then defined by taking the (finite) generating set $\pazocal{X}=Y\cup Q\cup \Theta^+_*$ and subjecting it to the (finite number of) relations:
\begin{equation}\notag\label{thetaRels}\begin{matrix}\theta_iy& = & y\theta_i &&&&\text{for all }&i\in \{1,\ldots, s\}, &y\in Y_i(\theta)\\ q_{i}\theta_{i+1}&=&\theta_{i}a_i(q_i')b_i&&&&\text{for all }&i\in \{1,\ldots, s\},  &\theta\in \Theta^+ \end{matrix}\end{equation}

As in the language of computations of $S$-machines, letters from $Q^{\pm1}$ are called \textit{$q$-letters} and those from $Y^{\pm1}$ are called \textit{$a$-letters}. Additionally, those from $(\Theta^+_*)^{\pm1}=\Theta^+_*\sqcup\Theta^-_*$ are called \textit{$\theta$-letters}. The relations of the form $q_i\theta_{i+1}=\theta_iv_iq_i'u_{i+1}$ are called \textit{$(\theta,q)$-relations}, while those of the form $\theta_ia=a\theta_i$ are called \textit{$(\theta,a)$-relations}.

Note that the number of $a$-letters in any part of $\theta$, and so in any defining relation of $M(\textbf{S})$, is at most two.  Further, note that if $\theta$ locks the $i$-th sector, then $Y_i(\theta)=\emptyset$ so that each $\theta_j$ has no relation with the elements of $Y_i$.


\medskip


\subsection{Bands and annuli} \label{sec-bands-annuli} \

The arguments presented in the forthcoming sections rely heavily on van Kampen (circular) and Schupp (annular) diagrams over (the presentation of) the group $M(\textbf{S})$ introduced in \Cref{sec-associated-groups}.  It is assumed that the reader is well acquainted with this notion; for reference, see \cite{Lyndon-Schupp} and \cite{O}.

To present these arguments efficiently, we first differentiate between the types of edges and cells that arise in such diagrams.  For simplicity, we adopt the convention that the contour of any diagram, subdiagram, or cell is traced in the counterclockwise direction.

If the label $\lab(\textbf{e})$ of an edge $\textbf{e}$ in such a diagram is a $q$-letter, then $\textbf{e}$ is called a \textit{$q$-edge}. Similarly, an edge labelled by an $a$-letter is called an \textit{$a$-edge} and one labelled by a $\theta$-letter is a \textit{$\theta$-edge}. 

For a path \textbf{p} in $\Delta$, the (combinatorial) length of $\textbf{p}$ is denoted $\|\textbf{p}\|$. Further, the path's \textit{$a$-length} $|\textbf{p}|_a$ is the number of $a$-edges in the path. The path's \textit{$\theta$-length} and \textit{$q$-length}, denoted $|\textbf{p}|_{\theta}$ and $|\textbf{p}|_q$, respectively, are defined similarly.  A cell whose contour label corresponds to a $(\theta,q)$-relation {\frenchspacing(resp. a $(\theta,a)$-relation)} is called a \textit{$(\theta,q)$-cell} {\frenchspacing(resp. a \textit{$(\theta,a)$-cell})}.

In the general setting of a reduced diagram $\Delta$ over an arbitrary presentation $\pazocal{P}=\gen{X\mid\pazocal{R}}$, let $\pazocal{Z}\subseteq X$.  An edge of $\Delta$ is called a $\pazocal{Z}$-edge if its label is in $\pazocal{Z}^{\pm1}$.  For $m\geq1$, a sequence of (distinct) cells $\pazocal{B}=(\Pi_1,\dots,\Pi_m)$ in $\Delta$ is called a \textit{$\pazocal{Z}$-band} of length $m$ if:

\begin{itemize}

\item every two consecutive cells $\Pi_i$ and $\Pi_{i+1}$ have a common boundary $\pazocal{Z}$-edge $\textbf{e}_i$, 

\item $\textbf{e}_{i-1}^{-1}$ and $\textbf{e}_i$ are the only two $\pazocal{Z}$-edges of $\partial\pi_i$, and

\item $\text{Lab}(\textbf{e}_i)$ are either all positive or all negative.

\end{itemize}

A $\pazocal{Z}$-band is \textit{maximal} if it is not a subsequence of another such band. Extending the definition so that a $\pazocal{Z}$-edge is a $\pazocal{Z}$-band of length $0$, it follows that that every $\pazocal{Z}$-edge in $\Delta$ is contained in a maximal $\pazocal{Z}$-band.

In a $\pazocal{Z}$-band $\pazocal{B}$ of length $m\geq1$ made up of the cells $(\Pi_1,\dots,\Pi_m)$, using only edges from the contours of $\Pi_1,\dots,\Pi_m$, there exists a closed path $\textbf{e}_0^{-1}\textbf{q}_1\textbf{e}_m\textbf{q}_2^{-1}$ such that $\textbf{q}_1$ and $\textbf{q}_2$ are simple (perhaps closed) paths. In this case, $\textbf{q}_1$ is called the \textit{bottom} of $\pazocal{B}$, denoted $\textbf{bot}(\pazocal{B})$, while $\textbf{q}_2$ is called the \textit{top} of $\pazocal{B}$ and denoted $\textbf{top}(\pazocal{B})$. When $\textbf{q}_1$ and $\textbf{q}_2$ need not be distinguished, they are called the \textit{sides} of the band.

\begin{figure}[H]
\centering
\begin{subfigure}[b]{0.48\textwidth}
\centering
\raisebox{0.5in}{\includegraphics[scale=1.25]{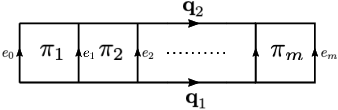}}
\caption{Non-annular $\pazocal{Z}$-band of length $m$}
\end{subfigure}\hfill
\begin{subfigure}[b]{0.48\textwidth}
\centering
\includegraphics[scale=1.25]{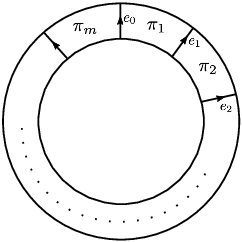}
\caption{Annular $\pazocal{Z}$-band of length $m$}
\end{subfigure}
\caption{ \ }
\end{figure}

If $\textbf{e}_0=\textbf{e}_m$ in a $\pazocal{Z}$-band $\pazocal{B}$ of length $m\geq1$, then $\pazocal{B}$ is called a \textit{$\pazocal{Z}$-annulus}.  Further, if either side bounds a contractible loop in $\Delta$, then $\pazocal{B}$ is called a \textit{contractible} $\pazocal{Z}$-annulus.

If $\pazocal{B}$ is a non-annular $\pazocal{Z}$-band of length $m\geq1$, then $\textbf{e}_0^{-1}\textbf{q}_1\textbf{e}_m\textbf{q}_2^{-1}$ is called the \textit{standard factorization} of the contour of $\pazocal{B}$. If either $(\textbf{e}_0^{-1}\textbf{q}_1\textbf{e}_m)^{\pm1}$ or $(\textbf{e}_m\textbf{q}_2^{-1}\textbf{e}_0^{-1})^{\pm1}$ is a subpath of $\partial\Delta$, then $\pazocal{B}$ is called a \textit{rim $\pazocal{Z}$-band}.

A $\pazocal{Z}_1$-band and a $\pazocal{Z}_2$-band \textit{cross} if they have a common cell and $\pazocal{Z}_1\cap\pazocal{Z}_2=\emptyset$.

In particular, in a reduced diagram over $M(\textbf{S})$, there exist \textit{$q$-bands} corresponding to bands arising from $\pazocal{Z}=Q_i$ for some $i$, where every cell is a $(\theta,q)$-cell. Similarly, there exist \textit{$\theta$-bands} for $\theta\in\Theta^+$ and \textit{$a$-bands} for $a\in Y$. Note that an $a$-band consists entirely of $(\theta,a)$-cells.

Note that by definition, distinct maximal $q$-bands cannot intersect.  In just the same way, distinct maximal $\theta$-bands cannot intersect, nor can distinct maximal $a$-bands.

Given an $a$-band $\pazocal{B}$, the makeup of the group relations dictates that the defining $a$-edges $\textbf{e}_0,\dots,\textbf{e}_m$ are labelled identically. Similarly, the $\theta$-edges of a $\theta$-band correspond to the same rule; however, the index of two such $\theta$-edges may differ.

If a maximal $a$-band contains a cell with an $a$-edge that is also on the contour of a $(\theta,q)$-cell, then the $a$-band is said to \textit{end} (or \textit{start}) on that $(\theta,q)$-cell and the corresponding $a$-edge is said to be the \textit{end} (or \textit{start}) of the band.  In the analogous way, a maximal $a$-band (or maximal $q$-band or maximal $\theta$-band) may end on the contour of the diagram.  

Note that a maximal band that ends in one part of the diagram must also end in another part.  In an annular diagram, a band that has one end on each boundary component is called \textit{radial}.

The natural projection of the label of the top (or bottom) of a $q$-band onto $F(\Theta^+)$ is called the \textit{history} of the band. The natural projection (without reduction) of the top (or bottom) of a $\theta$-band onto the alphabet $\{Q_0,\dots,Q_s\}$ is called the \textit{base} of the band.

Suppose the sequence of cells $(\pi_0,\pi_1,\dots,\pi_m)$ comprises a $\theta$-band $\pazocal{T}$ and $(\gamma_0,\gamma_1,\dots,\gamma_\ell)$ a $q$-band $\pazocal{Q}$ such that $\pi_0=\gamma_0$, $\pi_m=\gamma_\ell$, and no other cells are shared. Suppose further that $\partial\pi_0$ and $\partial\pi_m$ both contain $\theta$-edges on the outer countour of the annulus bounded by the two bands. Then the union of these two bands is called a \textit{$(\theta,q)$-annulus} (see \Cref{fig-annular-band}). 

A \textit{$(\theta,a)$-annulus} is defined similarly.

\begin{figure}[H] 
\centering
\includegraphics[scale=1.1]{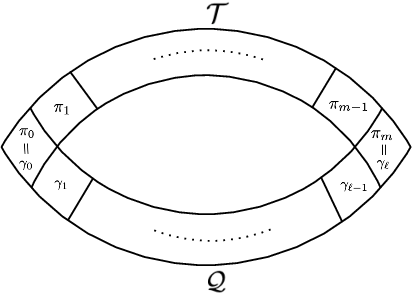} 
\caption{$(\theta,q)$-annulus with defining $\theta$-band $\pazocal{T}$ and $q$-band $\pazocal{Q}$} \label{fig-annular-band}
\end{figure}

The following statement is proved in a more general setting in \cite{O97}:

\begin{lemma}[Lemma 6.1 of \cite{O97}] \label{M(S) annuli}

A reduced circular diagram over $M(\textbf{S})$ contains no:

\begin{enumerate}[label=({\arabic*})]

\item $(\theta,q)$-annuli

\item $(\theta,a)$-annuli

\item $a$-annuli

\item $q$-annuli

\item $\theta$-annuli

\end{enumerate}

\end{lemma}

As a result, in a reduced circular diagram $\Delta$ over $M(\textbf{S})$, if a maximal $\theta$-band and a maximal $q$-band (respectively $a$-band) cross, then their intersection is exactly one $(\theta,q)$-cell (respectively $(\theta,a)$-cell). Further, every maximal $\theta$-band and maximal $q$-band ends on $\partial\Delta$ in two places.

\begin{lemma}\label{Dehn bound}
For any cyclic recognizing S-machine $\mathbf{S}$, $\text{Dehn}_{M(\mathbf{S})}(n)\preceq n^3$.
\end{lemma}

\begin{proof}

By van Kampen's Lemma \cite{v-K}, a word $w\in(\pazocal{X}\cup\pazocal{X}^{-1})^\ast$ represents the identity in $M(\textbf{S})$ if and only if there exists a reduced circular diagram $\Delta$ over $M(\textbf{S})$ with $\lab(\partial\Delta)\equiv w$.  It thus suffices to show that $\text{Area}(\Delta)$ is bounded above by some cubic function of $\|w\|=\|\partial\Delta\|$.

By \Cref{M(S) annuli}, $\Delta$ contains at most $\frac{1}{2}\|\partial\Delta\|$ maximal $\theta$-bands and at most $\frac{1}{2}\|\partial\Delta\|$ maximal $q$-bands, with any pair of these bands intersecting at most once.  As such, the number of $(\theta,q)$-cells in $\Delta$ is at most $\frac{1}{4}\|\partial\Delta\|^2$.

Moreover, \Cref{M(S) annuli} implies any maximal $a$-band in $\Delta$ must end in two places, with these ends on a $(\theta,q)$-cell or on $\partial\Delta$.  As any $(\theta,q)$-relation contains at most two $a$-letters, there are at most $\frac{1}{2}\|\partial\Delta\|^2+\|\partial\Delta\|$ places for such an $a$-band to end, and so there are at most $\frac{1}{4}\|\partial\Delta\|^2+\frac{1}{2}\|\partial\Delta\|$ maximal $a$-bands in $\Delta$.

But any maximal $a$-band and maximal $\theta$-band intersect at most once, so that the number of $(\theta,a)$-cells in $\Delta$ is at most $\frac{1}{8}\|\partial\Delta\|^3+\frac{1}{4}\|\partial\Delta\|^2$.  Hence, $\text{Area}(\Delta)\leq\frac{1}{8}\|\partial\Delta\|^3+\frac{1}{2}\|\partial\Delta\|^2$.

\end{proof}


The next statement follows immediately from the ``sewing and detaching procedure" described in \cite{Lyndon-Schupp} (see pp. 150-151) or the concept of ``$0$-cells" introduced in \cite{O}:

\begin{lemma}\label{exciseTwoPaths}
Let $\gamma_1$ and $\gamma_2$ be two disjoint simple loops in a circular {\frenchspacing(resp. annular)} diagram $\Delta$ over some group presentation $\pazocal{P}$.  Suppose $\gamma_1$ and $\gamma_2$ have freely conjugate labels and bound an annulus $\Gamma$ that contains at least one (positive) cell.  Then there exists a circular {\frenchspacing(resp. annular)} diagram $\Delta'$ over $\pazocal{P}$ with the same boundary label(s) as $\Delta$ and strictly smaller area.
    
\end{lemma}

\begin{proof}

Using the language of $0$-cells, van Kampen's Lemma produces an annular diagram $\Gamma'$ with the same boundary labels as $\Gamma$ and consisting entirely of $0$-cells.  As such, the statement follows by letting $\Delta'$ be the diagram obtained from $\Delta$ by excising $\Gamma$ and pasting $\Gamma'$ in its place.

\end{proof}

\medskip


\subsection{Trapezia} \label{sec-trapezia} \

Let $\Delta$ be a reduced circular diagram over $M(\textbf{S})$ whose contour is of the form $\textbf{p}_1^{-1}\textbf{q}_1\textbf{p}_2\textbf{q}_2^{-1}$, where $\textbf{p}_1$ and $\textbf{p}_2$ are sides of $q$-bands and $\textbf{q}_1$ and $\textbf{q}_2$ are maximal parts of the sides of $\theta$-bands whose labels start and end with $q$-letters. Then $\Delta$ is called a \textit{trapezium}.

In this case, $\textbf{q}_1$ and $\textbf{q}_2$ are called the \textit{bottom} and \textit{top} of the trapezium, respectively, while $\textbf{p}_1$ and $\textbf{p}_2$ are the \textit{left} and \textit{right} sides. 


The \textit{history} of the trapezium is the history of the rim $q$-band with side $\textbf{p}_2$ and the length of this history is the trapezium's \textit{height}. The base of $\text{Lab}(\textbf{q}_1)$ is called the \textit{base} of the trapezium.

It's easy to see from this definition that a $\theta$-band $\pazocal{T}$ whose first and last cells are $(\theta,q)$-cells can be viewed as a trapezium of height 1 as long as its top and bottom start and end with $q$-edges. We extend this to all such $\theta$-bands by merely disregarding any $a$-edges of the top and bottom that precede the first $q$-edge or follow the final $q$-edge. The paths formed by disregarding these edges are called the \textit{trimmed} top and bottom of the band and are denoted $\textbf{ttop}(\pazocal{T})$ and $\textbf{tbot}(\pazocal{T})$.

\begin{figure}[H]
\centering
\includegraphics[scale=1.8]{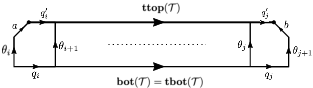}
\caption{$\theta$-band $\pazocal{T}$ with trimmed top}
\end{figure}

By Lemma \ref{M(S) annuli}, any trapezium $\Delta$ of height $h\geq1$ can be decomposed into $\theta$-bands $\pazocal{T}_1,\dots,\pazocal{T}_h$ connecting the left and right sides of the trapezium, with $\textbf{bot}(\pazocal{T}_1)$ and $\textbf{top}(\pazocal{T}_h)$ making up the bottom and top of $\Delta$, respectively.  Moreover, the first and last cells of each $\pazocal{T}_i$ are $(\theta,q)$-cells, so that $\textbf{ttop}(\pazocal{T}_i)=\textbf{tbot}(\pazocal{T}_{i+1})$ for all $1\leq i\leq h-1$.  In this case, the bands $\pazocal{T}_1,\dots,\pazocal{T}_h$ are said to be \textit{enumerated from bottom to top}.

The following two statements are proved in more generality in \cite{WMal} and exemplify how the group $M(\textbf{S})$ simulates the work of the $S$-machine $\textbf{S}$:

\begin{lemma} \label{trapezia are computations}

Let $\Delta$ be a trapezium with history $H\equiv\theta_1\dots\theta_h$ for $h\geq1$ and maximal $\theta$-bands $\pazocal{T}_1,\dots,\pazocal{T}_h$ enumerated from bottom to top. Letting $W_{j-1}\equiv\lab(\textbf{tbot}(\pazocal{T}_j))$ for $1\leq j\leq h$ and letting $W_h\equiv\lab(\textbf{ttop}(\pazocal{T}_h))$, then there exists a reduced computation $W_0\to\dots\to W_h$ of $\textbf{S}$ with history $H$.

\end{lemma}

\begin{lemma} \label{computations are trapezia}

For any non-empty reduced computation $W_0\to\dots\to W_t$ of $\textbf{S}$ with history $H$, there exists a trapezium $\Delta$ such that: 

\begin{enumerate} [label=(\alph*)]

\item $\lab(\textbf{tbot}(\Delta))\equiv W_0$

\item $\lab(\textbf{ttop}(\Delta))\equiv W_t$

\item The history of $\Delta$ is  $H$


\end{enumerate}

\end{lemma}

\begin{lemma}\label{history as powers}
Let $\Delta$ be a trapezium with base $B$, history $H\equiv H_0^k$ for $k\geq1$, and maximal $\theta$-bands $\pazocal{T}_1,\dots,\pazocal{T}_h$ enumerated from bottom to top.  Then for $L=\|H_0\|$, either:

\begin{enumerate}

\item $\lab(\textbf{tbot}(\pazocal{T}_i))\equiv \lab(\textbf{tbot}(\pazocal{T}_{i+L}))$ for all $i$, or 

\item There exists a two-letter subword $UV$ of $B$ such that for $w_0,w_h$ the tape words of $\lab(\textbf{tbot}(\Delta)),\lab(\textbf{ttop}(\Delta))$, respectively, in the $UV$-sector, $k\leq \|w_0\|+\|w_h\|+2L$.

\end{enumerate}

\end{lemma}
\begin{proof} For $W=\lab(\textbf{tbot}(\Delta))$, $\lab(\textbf{tbot}(\pazocal{T}_i))\equiv \lab(\textbf{tbot}(\pazocal{T}_{i+L}))$ if and only if $W\cdot H_0 \equiv W$. In particular, if (1) does not hold, then there must exist a subword $UV$ of $B$ such that such that the corresponding tape word $w_0$ in $W$ does not equal the corresponding tape word $w_0'$ in $W\cdot H_0$. By \Cref{simplify rules}, there exist tape words $v,u$ with $\|v\|,\|u\|\leq \|H_0\|$ such that $w_0'=vw_0u\neq w_0$. 
But then, the corresponding tape word $w_h$ in $W\cdot H_0^k\equiv\lab(\textbf{ttop}(\Delta))$ equals $v^k w_0 u^k$. Hence \cite[Lemma 8.1]{OSconj} implies $\|w_h\|\geq k-(||v||+||u||+||w_0||)$.

\end{proof}

\medskip


\subsection{Modified length function}\label{subsec-modified-length} \

To assist with the main argument, we now modify the length function on words over the generators of $M(\textbf{S})$. This is done in the same way as in \cite{O18}, \cite{OS19}, \cite{WEmb}, and \cite{W}. 

The standard length of a word/path will henceforth be referred to as its \textit{combinatorial length} and the modified length simply as its \textit{length}.

Define a word consisting of no $q$-letters, one $\theta$-letter, and at most one $a$-letter as a \textit{$(\theta,a)$-syllable}. Then, define the length of:

\begin{itemize}

\item any $q$-letter as 1

\item any $(\theta,a)$-syllable as 1

\item any $a$-letter as $\delta\equiv\frac{1}{2N+2}$

\end{itemize}


For a word $w$ over the generating set $\pazocal{X}$ of $M(\textbf{S})$, define a \textit{decomposition} of $w$ as a factorization of $w$ into a product of letters and $(\theta,a)$-syllables. The length of a decomposition of $w$ is then defined to be the sum of the lengths of the factors.  Finally, the length of $w$, denoted $|w|$, is defined to be the minimum of the lengths of its decompositions.  

The length of a path in a diagram over $M(\textbf{S})$ is defined to be the length of its label.  Similarly, given an element $g\in M(\textbf{S})$, the \textit{length} $|g|$ is defined to be the minimal length of any word representing $g$. 

The following statement gives some basic properties of the length function: 

\begin{lemma}[Lemma 6.2 of \cite{OS19}] \label{lengths}

Let \textbf{s} be a path in a diagram $\Delta$ over $M(\textbf{S})$ consisting of $c$ $\theta$-edges and $d$ $a$-edges. Then:

\begin{enumerate}[label=({\alph*})]

\item $|\textbf{s}|\geq\max(c,c+(d-c)\delta)$

\item $|\textbf{s}|=c$ if $\textbf{s}$ is the top or a bottom of a $q$-band

\item For any product $\textbf{s}=\textbf{s}_1\textbf{s}_2$ of two paths in a diagram,
$$|\textbf{s}_1|+|\textbf{s}_2|-\delta\leq|\textbf{s}|\leq|\textbf{s}_1|+|\textbf{s}_2|$$

\item Let $\pazocal{T}$ be a $\theta$-band with base of length $l_b$. If $\textbf{top}(\pazocal{T})$ (or $\textbf{bot}(\pazocal{T})$) has $l_a$ $a$-edges, then the number of cells in $\pazocal{T}$ is between $l_a-l_b$ and $l_a+3l_b$.

\end{enumerate}

\end{lemma}

\begin{lemma}\label{minBoundaryLengths} 

Let $\Delta$ be a reduced diagram over $M(\textbf{S})$.  Suppose $\Delta$ contains a rim $q$-band with two ends on the component $\textbf{p}$ of $\partial\Delta$.  Then there exists a diagram $\Delta'$ over $M(\textbf{S})$ with a corresponding boundary component $\textbf{p}'$ such that:

\begin{enumerate}[label=(\alph*)]

\item $\lab(\textbf{p}')$ and $\lab(\textbf{p})$ represent the same element of $M(\textbf{S})$

\item $|\textbf{p}|-|\textbf{p}'|\geq2$

\item For any other component of $\partial\Delta$, there exists a corresponding component of $\partial\Delta'$ with identical label.

\end{enumerate}

\end{lemma}

\begin{proof}

Form $\Delta'$ by simply removing the rim $q$-band, cutting along its side.  Using $0$-refinement, (a) and (c) immediately hold.

Then $\textbf{p}'$ is obtained from $\textbf{p}$ by removing two $q$-edges and replacing one side of the $q$-band with the other.  Since the removed subpath begins and ends with $q$-edges, \Cref{lengths} implies $|\textbf{p}|-|\textbf{p}'|\geq2$.

\end{proof}

\begin{lemma}\label{minBoundaryLengths-theta}

Let $\Delta$ be a reduced diagram over $M(\textbf{S})$.  Suppose $\Delta$ contains a rim $\theta$-band $\pazocal{T}$ with two ends on the component $\textbf{p}$ of $\partial\Delta$.  If the base of $\pazocal{T}$ is at most $N$, then there exists a diagram $\Delta'$ over $M(\textbf{S})$ with a corresponding boundary component $\textbf{p}'$ such that:

\begin{enumerate}[label=(\alph*)]

\item $\lab(\textbf{p}')$ and $\lab(\textbf{p})$ represent the same element of $M(\textbf{S})$

\item $|\textbf{p}|-|\textbf{p}'|\geq1$

\item For any other component of $\partial\Delta$, there exists a corresponding component of $\partial\Delta'$ with identical label.

\end{enumerate}

\end{lemma}

\begin{proof}

As in the proof of \Cref{minBoundaryLengths}, we form $\Delta'$ by cutting along the side of the band in question.  This time, $\textbf{p}'$ is obtained from $\textbf{p}$ by removing two $\theta$-edges and replacing one side of $\pazocal{T}$ with the other.  By the definition of the relations, this introduces at most two new $a$-edges for each $(\theta,q)$-cell in $\pazocal{T}$, and so at most $2N$ $a$-edges.  As such, \Cref{lengths} implies $|\textbf{p}|-|\textbf{p}'|\geq2-(2N+2)\delta=1$.

\end{proof}



{

\section{Proof of Theorem \ref{main-theorem}}

Our goal in this section is to show that for a fixed recognizing $S$-machine $\textbf{S}$, the finitely presented group $M(\textbf{E}_\textbf{S})$ is sufficient for the proof of \Cref{main-theorem}.  For ease of notation, given an arbitrary cyclic $S$-machine $\textbf{M}$, we denote the function $c_{M(\textbf{M}),\pazocal{X}}$ (see \Cref{def-CL}) simply by $c_\textbf{M}$.

The next two statements are satisfied for the groups over any such machine $\textbf{M}$.

\begin{lemma} \label{CL-no-q-cells}

Let $\Delta$ be an annular diagram over $M(\textbf{M})$ and denote the boundary labels of $\Delta$ by $u$ and $v$.  If $\Delta$ contains no $q$-bands, then $c_\textbf{M}(u,v)\leq(2N+2)(|u|+|v|)$.

\end{lemma}

\begin{proof}

It follows quickly from \Cref{M(S) annuli} that the $\theta$-letters and $a$-letters generate a subgroup of $M(\textbf{M})$ which is isomorphic to a right-angled Artin group (RAAG).  But then a result of Servatius \cite{Servatius} shows $c_\textbf{M}(u,v)\leq\|u\|+\|v\|$, so that the statement follows from the definition of the length.

\end{proof}

\begin{lemma} \label{CL-no-boundary-q}

Let $u$ and $v$ be words over the generators of $M(\textbf{M})$ such that $|u|_q=|v|_q=0$.  Then for $s=s(\textbf{M})$ the size of the standard base of $\textbf{M}$, $c_{\textbf{M}}(u,v)\leq(4N+s+4)(|u|+|v|)+s$.

\end{lemma}

\begin{proof}

We may clearly assume $u$ and $v$ represent conjugate elements of $M(\textbf{M})$.  As such, van Kampen's Lemma implies the existence of a reduced annular diagram $\Delta$ over $M(\textbf{M})$ with boundary labels $u$ and $v$.  Assume $\Delta$ has minimal area among all such reduced annular diagrams.

By \Cref{CL-no-q-cells}, it suffices to assume there exists a maximal $q$-band $\pazocal{Q}$ in $\Delta$.  \Cref{M(S) annuli}(4) then implies $\pazocal{Q}$ and indeed all maximal $q$-bands in $\Delta$ are non-contractible $q$-annuli.  Any (positive) cell of such a $q$-annulus marks the intersection of the annulus with a maximal $\theta$-band, which by  \Cref{M(S) annuli}(1) must be a radial $\theta$-band intersecting each $q$-annulus exactly once.

Let $\pazocal{T}$ be such a radial $\theta$-band.  Then cutting along a side of $\pazocal{T}$, each $q$-annulus can be viewed as a trapezium with some fixed history.  As such, if two $q$-annuli correspond to the same part of the state letters, then they have sides which are labelled identically.  As $\Delta$ is reduced, these sides must be disjoint.  But then \Cref{exciseTwoPaths} produces an annular diagram contradicting the minimality of $\text{Area}(\Delta)$.   Hence, the number of $q$-annuli in $\Delta$ is at most $s$.  

As a result, $c_\textbf{M}(u,v)\leq\|\textbf{bot}(\pazocal{T})\|+\|u\|+\|v\|\leq s+|\textbf{bot}(\pazocal{T})|_a+\delta^{-1}(|u|+|v|)$.  

Note that the makeup of the relations dictate that the sides of an $a$-band are labelled identically.  So, the existence of a non-contractible annular $a$-band would contradict the minimality of $\text{Area}(\Delta)$ through \Cref{exciseTwoPaths} in just the same way as above.  \Cref{M(S) annuli}(3) then implies every maximal $a$-band ends on $\partial\Delta$ or on a $(\theta,q)$-cell.

As every $(\theta,q)$-cell marks the intersection of a $q$-annulus with a radial $\theta$-band, there are at most $s(|u|+|v|)$ $(\theta,q)$-cells in $\Delta$.  So, there are at most $(\delta^{-1}+s)(|u|+|v|)$ maximal $a$-bands in $\Delta$.

But each $a$-edge of $\textbf{bot}(\pazocal{T})$ belongs to a distinct $a$-band, so that $|\textbf{bot}(\pazocal{T})|_a\leq(\delta^{-1}+s)(|u|+|v|)$.

\end{proof}

\begin{lemma} \label{CL upper bound}

$\CL_{\textbf{E}_\textbf{S}}(n)\preceq\TM_\textbf{S}(n)+n^2$

\end{lemma}

\begin{proof}

Let $u$ and $v$ be words the generators of $M(\textbf{E}_\textbf{S})$ and set $n=|u|+|v|$.  

Since $\delta\|\cdot\|\leq|\cdot|\leq\|\cdot\|$, it suffices to find $C>0$ independent of the choice of $u,v$ such that 
\begin{equation} \label{suffices-1}
c_{\textbf{E}_\textbf{S}}(u,v)\leq C\TM_\textbf{S}(Cn+C)+Cn^2+Cn+C
\end{equation}

For this purpose, it suffices to assume $|u|$ and $|v|$ are minimal for all words representing the respective elements of $M(\textbf{E}_\textbf{S})$.

Further, as in the proof of \Cref{CL-no-boundary-q}, we may assume there exists a reduced annular diagram $\Delta$ with boundary labels $u$ and $v$, and moreover that $\Delta$ is taken to have minimal area with respect to this condition.  Given the minimality of $|u|$ and $|v|$, Lemma \ref{minBoundaryLengths} implies $\Delta$ contains no rim $q$-bands.

If $|u|_q=|v|_q=0$, then (\ref{suffices-1}) follows from \Cref{CL-no-boundary-q} by taking $C=2\delta^{-1}+N$.  Hence, we may assume there exists a maximal $q$-band $\pazocal{Q}$ which has two ends on $\partial\Delta$.  \Cref{minBoundaryLengths-theta} then implies that $\pazocal{Q}$, and indeed every maximal $q$-band of $\Delta$, must be a radial $q$-band.

Let $\ell(\pazocal{Q}')$ be the length of some maximal $q$-band $\pazocal{Q}'$ in $\Delta$.  \Cref{lengths}(b) and the makeup of the $(\theta,q)$-relations then imply $\ell(\pazocal{Q}')=|\textbf{bot}(\pazocal{Q}')|\geq\frac{1}{3}\|\textbf{bot}(\pazocal{Q}')\|$, so that $$c_{\textbf{E}_\textbf{S}}(u,v)\leq\|\textbf{bot}(\pazocal{Q}')\|+\|u\|+\|v\|\leq2\ell(\pazocal{Q}')+\delta^{-1}n$$
Thus, in place of (\ref{suffices-1}) it suffices to find  $C_\ell>0$ independent of the choice of $u,v$ such that
\begin{equation} \label{suffices-2}
\ell(\pazocal{Q}')\leq C_\ell\TM_\textbf{S}(C_\ell n+C_\ell)+C_\ell n^2+C_\ell n+C_\ell
\end{equation}
for some $q$-band $\pazocal{Q}'$ in $\Delta$.  We now proceed in two cases:

\textbf{1.} Suppose $\Delta$ contains a $\theta$-annulus.

By \Cref{M(S) annuli}(5), the $\theta$-annuli of $\Delta$ are non-contractible, bounding an annular subdiagram $\Delta_2$.  Using $0$-refinement, we may then assume $\Delta\setminus\Delta_2$ consists of two components, $\Delta_1$ and $\Delta_3$, each of which is a reduced annular diagram sharing one boundary component with $\Delta_2$ and one with $\Delta$ (see \Cref{fig-theta-annulus}).  Note that by construction, every maximal $\theta$-band of $\Delta_1$ and $\Delta_3$ is a maximal $\theta$-band of $\Delta$ which ends twice on the same component of $\partial\Delta$.

\begin{figure}[H]
\centering
\includegraphics[scale=1.4]{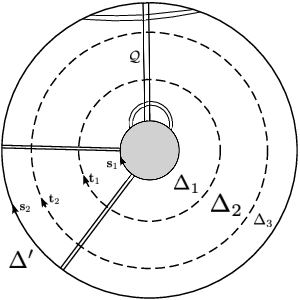}
\caption{} \label{fig-theta-annulus}
\end{figure}

There then exists a decomposition of $\pazocal{Q}=\pazocal{Q}_1\pazocal{Q}_2\pazocal{Q}_3$ such that each $\pazocal{Q}_i$ is a radial $q$-band in $\Delta_i$.  So, letting $\ell_i$ be the length of $\pazocal{Q}_i$, $$\ell(\pazocal{Q})=\ell_1+\ell_2+\ell_3\leq\ell_2+|u|_\theta+|v|_\theta\leq\ell_2+n$$
Let $\Gamma_2$ be the circular diagram obtained from $\Delta_2$ by cutting along the side of $\pazocal{Q}_2$ and pasting another copy of $\pazocal{Q}_2$ to the side from which $\pazocal{Q}_2$ is removed.  Then $\Gamma_2$ is a trapezium with circular base $B$ and height $\ell_2$.

By \Cref{trapezia are computations}, there then exists a reduced computation $\pazocal{C}_2$ of $\textbf{E}_\textbf{S}$ between $W_b\equiv\lab(\textbf{tbot}(\Gamma_2))$ and $W_t\equiv\lab(\textbf{ttop}(\Gamma_2))$, so that $(W_b,W_t)\in\REACH^{uni}_{\textbf{E}_\textbf{S}}$.  

There thus exists a reduced computation $\bar{\pazocal{C}}_2$ between $W_b$ and $W_t$ which satisfies the length bound given by \Cref{universal complexity}. Applying \Cref{computations are trapezia} then produces a trapezium $\bar{\Gamma}_2$ corresponding to $\bar{\pazocal{C}}_2$.  Identifying the side $q$-bands of $\bar{\Gamma}_2$ yields a reduced annular diagram $\bar{\Delta}_2$ with boundary labels equivalent to those of $\Delta_2$.  But then we may excise $\Delta_2$ from $\Delta$ and replace it with $\bar{\Delta}_2$.

Hence, it suffices to assume $\pazocal{C}_2$ satisfies the length bound given by \Cref{universal complexity}.  In particular, there exists a subword $B'$ of $B$ with $\|B'\|\leq N+1$ such that for $W_b',W_t'$ the admissible subwords of $W_b,W_t$ with base $B'$, we have
$$\ell_2\leq14\TM_\textbf{S}(12m+12N)+12m^2+28m+24$$
for $m=\max(|W_b'|_a,|W_t'|_a)$.

Let $\Delta'$ be the subdiagram of $\Delta$ bounded by the (radial) $q$-bands corresponding to $B'$ and let $\Delta_i'$ be the maximal (perhaps annular) subdiagram of $\Delta'$ contained in $\Delta_i$.  Then for $i=1,3$, there exist (perhaps closed) maximal subpaths $\textbf{s}_i$ and $\textbf{t}_i$ of $\partial\Delta_i'$ that are also boundary paths of $\Delta$ and $\Delta_2$, respectively.  As such, \Cref{M(S) annuli}(1) implies the total combined number of $(\theta,q)$-cells in $\Delta_1'$ and $\Delta_3'$ is at most $(N+1)(|u|_\theta+|v|_\theta)\leq(N+1)n$.  

Note that every $a$-letter of $W_b$ and of $W_t$ corresponds to an edge of $\textbf{t}_i$ for some $i\in\{1,3\}$, and the maximal $a$-band of $\Delta_i'$ that start at such an edge must end on either $\textbf{s}_i$ or on a $(\theta,q)$-cell.  Hence, $|\textbf{s}_1|_a+|\textbf{s}_3|_a\geq|\textbf{t}_1|_a+|\textbf{t}_3|_a-2(N+1)n$, and so:
$$n\geq\delta(|u|_a+|v|_a)\geq\delta(|\textbf{s}_1|_a+|\textbf{s}_3|_a)\geq\delta m-2(N+1)n$$
Thus, (\ref{suffices-2}) follows by taking $C_\ell$ sufficiently large with respect to $\delta^{-1}$ and $N$.

\textbf{2.} Otherwise, suppose every maximal $\theta$-band is radial.

If every maximal $\theta$-band crosses $\pazocal{Q}$ at most twice, then $\ell(\pazocal{Q})\leq|u|_\theta+|v|_\theta\leq n$.  Hence, we may assume there exists a maximal $\theta$-band $\pazocal{T}$ which crosses $\pazocal{Q}$ in $k\geq3$ cells.  Note that this is possible if $\pazocal{T}$ forms a `spiral' in $\Delta$ (see \Cref{fig-spiral}); such structures are studied in a similar context by Ol'shanskii and Sapir in \cite{OSconj}, and in another context by Arenas and Wise in \cite{ArenasWise}.

\begin{figure}[H]
\centering
\includegraphics[scale=1.4]{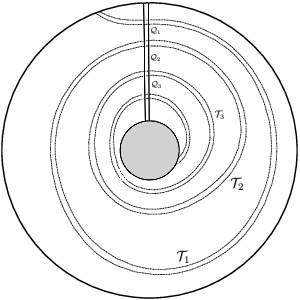}
\caption{Spiral with $k=4$} \label{fig-spiral}
\end{figure}

Let $\pazocal{T}_1$ be the subband of $\pazocal{T}$ between the first and second crossings of $\pazocal{T}$ and $\pazocal{Q}$.  In the same way, define the subbands $\pazocal{T}_2,\dots,\pazocal{T}_{k-1}$.  By construction each $\pazocal{T}_i$ crosses $\pazocal{Q}$ exactly twice: In its first cell $\pi_i$ and its last cell $\pi_{i+1}$.

Similarly, for $i=1,\dots,k-1$ let $\pazocal{Q}_i$ be the subband of $\pazocal{Q}$ between $\pi_i$ and $\pi_{i+1}$, not including $\pi_{i+1}$.  As $\theta$-bands cannot cross, it follows that there exist a sequence of `spirals' which cross each $\pazocal{Q}_i$ exactly once.  As such, each $\pazocal{Q}_i$ has the same history, denoted $H_0$, which satisfies $\|H_0\|\leq n$.

As in the previous case, let $\Gamma$ be the circular diagram formed from $\Delta$ by cutting along the side of $\pazocal{Q}$ and pasting a copy of $\pazocal{Q}$ to the side from which $\pazocal{Q}$ is removed.  Then the sides of $\pazocal{T}_1$ and $\pazocal{T}_{k-1}$ bound a subdiagram $\Gamma_0$ in $\Gamma$ which is a trapezium with history $H_0^{k-2}$.  Let $\pazocal{C}_0:W_0\to\dots\to W_h$ be the reduced computation corresponding to $\Gamma_0$ by \Cref{trapezia are computations}.

Suppose $W_0\cdot H_0\equiv W_0$, so that the sides of $\pazocal{T}_1$ and $\pazocal{T}_2$ are identically labelled.  Since the sides of $\pazocal{Q}_1$ and $\pazocal{Q}_2$ are also identically labelled, these paths may be combined to form loops $\gamma_1$ and $\gamma_2$ in $\Delta$ which are identically labelled.  Using $0$-refinement these loops may be assumed to be disjoint.  But then they bound an annular diagram containing $\pazocal{T}_1$, so that \Cref{exciseTwoPaths} contradicts the minimality of $\text{Area}(\Delta)$.

Hence, $W_0\cdot H_0\not\equiv W_0$, and so \Cref{history as powers} implies there exists a two-letter subword $UV$ of the base of $\pazocal{C}_0$ such that $k-2\leq\|w_0\|+\|w_h\|+2n$ for $w_0,w_h$ the tape words of $W_0,W_h$ in the $UV$-sector.

Let $\pazocal{Q}'$, $\pazocal{Q}''$ be the maximal (radial) $q$-bands in $\Delta$ corresponding to the $UV$-sector.  There then exists a decomposition $\pazocal{Q}'=\pazocal{Q}_1'\pazocal{Q}_2'\pazocal{Q}_3'$ such that the history of $\pazocal{Q}_2'$ is $H_0^{k-2}$.  Since distinct $\theta$-bands cannot cross, at most two cells of $\pazocal{Q}_1'$ (or of $\pazocal{Q}_3')$ correspond to the same maximal $\theta$-band.  As such, $\ell(\pazocal{Q}_1')+\ell(\pazocal{Q}_3')\leq n$.  

Decomposing $\pazocal{Q}''=\pazocal{Q}_1''\pazocal{Q}_2''\pazocal{Q}_3''$ in the analogous way, we similarly see that the history of $\pazocal{Q}_2''$ is $H_0^{k-2}$ and $\ell(\pazocal{Q}_1'')+\ell(\pazocal{Q}_3'')\leq n$.

Now, let $\textbf{p}_1$ be the subpath of a side of $\pazocal{T}_1$ corresponding to the tape word of the $UV$-sector.  Then $\textbf{p}_1$, sides of $\pazocal{Q}_1'$, $\pazocal{Q}_1''$, and a subpath $\textbf{q}_1$ of a boundary component of $\Delta$ bound a circular subdiagram $\Psi_1$ of $\Delta$ which contains no $(\theta,q)$-cells.  Note then that any maximal $a$-band of $\Psi_1$ which ends on $\textbf{p}_1$ also ends on a side of $\pazocal{Q}_1'$, on a side of $\pazocal{Q}_1''$, or on $\textbf{q}_1$.  As such, $|\textbf{q}_1|_a\geq|\textbf{p}_1|_a-\ell(\pazocal{Q}_1')-\ell(\pazocal{Q}_1'')$.

Similarly, letting $\textbf{p}_3$ be the subpath of a side of $\pazocal{T}_{k-1}$ corresponding to the tape word of the $UV$-sector, we find a subpath $\textbf{q}_3$ of a boundary component of $\Delta$ such that $|\textbf{q}_3|_a\geq|\textbf{p}_3|_a-\ell(\pazocal{Q}_3')-\ell(\pazocal{Q}_3'')$.  But $\textbf{q}_1$ and $\textbf{q}_3$ are subpaths of opposite boundary components, so that:
$$\delta^{-1}n\geq|\textbf{q}_1|_a+|\textbf{q}_3|_a\geq|\textbf{p}_1|_a+|\textbf{p}_3|_a-(\ell(\pazocal{Q}_1')+\ell(\pazocal{Q}_3'))-(\ell(\pazocal{Q}_1'')+\ell(\pazocal{Q}_3''))\geq\|w_1\|+\|w_h\|-2n$$
As a result, $k-2\leq\|w_0\|+\|w_h\|+2n\leq(\delta^{-1}+4)n$, so that $\|H_0^{k-2}\|\leq(\delta^{-1}+4)n^2$.

Thus, $\ell(\pazocal{Q}')\leq(\delta^{-1}+4)n^2+n$, so that (\ref{suffices-2}) follows by taking $C_\ell=\delta^{-1}+4$.

\end{proof}


\begin{proof}[Proof of \Cref{main-theorem}] \

In light of \Cref{CL upper bound}, it suffices to show that $\CL_{M(\textbf{E}_\textbf{S})}\succeq\TM_\textbf{S}$.

Let $\a$ be an accepted input of $\textbf{S}$ and let $W_\a$ be the input configuration of $\textbf{E}_\textbf{S}$ corresponding to $\a$.  Since the $R_sP_0$-sector has empty tape alphabet, any accepting computation corresponds through \Cref{computations are trapezia} to a trapezium $\Delta_\a$ with bottom label $W_\a$, top label $ W_{ac}$, and identically labelled sides.  Hence, $W_\a$ and $W_{ac}$ are conjugate in $M(\textbf{E}_\textbf{S})$.

Now, by van Kampen's Lemma we may form a reduced annular diagram $\Gamma_\a$ with boundary components $W_\a$ and $W_{ac}$ such that there exists a path $\gamma$ between the boundary components with $\|\gamma\|=c_{\textbf{E}_\textbf{S}}(W_\a,W_{ac})$.  \Cref{M(S) annuli}(5) then implies any of the $k$ maximal $\theta$-bands in $\Gamma_\a$ must be a non-contractible $\theta$-annulus.  Then $\gamma$ must contain at least one $\theta$-edge from each such annulus, so that $c_{\textbf{E}_\textbf{S}}(W_\a,W_{ac})\geq k$.

As each $q$-letter of a configuration belongs to a distinct part of the state letters, every maximal $q$-band in $\Gamma_\a$ must be radial.  Cutting along the side of the $q$-band corresponding to the first letter of $W_\a$ then produces a trapezium $\Gamma_\a'$ with bottom label $W_\a$, top label $W_{ac}$, and history of length $k$.  So, \Cref{trapezia are computations} produces a reduced computation of $\textbf{E}_\textbf{S}$ of length between $W_\a$ and $W_{ac}$.  But then \Cref{E0 language} implies $c_{\textbf{E}_\textbf{S}}(W_\a,W_{ac})\geq\TM_\textbf{S}(|W_\a|_a)$, so that $$\CL_{M(\textbf{E}_\textbf{S})}(2N+\|\a\|)\geq\TM_\textbf{S}(\|\a\|)$$

Thus $\CL_{M(\textbf{E}_\textbf{S})}(n+2N)\geq\TM_\textbf{S}(n)$ for all $n$, completing the proof.

\end{proof}

}

\section{Proof of Theorem \ref{Dehn-to-CL}}

As our goal in this section is to prove \Cref{Dehn-to-CL}, we begin by fixing a finite presentation $\gen{X\mid\pazocal{R}}$ for the group $G$ and look to construct the finitely presented group $H$ with $\CL_H\sim\text{Dehn}_G$.  Without loss of generality, we assume $\pazocal{R}$ consists entirely of non-trivial cyclically reduced words and is closed under inverses and cyclic permutations {\frenchspacing(i.e. it is symmetrized)}.

As Servatius shows that RAAG's have linear conjugator length function \cite{Servatius}, it suffices by the isoperimetric gap \cite{Gromov,Bowditch,O91} to assume $\text{Dehn}_H(n)\succeq n^2$.  Consequently, it suffices to construct a cyclic $S$-machine $\textbf{S}$ with $\TM_\textbf{S}\sim\text{Dehn}_G$, as then \Cref{main-theorem} implies the finitely presented group $M(\textbf{E}_\textbf{S})$ has the desired conjugator length function.

A similar task is faced in Theorem 1.1 of \cite{SBR}, where a non-deterministic Turing machine $\pazocal{T}$ is constructed to have time function equivalent to a given Dehn function.  We mimic this construction for $S$-machines, using the known complexity of $\pazocal{T}$ to our advantage in simplifying the proof:

\medskip

Using a `standard trick' (see Lemma 12.17 and Exercise 12.12 of \cite{Rotman}), we construct a new finite presentation $\tilde{\pazocal{P}}=\gen{Y\mid\tilde{\pazocal{R}}\cup\tilde{\pazocal{S}}}$ for $G$ as follows:

Let $\bar{X}$ be a copy of $X$ and set $Y=X\sqcup\bar{X}$.  For each $x\in X$, denote $\bar{x}$ to be the corresponding letter of $\bar{X}$.  Extend this to define $\bar{y}$ for any $y\in Y$ by setting $\bar{y}=x$ if $y=\bar{x}$.

For any word $w\in(Y\cup Y^{-1})^\ast$, define $p(w)$ to be the positive word in $Y$ obtained by replacing any occurrence of a negative letter $y^{-1}$ with the letter $\bar{y}$.  Note that $\|p(w)\|=\|w\|$ and $p(w_1w_2)\equiv p(w_1)p(w_2)$.  We then define $\tilde{\pazocal{R}}=\{p(r)\mid r\in\pazocal{R}\}$ and $\tilde{\pazocal{S}}=\{y\bar{y}\mid y\in Y\}$.

Note then that $\tilde{\pazocal{R}}\cup\tilde{\pazocal{S}}$ is a set of positive words in $Y$ which is closed under cyclic permutation.  Moreover, for any $w\in(Y\cup Y^{-1})^*$, it can be derived from the relators $\tilde{\pazocal{S}}$ that $w$ and $p(w)$ represent the same element of $G$.

The standard base of $\textbf{S}$ is then $Q_0Q_1Q_2$, with the $Q_0Q_1$-sector taken to be the input sector.  Each part $Q_i$ consists of a start letter $q_{i,s}$, an end letter $q_{i,f}$, and several other `intermediate states' described by the software of the machine.  The input tape alphabet is identified with $Y$, while the tape alphabet of the $Q_1Q_2$-sector is taken to be a copy $Y'$ of $Y$; for simplicity, denote the element of $Y'$ corresponding to the element $y\in Y$ by $y'$.

The positive rules of $\textbf{S}$ come in the following forms:

\begin{itemize}

\item For any $y\in Y$, there is a corresponding positive rule $$\sigma_y=[q_{0,s}\to q_{0,s}, \ \ q_{1,s}\to yq_{1,s}, \ \ q_{2,s}\to y'q_{2,s}]$$

\medskip

\item For any $r\in\pazocal{R}$, fix $y_1,\dots,y_k\in Y$ such that $p(r)\equiv y_1\dots y_k$.  For $i=0,1,2$ we introduce intermediate states $q_{i,r,1},\dots,q_{i,r,k-1}\in Q_i$ and, for completeness, denote $q_{i,r,0}=q_{i,r,k}=q_{i,s}$.  We then define positive rules $\rho_{r,1},\dots,\rho_{r,k}$ given by:
$$\rho_{r,j}=[q_{0,r,j-1}\to q_{0,r,j}, \ \ q_{1,r,j-1}\to y_jq_{1,r,j}, \ \ q_{2,r,j-1}\to q_{2,r,j}]$$

\medskip

\item For any $y\in Y$ and any $i=0,1,2$, there is an intermediate state $q_{i,y}\in Q_i$ and positive rules $\tau_{1,y},\tau_{2,y}$ given by:
\begin{align*}
\tau_{1,y}&=[q_{0,s}\to q_{0,y}, \ \ q_{1,s}\to yq_{1,y}, \ \ q_{2,s}\to q_{2,y}] \\
\tau_{2,y}&=[q_{0,y}\to q_{0,s}, \ \ q_{1,y}\to \bar{y}q_{1,s}, \ \ q_{2,y} \to q_{2,s}]
\end{align*}

\medskip

\item There is a positive rule $\omega$ which transitions to the end state, given by
$$\omega=[q_{0,s}\xrightarrow{\ell}q_{0,f}, \ \ q_{1,s}\xrightarrow{\ell}q_{1,f}, \ \ q_{2,s}\to q_{2,f}]$$

\end{itemize}

\medskip

Define the word $\tau_y=\tau_{1,y}\tau_{2,y}$ for any $y\in Y$.  Similarly, for any relator $r\in\pazocal{R}$, define the word $\rho_r=\rho_{r,1}\dots\rho_{r,k}$ for $k=\|r\|$.

A reduced computation $\pazocal{C}:W_0\to\dots\to W_t$ of $\textbf{S}$ is called a \textit{start computation} if both $W_0$ and $W_t$ are start configurations.  Note that since an end configuration is only admissible for the rule $\omega^{-1}$, the history of a start computation cannot contain any letter of the form $\omega^{\pm1}$.  As there are exactly two rules (one positive and one negative) for which any configuration in an intermediate state is admissible, the next statement then follows immediately:

\begin{lemma} \label{lem-factorization}

For any start computation $\pazocal{C}$, there exists a factorization $H\equiv H_1\dots H_\ell$ of the history of $\pazocal{C}$ such that for each $1\leq i\leq \ell$, either:

\begin{itemize}

\item $H_i$ is a (reduced) word in the $\sigma_y^{\pm1}$,

\item $H_i\equiv\tau_y^{\pm1}$ for some $y\in Y$, or

\item $H_i\equiv\rho_r^{\pm1}$ for some $r\in\pazocal{R}$

\end{itemize}

\end{lemma}

Given a start configuration $W\equiv q_{0,s}~w~q_{1,s}~v'~q_{2,s}$, let $v\in(Y\cup Y^{-1})^\ast$ be the natural copy of $v'\in F(Y')$ and define the word $h(W)=wv^{-1}\in(Y\cup Y^{-1})^\ast$.

\begin{lemma} \label{lem-closed}

For any start computation $\pazocal{C}:W_0\to\dots\to W_t$, $h(W_0)$ and $h(W_t)$ represent the same element of $G$.

\end{lemma}

\begin{proof}

By \Cref{lem-factorization}, it suffices to assume the history $H$ of $\pazocal{C}$ is of the form $\sigma_y$, $\tau_y$, or $\rho_r$.

But in each case, $h(W_t)$ is obtained from $h(W_0)$ by inserting a relation of $H$: A cancellable pair of letters in case $\sigma_y$, an element of $\tilde{\pazocal{S}}$ in case $\tau_y$, and an element of $\tilde{\pazocal{R}}$ in case $\rho_r$.

\end{proof}

\begin{lemma} \label{lem-positive}

Fix a start configuration $W\equiv q_{0,s}~w~q_{1,s}~v'~q_{2,s}$.  Suppose $w$ is freely equal to $w_1w_2$ and set $w'=w_1p(w_2)$.  Then there exists a start computation of length at most $4\|w_2\|$ between $W$ and $W'\equiv q_{0,s}~w'~q_{1,s}~v'~q_{2,s}$.

\end{lemma}

\begin{proof}

Fix $w_2\equiv y_1^{\eps_1}\dots y_k^{\eps_k}$ for $y_i\in Y$ and $\eps_i\in\{\pm1\}$.

If $\eps_i=1$, set $z_i=y_i$ and $H_i=\sigma_{z_i}^{-1}$; if instead $\eps_i=-1$, set $z_i=\bar{y}_i$ and $H_i=\tau_{y_i}\sigma_{z_i}^{-1}$.  Note then that $p(w)\equiv z_1\dots z_k$.

Then for $H$ the reduced word freely equal to $H_k\dots H_1\sigma_{z_1}\dots\sigma_{z_k}$, we have $W\cdot H\equiv W'$.  Moreover, since $\|H_i\|\leq3$ for all $i$, we have $\|H\|\leq4k$.

\end{proof}

Note that the Turing machine $\pazocal{T}$ consists of two tapes on which (positive) words in $Y$ are written.  In the original setting of \cite{SBR}, though, $\bar{X}$ is simply referred to as $X^{-1}$; however, the `negative letters' there are simply formal inverses akin to our $\bar{X}$ rather than `actual' inverses.

As such, for each configuration $V$ of $\pazocal{T}$ in the start (or end) state, there exists a start (or end) configuration $V^{\textbf{S}}$ of $\textbf{S}$ whose tape words correspond to the words written on the tapes in $V$.

It should be noted, however, that this does not define a correspondence between the start (or end) configurations of the two machines, as the tape word of a configuration of $\textbf{S}$ need not be positive.

\medskip

\begin{lemma} \label{lem-emulation}

Let $\pazocal{C}$ be a computation of $\pazocal{T}$ of length $k$ between start or end configurations $V_1$ and $V_2$.  Then there exists a reduced computation $\pazocal{C}^\textbf{S}$ of $\textbf{S}$ of length $\leq5k$ between $V_1^\textbf{S}$ and $V_2^\textbf{S}$.

\end{lemma}

\begin{proof}

It suffices to show that any computation of $\pazocal{T}$ of length $k$ corresponding to one of the four main `operations' can be emulated by a reduced computation of $\textbf{S}$ whose length is at most $5k$.

\begin{enumerate}

\item (move) In this case, $k=1$ and both $V_1$ and $V_2$ are configurations in the start state.  Assuming $y\in Y$ is the rightmost letter of the word written on the input tape in $V_1$, this operation erases this letter and writes $\bar{y}$ on the right of the second tape.

In $\textbf{S}$, the analogue is achieved by a reduced computation with history $\tau_{\bar{y}}^{-1}\sigma_{\bar{y}}$.


\medskip

\item (substitution) In this case, there exists $r\equiv r_1r_2\in\pazocal{R}$ such that $k=\|r\|$, both $V_1$ and $V_2$ are configurations in the start state, and the word written on the input tape $V_1$ is of the form $wp(r_1)$.  This operation then erases the suffix $p(r_1)$ and replaces it with $p(r_2^{-1})$.  

As $\pazocal{R}$ is symmetrized, we have $s\equiv r_2r_1\in\pazocal{R}$, so that $p(s)\equiv p(r_2)p(r_1)\in\tilde{\pazocal{R}}$.  Then $\tau_s^{-1}$ is the history of a start computation of length $\|s\|=k$ between $V_1^\textbf{S}$ and the start configuration $W$ with the same tape word in the $Q_1Q_2$-sector and input tape word freely equal to $wp(r_2)^{-1}$.  But since $p(p(r_2)^{-1})=p(r_2^{-1})$, \Cref{lem-positive} then yields a start computation of length $\leq4k$ between $W$ and $V_2^\textbf{S}$.  

\medskip

\item (reduction) As in (move), $k=1$ and $V_1$ and $V_2$ are configurations in the start state.  In this case, $y\in Y$ is the rightmost letter of the word written on each tape of $V_1$, and this operation erases both of these letters.

The analogue in $\textbf{S}$ is then achieved by a reduced computation with history $\sigma_y^{-1}$.

\medskip

\item (accept) Finally, in this case $k=1$ and $V_1$ and $V_2$ are the start and end configurations, respectively, whose tapes are empty.

The analogue is thus achieved by a reduced computation with history $\omega$.

\end{enumerate}

\end{proof}

\begin{proof}[Proof of \Cref{Dehn-to-CL}] \

Define $\pazocal{L}=\{w\in(Y\cup Y^{-1})^\ast\mid w\text{ is freely reduced, }w=_G1\}$.  

For any reduced word $w$ over $Y\cup Y^{-1}$, let $I(w)$ be the input configuration of $\textbf{S}$ with input $w$.  Note then that $h(I(w))=w$.

If $w$ is an accepted input, then for any history $H$ of an computation accepting $I(w)$, the word $H\omega^{-1}$ is freely equal to the history of a start computation between $I(w)$ and $I(1)$.  \Cref{lem-closed} thus implies $w=_G1$, so that $w\in\pazocal{L}$.

Conversely, fix $w\in\pazocal{L}$ with $\|w\|\leq n$.  \Cref{lem-positive} then produces a reduced computation $\pazocal{C}_1$ of length at most $4n$ between $I(w)$ and $I(p(w))$.  By the proof of Theorem 1.1 in \cite{SBR}, $p(w)$ is an accepted input of $\pazocal{T}$, so that there exists a computation of $\pazocal{T}$ accepting $p(w)$ with length at most $\TM_\pazocal{T}(n)$.  \Cref{lem-emulation} thus produces a reduced computation $\pazocal{C}_2$ of $\textbf{S}$ accepting $I(p(w))$ with length at most $5\TM_\pazocal{T}(n)$.  Thus, concatenating $\pazocal{C}_1$ and $\pazocal{C}_2$ yields an accepting computation for $I(w)$ of length at most $5\TM_\pazocal{T}(n)+4n$.

It therefore follows that $\pazocal{L}$ is the language of accepted inputs of $\textbf{S}$ and that $\TM_\textbf{S}\preceq\TM_\pazocal{T}$.

Now suppose $\pazocal{C}$ is an arbitrary reduced computation accepting $I(w)$.  \Cref{lem-factorization} then implies there exists a factorization $H\equiv H_1H_1'H_2H_2'\dots H_kH_k'H_{k+1}\omega$ of the history $H$ of $\pazocal{C}$ such that:

\begin{itemize}

\item Each $H_i$ is a reduced (perhaps empty) word over the letters $\sigma_y^{\pm1}$

\item Each $H_i'$ is a word of the form $\tau_y^{\pm1}$ or $\rho_r^{\pm1}$.

\end{itemize}

Let $\b_i$ be the natural copy of $H_i$ over $Y$.  Similarly, let $r_i=p(r)\in\tilde{\pazocal{R}}$ if $H_i'=\rho_r^{\pm1}$ and let $r_i=y\bar{y}\in\tilde{S}$ if $H_i'=\tau_y^{\pm1}$.

Let $V\equiv I(w)\cdot H_1H_1'H_2H_2'\dots H_kH_k'H_{k+1}$ and let $v_i$ be its tape word in the $Q_{i-1}Q_i$-sector.  

By the operation of the rules, $v_1=_{F(Y)}w\b_1r_1^{\pm1}\b_2r_2^{\pm1}\dots\b_kr_k^{\pm1}\b_{k+1}$.

For $j=1,\dots,k$, let $f_j$ be the reduced word freely equal to the product $\b_1\dots\b_j$.  Then we may `collect conjugates' so that $v_1$ is freely equal to:
$$w\left(\prod_{i=1}^k f_ir_i^{\pm1}f_i^{-1}\right)f_{k+1}$$
Similarly, the operation of the rules dictates that $v_2$ is freely equal to $\b_1\dots\b_{k+1}=f_{k+1}$.

But $V$ is $\omega$-admissible, so that each of these tape words must be trivial.  As such, $w$ is freely equal to a product of $k$ conjugates of relators of $\tilde{\pazocal{P}}$.  Hence, since each $H_i'$ must be non-trivial, we have $\|H\|\geq k\geq\text{Area}_{\tilde{\pazocal{P}}}(w)$, and so $\tm(I(w))\geq\text{Area}_{\tilde{\pazocal{P}}}(w)$.  Assuming $\tm(I(w))$ is maximal amongst all accepted input configurations with $a$-length at most $n$, we thus have $\TM_\textbf{S}(n)\geq\text{Dehn}_{\tilde{\pazocal{P}}}(n)$

Combining these relations with the main result of Theorem 1.1 in \cite{SBR} therefore implies $$\text{Dehn}_G\preceq\TM_\textbf{S}\preceq\TM_\pazocal{T}\sim\text{Dehn}_G$$ completing the proof.

\end{proof}

\bigskip

\section{Proof of Theorem \ref{noCLBound}}
Let $H=\langle a,b,c\mid [a,b]=c, ~[a,c]=[b,c]=1\rangle$ be the 3-dimensional integral Heisenberg group and define $H_\textbf{S}$ to be the direct product $M(\textbf{E}_\textbf{S})\times H$.  A combination theorem of Brick \cite{Brick} then implies $$\text{Dehn}_{M(\textbf{E}_\textbf{S})}(n)+\text{Dehn}_H(n)\preceq\text{Dehn}_{H_\textbf{S}}(n)\preceq n^2+\text{Dehn}_{M(\textbf{E}_\textbf{S})}(n)+\text{Dehn}_H(n)$$ 
But it is shown in \cite{epstein1992word} that $\text{Dehn}_{H}(n)\sim n^3$, so that \Cref{Dehn bound} implies $\text{Dehn}_{H_\textbf{S}}(n)\sim n^3$.

A similar straightforward argument produces a combination theorem for conjugator length functions, implying $\CL_{H_\textbf{S}}\sim\CL_{M(\textbf{E}_\textbf{S})}+\CL_H$.  But it is shown in \cite{BRS} that $\CL_H(n)\sim n^2$, so that \Cref{main-theorem} implies $\CL_{H_\textbf{S}}\sim \TM_\textbf{S}$.

\bigskip

\section{Proof of Theorem \ref{annular-Dehn}}

It follows quickly from the definition of the annular Dehn function \cite[Proposition 1.3]{BrickCorson} that $\text{Ann}_G(n)\preceq\text{Dehn}_G(n+\CL_G(n))$ for any finitely presented group $G$.  Hence, \Cref{Dehn bound} implies $\text{Ann}_{M(E_\textbf{S})}\preceq\TM_\textbf{S}^3$.

Conversely, given an accepted input $\a$ of $\textbf{S}$, we saw in the proof of \Cref{main-theorem} that any reduced annular diagram $\Gamma_\a$ over $M(\textbf{E}_\textbf{S})$ with boundary labels $W_\a$ and $W_{ac}$ corresponds to a trapezium $\Gamma_\a'$ with bottom label $W_\a$ and top label $W_{ac}$.  By \Cref{trapezia are computations} there then exists a corresponding reduced computation $\pazocal{C}$ of $\textbf{E}_\textbf{S}$ accepting $W_\a$.

Lemmas \ref{E primitive step history} and \ref{E run step history} then imply that the step history of $\pazocal{C}$ must have prefix $(1)(2)(3)(4)(5)$.  Let $\pazocal{C}_3:W'\to\dots\to W''$ be the maximal subcomputation with step history $(3)$ in this prefix.

As in the proof of \Cref{E0 language}, $W'$ has $\a$ written in its input sector, the copy of the same word $H\in F(\Phi^+)$ written in the left historical alphabet of each $Q_{i,\ell}Q_{i,r}$-sector, and all other sectors empty.  As in that setting, letting $I_\textbf{S}(\a)$ be the input configuration of $\textbf{S}$ corresponding to the input $\a$, we have $\|H\|\geq\tm(I_\textbf{S}(\a))$.

Let $\pazocal{C}_3'$ be the restriction of $\pazocal{C}_3$ to any $Q_{i,\ell}Q_{i,r}$-sector.  Then there exists a subdiagram $\Gamma_3'$ of $\Gamma_\a'$ which is a trapezium corresponding to $\pazocal{C}_3'$ in the sense of \Cref{computations are trapezia}.

For $H\equiv h_1\dots h_k$, let $H_{\ell}\equiv h_{1,\ell}\dots h_{k,\ell}$ be the tape word of the initial admissible subword of $\pazocal{C}_3'$.  Then the $a$-edge of $\textbf{tbot}(\Gamma_3')$ corresponding to the letter $h_{j,\ell}$ marks the start of a maximal $a$-band of $\Gamma_3'$ that crosses the bottom $j-1$ maximal $\theta$-bands before ending on a $(\theta,q)$-cell.  A similar analysis applies to the $a$-bands which end on $\textbf{ttop}(\Gamma_3')$, so that $\text{Area}(\Gamma_3')\geq\sum_{j=1}^k 2j\geq k^2$.

But then $\text{Area}(\Gamma_\a)=\text{Area}(\Gamma_\a')\geq\|H\|^2\geq\tm(I_\textbf{S}(\a))^2$.

Thus, fixing $n$ and taking $\a$ be the accepted input with $\|\a\|\leq n$ realizing $\tm(I_\textbf{S}(\a))=\TM_\textbf{S}(n)$, $$\text{Ann}_{M(\textbf{E}_\textbf{S})}(2N+n)\geq\TM_\textbf{S}(n)^2$$
completing the proof.

\bigskip

\bibliographystyle{plain}
\bibliography{biblio}

@Inbook{Brady2007,
title="The Isoperimetric Spectrum",
bookTitle="The Geometry of the Word Problem for Finitely Generated Groups",
year="2007",
publisher="Birkh{\"a}user Basel",
address="Basel",
pages="5--27",
author="Brady, N.",
isbn="978-3-7643-7950-6",
doi="10.1007/978-3-7643-7950-6_1",
url="https://doi.org/10.1007/978-3-7643-7950-6_1"
}

@article{BradyBridson,
author = {Brady, N. and Bridson, M.},
year = {2000},
month = {12},
pages = {1053-1070},
title = {There is only one gap in the isoperimetric spectrum},
volume = {10},
journal = {Geometric and Functional Analysis},
doi = {10.1007/PL00001646}
}

@misc{vandeputte2026residualfinitenesspropertieshalls,
      title={Residual finiteness properties of some of {H}alls groups}, 
      author={L. Vandeputte},
      year={2026},
      eprint={2508.16452},
      archivePrefix={arXiv},
      primaryClass={math.GR},
      url={https://arxiv.org/abs/2508.16452}, 
}

@article{BBFS,
  title={Snowflake groups, Perron--Frobenius eigenvalues and isoperimetric spectra},
  author={Brady, N. and Bridson, M. and Forester, M. and Shankar, K.},
  journal={Geometry \& Topology},
  volume={13},
  number={1},
  pages={141--187},
  year={2008},
  publisher={Mathematical Sciences Publishers}
}

@misc{BridsonRileyNilp,
      title={Linear Diophantine equations and conjugator length in 2-step nilpotent groups}, 
      author={M. R. Bridson and T. R. Riley},
      year={2025},
      eprint={2506.01239},
      archivePrefix={arXiv},
      primaryClass={math.GR},
      url={https://arxiv.org/abs/2506.01239}, 
}

@book{epstein1992word,
  title={Word Processing in Groups},
  author={Epstein, D. and Cannon, J. and Holt, D. and Levy, S. and Thurston, W. },
  isbn={9780867202441},
  lccn={lc91046119},
  series={A {K} Peters Series},
  url={https://books.google.com/books?id=DQ84QlTr-EgC},
  year={1992},
  publisher={Taylor \& Francis}
}

@unpublished{BRS,
	author = {M. R. Bridson and T. R. Riley and A. Sale},
	note = {In preparation},
	title = {Conjugator length in finitely presented groups}}

@misc{BridsonRileyFil,
      title={The lengths of conjugators in the model filiform groups}, 
      author={M. R. Bridson and T. R. Riley},
      year={2025},
      eprint={2506.01235},
      archivePrefix={arXiv},
      primaryClass={math.GR},
      url={https://arxiv.org/abs/2506.01235}, 
}

@misc{BridsonRileySnowflake,
      title={Snowflake groups and conjugator length functions with non-integer exponents}, 
      author={M. R. Bridson and T. R. Riley},
      year={2025},
      eprint={2512.14038},
      archivePrefix={arXiv},
      primaryClass={math.GR},
      url={https://arxiv.org/abs/2512.14038}, 
}

@book{Miller71,
 ISBN = {9780691080918},
 URL = {http://www.jstor.org/stable/j.ctt1b7x83g},
 author = {C. Miller},
 publisher = {Princeton University Press},
 title = {On Group-Theoretic Decision Problems and Their Classification. (AM-68)},
 urldate = {2022-12-26},
 year = {1971}
}

@article{BrickCorson, title={Annular {D}ehn functions of groups}, volume={58}, DOI={10.1017/S0004972700032433}, number={3}, journal={Bulletin of the Australian Mathematical Society}, author={Brick, S. G. and Corson, J. M.}, year={1998}, pages={453–464}}

@misc{BridsonRileyFastGrowing,
      title={Groups with fast-growing conjugator length functions}, 
      author={M. R. Bridson and T. R. Riley},
      year={2025},
      eprint={2512.23674},
      archivePrefix={arXiv},
      primaryClass={math.GR},
      url={https://arxiv.org/abs/2512.23674}, 
}

@article{Brick,
 ISSN = {00029947},
 URL = {http://www.jstor.org/stable/2154273},
 author = {S. G. Brick},
 journal = {Transactions of the American Mathematical Society},
 number = {1},
 pages = {369--384},
 publisher = {American Mathematical Society},
 title = {On Dehn Functions and Products of Groups},
 urldate = {2026-01-08},
 volume = {335},
 year = {1993}
}

@unpublished{WEmb,
	author = {Wagner, F.},
	doi = {10.48550/arXiv.2509.17841},
	month = {09},
	title = {Quasi-isometric Higman embeddings and the {D}ehn function},
	year = {2025}}

@article{W,
	author = {F. Wagner},
	doi = {10.1090/memo/1589},
	journal = {Memoirs of the American Mathematical Society},
	title = {Torsion Subgroups of Groups with Quadratic {D}ehn Function},
	volume = {313},
	year = {2025}}

@misc{WCubic,
	archiveprefix = {arXiv},
	author = {F. Wagner},
	eprint = {2001.03261},
	primaryclass = {math.GR},
	title = {Torsion Subgroups of Groups with Cubic {D}ehn Function},
	url = {https://arxiv.org/abs/2001.03261},
	year = {2020}}

@article{BORS,
	author = {J.-C. Birget and A. Yu Ol'shanskii and E. Rips and M. V. Sapir},
	issn = {0003486X},
	journal = {Annals of Mathematics},
	number = {2},
	pages = {467--518},
	publisher = {Annals of Mathematics},
	title = {Isoperimetric Functions of Groups and Computational Complexity of the Word Problem},
	url = {http://www.jstor.org/stable/3597196},
	urldate = {2024-06-25},
	volume = {156},
	year = {2002}}

@article{Servatius,
	author = {Servatius, H.},
	coden = {JALGA4},
	doi = {10.1016/0021-8693(89)90319-0},
	fjournal = {Journal of Algebra},
	issn = {0021-8693},
	journal = {J. Algebra},
	mrclass = {20F28 (05C25 16A60)},
	mrnumber = {1023285 (90m:20043)},
	mrreviewer = {Stephen J. Pride},
	number = {1},
	pages = {34--60},
	title = {Automorphisms of graph groups},
	url = {http://dx.doi.org/10.1016/0021-8693(89)90319-0},
	volume = {126},
	year = {1989}}

@misc{WMal,
	archiveprefix = {arXiv},
	author = {F. Wagner},
	eprint = {2404.00841},
	primaryclass = {math.GR},
	title = {Malnormal Subgroups of Finitely Presented Groups},
	year = {2024}}

@misc{ArenasWise,
      title={Linear isoperimetric functions for surfaces in hyperbolic groups}, 
      author={M. Arenas and D. T. Wise},
      year={2022},
      eprint={2205.10096},
      archivePrefix={arXiv},
      primaryClass={math.GT},
      url={https://arxiv.org/abs/2205.10096}, 
}

@misc{CW,
	archiveprefix = {arXiv},
	author = {B. Chornomaz and F. Wagner},
	eprint = {2304.07603},
	primaryclass = {math.GR},
	title = {Quasilinear Emulation of {T}uring Machines by {S}-machines},
	year = {2023}}

@article{v-K,
	author = {E. van Kampen},
	journal = {American Journal of Mathematics},
	number = {1},
	pages = {268--273},
	publisher = {JSTOR},
	title = {On some lemmas in the theory of groups},
	volume = {55},
	year = {1933}}

@article{S06,
	author = {M. V. Sapir},
	journal = {International Congress of Mathematicians, ICM 2006},
	month = {03},
	title = {Algorithmic and asymptotic properties of groups},
	volume = {2},
	year = {2006}}

@book{Rotman,
	author = {J. Rotman},
	publisher = {Springer Science \& Business Media},
	title = {An Introduction to the Theory of Groups},
	volume = {148},
	year = {2012}}

@article{OS19,
	author = {A. Yu. Ol'shanskii and M. V. Sapir},
	doi = {10.1142/S1664360719500231},
	journal = {Bulletin of Mathematical Sciences},
	month = {11},
	title = {Conjugacy problem in groups with quadratic {D}ehn function},
	volume = {10},
	year = {2019}}

@article{OS12,
	author = {A. Yu Ol'shanskii and M. V. Sapir},
	doi = {10.1112/jtopol/jts020},
	eprint = {https://academic.oup.com/jtopol/article-pdf/5/4/785/2715519/jts020.pdf},
	issn = {1753-8416},
	journal = {Journal of Topology},
	month = {09},
	number = {4},
	pages = {785-886},
	title = {Groups with undecidable word problem and almost quadratic {D}ehn function},
	url = {https://doi.org/10.1112/jtopol/jts020},
	volume = {5},
	year = {2012}}

@article{OS06,
	author = {A. Yu. Ol'shanskii and M. V. Sapir},
	journal = {Geometric \& Functional Analysis GAFA},
	number = {6},
	pages = {1324--1376},
	publisher = {Springer},
	title = {Groups with small {D}ehn functions and bipartite chord diagrams},
	volume = {16},
	year = {2006}}

@book{OSconj,
	author = {A. Yu. Ol'shanskii and M. V. Sapir},
	publisher = {American Mathematical Soc.},
	title = {The conjugacy problem and {H}igman embeddings},
	volume = {170},
	year = {2004}}

@article{OS01,
	author = {A. Yu. Ol'shanskii and M. V. Sapir},
	journal = {International Journal of Algebra and Computation},
	number = {02},
	pages = {137--170},
	publisher = {World Scientific},
	title = {Length and area functions on groups and quasi-isometric {H}igman embeddings},
	volume = {11},
	year = {2001}}

@article{O18,
	author = {A. Yu. Ol'shanskii},
	journal = {Journal of Combinatorial Algebra},
	number = {4},
	pages = {311--433},
	title = {Polynomially-bounded {D}ehn Functions of Groups},
	volume = {2},
	year = {2018}}

@article{O97,
	author = {A. Yu. Ol'shanskii},
	journal = {Sbornik: Mathematics},
	number = {11},
	publisher = {IOP Publishing},
	title = {On subgroup distortion in finitely presented groups},
	volume = {188},
	year = {1997}}

@book{O,
	author = {A. Yu. Ol'shanskii},
	publisher = {Springer Science \& Business Media},
	title = {Geometry of Defining Relations in Groups},
	volume = {70},
	year = {2012}}

@article{O91,
	author = {A. Yu. Ol'shanskii},
	journal = {International Journal of Algebra \& Computation},
	number = {3},
	title = {Hyperbolicity of Groups with Subquadratic isoperimetric inequality.},
	volume = {1},
	year = {1991}}

@book{Lyndon-Schupp,
	author = {R. C. Lyndon and P. E. Schupp},
	publisher = {Springer},
	title = {Combinatorial Group Theory},
	volume = {188},
	year = {1977}}

@inbook{Gromov,
	address = {New York, NY},
	author = {Gromov, M.},
	booktitle = {Essays in Group Theory},
	doi = {10.1007/978-1-4613-9586-7_3},
	editor = {Gersten, S. M.},
	isbn = {978-1-4613-9586-7},
	pages = {75--263},
	publisher = {Springer New York},
	title = {Hyperbolic Groups},
	url = {https://doi.org/10.1007/978-1-4613-9586-7_3},
	year = {1987}}

@book{Bowditch,
	address = {Singapore},
	author = {B. H. Bowditch},
	isbn = {981-02-0442-6},
	publisher = {World Scientific},
	title = {Notes on {G}romov's Hyperbolicity Criterion for Path-metric Spaces},
	ty = {BOOK},
	url = {http://inis.iaea.org/search/search.aspx?orig_q=RN:24036539},
	year = {1991}}

@article{SBR,
	author = {M. V. Sapir and J.-C. Birget and E. Rips},
	journal = {Annals of Mathematics},
	pages = {345-466},
	title = {Isoperimetric and isodiametric functions of groups},
	url = {https://api.semanticscholar.org/CorpusID:119728458},
	volume = {156},
	year = {1998}}

\end{document}